\documentclass[11pt,letterpaper,reqno]{amsart}

\usepackage{amssymb}
\usepackage{amsmath}
\usepackage{amsthm}
\usepackage{amsfonts}
\usepackage{bbm}
\usepackage{mathrsfs}
\usepackage{centernot}
\usepackage{pgfplots}
\usepackage{esint}

\usepackage{bookmark}
\usepackage{hyperref}
\hypersetup{pdfstartview={FitH}}

\newtheorem{theorem}{Theorem}
\newtheorem{lemma}[theorem]{Lemma}
\newtheorem{proposition}[theorem]{Proposition}
\newtheorem{problem}[theorem]{Problem}

\theoremstyle{definition}
\newtheorem{remark}[theorem]{Remark}

\numberwithin{equation}{section}

\begin{document}

\title[A Szemer\'{e}di-type theorem for subsets of the unit cube]{A Szemer\'{e}di-type theorem for subsets of the unit cube}

\author[Polona Durcik]{Polona Durcik}
\address{Polona Durcik, California Institute of Technology, 1200 E California Blvd, Pasadena, CA 91125, USA}
\email{durcik@caltech.edu}

\author[Vjekoslav Kova\v{c}]{Vjekoslav Kova\v{c}}
\address{Vjekoslav Kova\v{c}, Department of Mathematics, Faculty of Science, University of Zagreb, Bijeni\v{c}ka cesta 30, 10000 Zagreb, Croatia}
\email{vjekovac@math.hr}

%\date{\today}

\subjclass[]{
Primary
05D10, %Combinatorics - Extremal combinatorics - Ramsey theory
42B20; %Harmonic  analysis  on  Euclidean spaces - Harmonic  analysis  in  several  variables - Singular and oscillatory integrals
Secondary
11B30} %Number theory - Sequences and sets - Arithmetic combinatorics; higher degree uniformity

\keywords{Euclidean Ramsey theory, arithmetic progression, density theorem, multilinear estimate, singular integral, oscillatory integral}

\begin{abstract}
We investigate gaps of $n$-term arithmetic progressions $x$, $x+y$, \ldots, $x+(n-1)y$ inside a positive measure subset $A$ of the unit cube $[0,1]^d$. If lengths of their gaps $y$ are evaluated in the $\ell^p$-norm for any $p$ other than $1$, $2$, \ldots, $n-1$, and $\infty$, and if the dimension $d$ is large enough, then we show that the numbers $\|y\|_{\ell^p}$ attain all values from an interval, the length of which depends only on $n$, $p$, $d$, and the measure of $A$. Known counterexamples prevent generalizations of this result to the remaining values of the exponent $p$. We also give an explicit bound for the length of the aforementioned interval. The proof makes the bound depend on the currently available bounds in Szemer\'{e}di's theorem on the integers, which are used as a black box. A key ingredient of the proof are power-type cancellation estimates for operators resembling the multilinear Hilbert transforms. As a byproduct of the approach we obtain a quantitative improvement of the corresponding (previously known) result for side lengths of $n$-dimensional cubes with vertices lying in a positive measure subset of $([0,1]^2)^n$.
\end{abstract}

\maketitle

%%%%%%%%%%%%%%%%%%%%%%%%%%%%%%%%%%%%%%%%%%%%%%%%%%%%%%%%%%%%%%%%%%%%%%%%%%%%%%%%%%%%%%%%%%%%%%%%%%%%%%%%%%%%%%%%%

\section{Introduction}
For a positive integer $n\geq 3$ and a number $0<\delta\leq 1/2$ let $N(n,\delta)$ denote the smallest positive integer $N$ such that each set $S\subseteq\{0,1,2,\ldots,N-1\}$ with at least $\delta N$ elements must contain a nontrivial arithmetic progression of length $n$. Celebrated theorems of Roth \cite{R53:Roth} and Szemer\'{e}di \cite{S69:Szem,S75:Szem} guarantee the existence of the numbers $N(n,\delta)$ for $n=3$ and $n\geq 4$ respectively. Bound of the form
\begin{equation}\label{eq:roth3}
N(3,\delta)\leq\exp(\delta^{-C}),
\end{equation}
for some absolute constant $C$, was first shown by Heath-Brown \cite{HB87:3}, while the analogous result for $4$-term progressions,
\begin{equation}\label{eq:roth4}
N(4,\delta)\leq\exp(\delta^{-C}),
\end{equation}
was established much more recently by Green and Tao \cite{GT17:4}. On the other hand, Gowers \cite{Gow98:4,Gow01:n} proved the bound
\begin{equation}\label{eq:szemeredi}
N(n,\delta)\leq\exp(\exp(\delta^{-C(n)})),
\end{equation}
which is still the best known one when $n\geq 5$.

In the present paper we are interested in density theorems for subsets of the Euclidean space. As a straightforward consequence of the above discrete results one can deduce the following quantitative Szemer\'{e}di-type theorem for subsets of $[0,1]^d$, where $d$ is an arbitrary dimension. We do not regard it as new result of this paper, but rather as a mere continuous-parameter reformulation of Szemer\'{e}di's theorem.

\begin{proposition}\label{prop:contszemeredi}
For integers $n\geq 3$ and $d\geq 1$ there exists a constant $C(n,d)$ such that for any number $0<\delta\leq 1/2$ and any measurable set $A\subseteq [0,1]^d$ with the Lebesgue measure at least $\delta$ one has
\[ \int_{[0,1]^d} \int_{[0,1]^d} \prod_{i=0}^{n-1} \mathbbm{1}_{A}(x+iy) \,\textup{d}y \,\textup{d}x \geq
\begin{cases} \big(\exp(\delta^{-C(n,d)})\big)^{-1} & \text{when } 3\leq n\leq 4, \\ \big(\exp(\exp(\delta^{-C(n,d)}))\big)^{-1} & \text{when } n\geq 5. \end{cases} \]
\end{proposition}

Proof of the above bound simply uses the continuous variant of the trick by Varnavides \cite{Var59:dens}. We repeat that argument in Section~\ref{sec:structuredpart}; see Lemma~\ref{lemma:lowerbound}. A very similar deduction can be found in \cite[Lemma~3.2]{DKR17:corner}.
It is less important for the present paper but also interesting to remark that, conversely, Proposition~\ref{prop:contszemeredi} (even when specialized to any fixed dimension $d\geq 1$) implies Szemer\'{e}di's theorem with bounds of the form \eqref{eq:roth3}--\eqref{eq:szemeredi}; see Remark~\ref{rem:szemerediequiv} in Section~\ref{sec:structuredpart}.
Let us also mention that Candela and Sisask \cite{CS13:torus} formulate Proposition~\ref{prop:contszemeredi} as Szemer\'{e}di's theorem on the torus, while Candela, Szegedy, and Vena \cite{CSV16:abelian} generalize it to the setting of a general compact abelian group. However, those papers do not discuss explicit dependence of the right hand side of the above estimate on the parameter $\delta$. In fact, \cite{CSV16:abelian} shows that the right hand side can be made independent of the dimension $d$, but at the same time it relies on techniques that give at least a tower-type dependence on $\delta$.

One is naturally led to study the set of possible gaps of $n$-term arithmetic progressions in a set $A\subseteq [0,1]^d$, namely
\[ \textup{gaps}_n(A) := \big\{ y\in[-1,1]^d : (\exists x\in[0,1]^d) \big(x, x+y, x+2y,\ldots, x+(n-1)y \in A\big) \big\}. \]
If $A$ has strictly positive measure, then modifying the argument of Stromberg \cite{S72:Steinhaus} one easily sees that $\textup{gaps}_n(A)$ contains a ball around the origin. Indeed, by regularity of the Lebesgue measure, one can assume that $A$ is compact and find an open set $U$ that contains $A$ and has measure at most $(1+1/2n)|A|$. Knowing that the distance between compact set $A$ and disjoint closed set $\mathbb{R}^d\setminus U$ is nonzero, we can define $\varepsilon:=\textup{dist}(A,\mathbb{R}^d\setminus U)/n$. Then for every $y\in[-1,1]^d$ satisfying $\|y\|_{\ell^2}<\varepsilon$ we have that intersection
\[ A\cap(A-y)\cap\cdots\cap(A-(n-1)y) \]
is still contained in $U$ and occupies at least half of $U$, so in particular it is nonempty. Taking any point $x$ from this intersection we arrive at an $n$-term progression with gap $y$ belonging entirely to the set $A$.
It is important to remark that this argument does not give any lower bound on the radius $\varepsilon$ of the ball contained in $\textup{gaps}_n(A)$ that would depend only on the measure of $A$. Such a bound is, in fact, impossible, due to the following variant of a counterexample by Bourgain \cite{B86:roth}. It disproves an even weaker plausible statement and it already applies to the case of three-term progressions.

Let us measure gaps in the $\ell^2$-norm, which is the usual Euclidean norm. In other words, to each $A\subseteq[0,1]^d$ we associate the set
\[ \textup{$\ell^2$-gaps}_n(A) := \big\{ \lambda\in[0,\infty) : \exists (x,y) \big(x, x+y,\ldots, x+(n-1)y \in A \text{ and } \|y\|_{\ell^2}=\lambda\big) \big\}. \]
It is natural to ask if the set $\textup{$\ell^2$-gaps}_n(A)$ contains an interval of length depending only on $n$, $d$, and the measure $|A|$ of a positive-measure set $A\subseteq[0,1]^d$, but the answer is negative.
We can rescale a construction by Bourgain \cite{B86:roth} by a factor $0<\varepsilon\leq 1$ and consider a union of annuli,
\[ A := \big\{ x\in[0,1]^d : (\exists m\in\mathbb{Z}) \big( m-1/10 < \|\varepsilon^{-1}x\|_{\ell^2}^2 < m+1/10 \big) \big\}. \]
By the parallelogram law,
\[ \|x\|_{\ell^2}^2 - 2\|x+y\|_{\ell^2}^2 + \|x+2y\|_{\ell^2}^2 = 2\|y\|_{\ell^2}^2, \]
a gap $y$ of any $3$-term progression $x,x+y,x+2y$ inside $A$ satisfies $m-2/5<2\|\varepsilon^{-1}y\|_{\ell^2}^2<m+2/5$ for some integer $m$, so any interval contained in the set $\textup{$\ell^2$-gaps}_3(A)$ is necessarily shorter than $\varepsilon$. On the other hand, by adding volumes of the annuli comprising $A$, it is easy to see that measure of $A$ is bounded from below uniformly in $\varepsilon$.

A remarkable way out of this apparent dead end was suggested by Cook, Magyar, and Pramanik \cite{CMP15:roth}, who studied $3$-term progressions and started measuring gaps in other $\ell^p$-norms. For any $p\in[1,\infty)$ we can define
\[ \textup{$\ell^p$-gaps}_n(A) := \big\{ \lambda\in[0,\infty) : (\exists x,y) \big(x, x+y,\ldots, x+(n-1)y \in A \text{ and } \|y\|_{\ell^p}=\lambda\big) \big\}. \]
A direct consequence of \cite[Theorem~2.2]{CMP15:roth} is that, for $p\neq 1,2$, sufficiently large $d$, and $\delta>0$, any measurable set $A\subseteq[0,1]^d$ of measure $|A|\geq\delta$ leads to an interval $I\subseteq\textup{$\ell^p$-gaps}_3(A)$ having length depending only on $p$, $d$, and $\delta$. Our goal is to generalize this result to arithmetic progressions of length $n\geq 4$; see Theorem~\ref{thm:maintheorem} below. In order to arrive at the correct formulation, one can first modify Bourgain's construction as it was done by Rimani\'{c} and the present authors \cite[Proposition~4.1]{DKR17:corner}. Namely, for integers $n\geq 3$, $1\leq p\leq n-1$, $d\geq 1$, and for a number $0<\varepsilon\leq 1$ we can define
\[ A := \big\{ x\in[0,1]^d : (\exists m\in\mathbb{Z}) \big( m-2^{-p-2} < \|\varepsilon^{-1}x\|_{\ell^p}^p < m+2^{-p-2} \big) \big\}. \]
Once again, it is easy to verify that measure of $A$ is bounded from below uniformly in $\varepsilon$, but any interval contained in $\textup{$\ell^p$-gaps}_n(A)$ must have length smaller than $\varepsilon$; see \cite{DKR17:corner} for details.

Now we formulate the main result of this paper. It only applies when the ambient space has sufficiently large dimension $d$, in perfect analogy with the results of Cook, Magyar, and Pramanik \cite{CMP15:roth}. For that reason we introduce the following dimensional threshold, depending on $n$ and $p$,
\begin{equation}\label{eq:dimthreshhold}
D(n,p) := 2^{n+3}(n+p).
\end{equation}

\begin{theorem}\label{thm:maintheorem}
For every integer $n\geq 3$, exponent $p\in[1,\infty)\setminus\{1,2,\ldots,n-1\}$, and dimension $d\geq D(n,p)$ there exists a finite positive constant $C(n,p,d)$ with the following property: if $0<\delta\leq 1/2$ and $A\subseteq[0,1]^d$ is a measurable set with the Lebesgue measure at least $\delta$, then the set $\textup{$\ell^p$-gaps}_n(A)$ contains an interval of length at least
\[ \begin{cases} \big(\exp(\exp(\delta^{-C(n,p,d)}))\big)^{-1} & \text{when } 3\leq n\leq 4, \\ \big(\exp(\exp(\exp(\delta^{-C(n,p,d)})))\big)^{-1} & \text{when } n\geq 5. \end{cases} \]
\end{theorem}

In words, if we decide to measure gaps in the $\ell^p$-norm for any $1\leq p<\infty$ other than $1,2,\ldots,n-1$, then we have a concrete lower bound for the length of the longest interval of values attained by the numbers $\|y\|_{\ell^p}$ as the $n$-tuple $x,x+y,\ldots,x+(n-1)y$ ranges over arithmetic progressions in a fixed subset $A$ of the unit cube with measure at least $\delta$. The above estimates reflect the best known bounds in the theorems of Roth and Szemer\'{e}di. Since we will be using \eqref{eq:roth3}--\eqref{eq:szemeredi} in our proof, any possible future improvement of these bounds could lead to a quantitative improvement of Theorem~\ref{thm:maintheorem}. Otherwise, counterexamples mentioned above make Theorem~\ref{thm:maintheorem} a somewhat definite result: it handles all progression lengths $n$ and all possible values of the exponent $p$. On the other hand, the dimensional threshold $D(n,p)$ from \eqref{eq:dimthreshhold} is certainly not optimal, and clarifying in which dimensions $d$ the theorem holds comprises an interesting open problem.

We find Theorem~\ref{thm:maintheorem} a significant generalization of the aforementioned result by Cook, Magyar, and Pramanik \cite{CMP15:roth}. Indeed, the proof in \cite{CMP15:roth} uses estimates for certain multilinear singular integral operators. Their counterparts corresponding to longer progressions lead to integral operators whose boundedness is not yet confirmed. The typical examples are the multilinear Hilbert transforms, continuity of which comprises a major open problem in harmonic analysis; see \cite{T16:mht}. Instead, we will apply techniques similar to the ones used by Thiele and the present authors \cite{DKT16:splx} to establish only nontrivial, but very quantitative cancellation between different scales of these operators. Paper \cite{DKT16:splx} was, in turn, motivated by the problem of quantifying the results of Tao \cite{T16:mht} and Zorin-Kranich \cite{Z17:splx}, while its techniques can be traced back to the work of \v{S}kreb, Thiele, and the present authors \cite{DKST19:NVEA} on an unrelated problem in ergodic theory.

Analogs of Theorem~\ref{thm:maintheorem} for several other instances of patterns to which Bourgain's counterexample applies have already been studied in the previous literature, besides \cite{CMP15:roth}, which covered $3$-term arithmetic progressions. The so-called \emph{corners}, i.e.\@ patterns of the form $(x,y)$, $(x+z,y)$, $(x,y+z)$, were studied by Rimani\'{c} and the present authors \cite{DKR17:corner}. Cartesian products of $3$-term progressions or $3$-element corners were studied by the present authors in \cite{DK18:products}. However, these papers fall short of giving any positive results on $4$-term arithmetic progressions.

There also exists an extensive literature in geometric Ramsey theory on proving the corresponding results for many simpler patterns. We regard those patterns as ``simpler'' only because Bourgain's counterexample (or its modifications, such as the one by Graham \cite{Gra94:countex}) does not apply to them. The existing results measure sizes of those patterns in the Euclidean norm $\|\cdot\|_{\ell^2}$, as this canonical choice is available for them. The first studied pattern was a two-point set, for which the result was established by Furstenberg, Katznelson, and Weiss \cite{FKW90:dist} and also, independently, by Falconer and Marstrand \cite{FM86:dist}. Bourgain \cite{B86:roth} gave yet another proof and a generalization to non-degenerate simplices. For more recent related literature the reader can consult the papers \cite{HLM16:dist,LM19:distgraphs,LM16:prod,LM19:hypergraphs} and references therein. In fact, our formulation of Theorem~\ref{thm:maintheorem} was motivated by wordings of \cite[Theorem~2~(ii)]{LM19:distgraphs}, \cite[Theorem~1.1~(ii)]{LM19:hypergraphs}, and \cite[Theorem~1.2~(ii)]{LM19:hypergraphs} by Lyall and Magyar.

On the other hand, techniques used in this paper also apply to the latter type of patterns, typically giving quantitatively stronger results than those available in the current literature. We establish the following quantitative improvement of one result by Lyall and Magyar \cite[Theorem~1.1~(ii)]{LM19:hypergraphs}; its particular cases or weaker formulations have previously been confirmed by Lyall and Magyar \cite[Theorem~1.2]{LM16:prod} and the present authors \cite[Theorem~1]{DK18:products}.

\begin{theorem}\label{thm:quantifiedlm}
For every positive integer $n$ there exists a finite positive constant $C'(n)$ with the following property: if $0<\delta\leq 1/2$ and $A\subseteq([0,1]^2)^n$ is a measurable set with the Lebesgue measure at least $\delta$, then there exists an interval $I=I(n,A)\subseteq(0,1]$ of length at least
\[ \big(\exp(\delta^{-C'(n)})\big)^{-1} \]
such that for every $\lambda\in I$ one can find
$x_1,\ldots,x_n, y_1,\ldots,y_n \in\mathbb{R}^2$ satisfying
\begin{equation}\label{eq:thquantlm1}
(x_1 + r_1 y_1, x_2 + r_2 y_2, \ldots, x_n + r_n y_n) \in A
\end{equation}
for every choice of $(r_1,\ldots,r_n)\in\{0,1\}^n$ and
\begin{equation}\label{eq:thquantlm2}
\|y_i\|_{\ell^2}=\lambda
\end{equation}
for every index $i\in\{1,\ldots,n\}$.
\end{theorem}

In words, if $A$ is a positive measure subset of the $2n$-dimensional unit cube, then there exists an interval $I$ with length depending only on $n$ and $|A|$ such that for every $\lambda\in I$ the set $A$ contains vertices of an $n$-dimensional cube with side length $\lambda$. Moreover, it is sufficient to consider only cubes that have one-dimensional edges parallel to mutually orthogonal two-dimensional coordinate planes. Everything said also applies to rectangular boxes with other aspect ratios, i.e., which are not necessarily cubes. We wanted to make the formulation of Theorem~\ref{thm:quantifiedlm} as simple as possible, since it primarily serves as another, much simpler illustration of the techniques used to prove Theorem~\ref{thm:maintheorem}.
In \cite{LM19:hypergraphs} Lyall and Magyar remarked that the approach from their paper only ensures that length of the interval $I$ is at least
\[ \big(\exp(\exp(\cdots\exp(C'(n)\delta^{-3\cdot 2^n})\cdots))\big)^{-1}, \]
where the tower of exponentials has length $n$. Theorem~\ref{thm:quantifiedlm} is a quantitative refinement of their result. Comparing the single exponential bound from Theorem~\ref{thm:quantifiedlm} with double or triple exponential bounds from Theorem~\ref{thm:maintheorem}, we support the heuristics that an $n$-dimensional cube is a much simpler pattern than a $3$-term arithmetic progression.

One needs to emphasize that each of the previously cited papers simultaneously also establishes a stronger fact for subsets $A\subseteq\mathbb{R}^d$ of positive upper Banach density: that appropriately measured sizes of copies of the studied pattern within $A$ actually attain an unbounded interval of the form $[\lambda_0,\infty)$. Recall that the \emph{upper Banach density} of a measurable set $A\subseteq\mathbb{R}^d$ is defined to be the number
\[ \overline{\delta}(A) := \limsup_{R\rightarrow \infty} \sup_{x\in\mathbb{R}^d} \frac{|A\cap (x+[0,R]^d)|}{R^d} \in [0,1]. \]
Our proof of Theorem~\ref{thm:quantifiedlm} also shows this stronger property of cubes, see Remark~\ref{rem:reprovinglm} in Section~\ref{sec:quantifiedlm}, but this has already been proven by Lyall and Magyar \cite[Theorem~1.1~(i)]{LM19:hypergraphs}.
At the time of writing we are not able to address this stronger statement for arithmetic progressions of length $n\geq 4$. We use the opportunity to pose this question as an interesting open problem, even though it has already been formulated implicitly in \cite[Section~4]{DKR17:corner}.

\begin{problem}\label{prob:openproblem}
Prove or disprove that for every integer $n\geq 4$ and every exponent $p\in[1,\infty)\setminus\{1,2,\ldots,n-1\}$ there exists a number $D'(n,p)\in(0,\infty)$ such that in every dimension $d\geq D'(n,p)$ the following holds: if $A\subseteq\mathbb{R}^d$ is a measurable set satisfying $\overline{\delta}(A)>0$, then there exists $\lambda_0=\lambda_0(n,p,d,A)\in(0,\infty)$ such that for every $\lambda\in[\lambda_0,\infty)$ one can find $x,y\in\mathbb{R}^d$ satisfying $x, x+y,\ldots, x+(n-1)y \in A$ and $\|y\|_{\ell^p}=\lambda$.
\end{problem}

Difficulty of Problem~\ref{prob:openproblem} seems to be related to the lack of understanding of boundedness properties of the multilinear Hilbert transforms; see the discussion in Section~\ref{sec:errorpart}.

The paper is outlined as follows. Section~\ref{sec:notation} discusses more or less standard notation. Most of the paper, i.e., Sections~\ref{sec:mainproof}--\ref{sec:uniformpart} are dedicated to the proof of Theorem~\ref{thm:maintheorem}. Section~\ref{sec:mainproof} introduces quantities $\mathcal{N}^{\varepsilon}_{\lambda}(A)$ that ``count'' arithmetic progression of $\ell^p$-size $\lambda\in(0,\infty)$ in a set $A$, but using a measure that is smoothened or ``blurred'' up to scale $\varepsilon\in[0,1]$. Quantity $\mathcal{N}^{0}_{\lambda}(A)$ is then decomposed into a structured term $\mathcal{N}^{1}_{\lambda}(A)$, an error term $\mathcal{N}^{\varepsilon}_{\lambda}(A)-\mathcal{N}^{1}_{\lambda}(A)$, and a uniform term $\mathcal{N}^{0}_{\lambda}(A)-\mathcal{N}^{\varepsilon}_{\lambda}(A)$, somewhat in analogy with regularity lemmas in the fields of extremal graph theory and additive combinatorics; see the expository papers by Tao \cite{Tao07:exp} and Gowers \cite{Gow10:exp}. Following that principle, Section~\ref{sec:mainproof} also reduces Theorem~\ref{thm:maintheorem} to three propositions of mutually different nature. Proposition~\ref{prop:structured} handles the structured part, relying on Szemer\'{e}di's theorem and using bounds \eqref{eq:roth3}--\eqref{eq:szemeredi} as black boxes. Proposition~\ref{prop:error} takes care of the error part, proving bounds for a multilinear singular integral. Proposition~\ref{prop:uniform} controls the uniform part by bounding an oscillatory integral. These propositions are then established in Sections~\ref{sec:structuredpart}, \ref{sec:errorpart}, and \ref{sec:uniformpart}, respectively. Proposition~\ref{prop:contszemeredi} from the beginning of this introduction is also shown in Section~\ref{sec:structuredpart}, along the lines of the proof of Proposition~\ref{prop:structured}. Finally, Section~\ref{sec:quantifiedlm} gives a very short and self-sufficient proof of Theorem~\ref{thm:quantifiedlm}, as all of the above steps simplify when an $n$-term progression is replaced with vertices of an $n$-dimensional cube.

Parameters $\lambda$ and $\varepsilon$ of the counting quantity $\mathcal{N}^{\varepsilon}_{\lambda}(A)$ can be interpreted as the \emph{scale of largeness} and the \emph{scale of smoothness}, respectively. For this reason one might call the presented method of proof the \emph{largeness--smoothness multiscale approach}. That method draws inspiration from a paper by Bourgain \cite{B86:roth}, but, as we understand it, it seems to have been first used by Cook, Magyar, and Pramanik \cite{CMP15:roth}. We roughly follow the same approach, but we also prefer to make several changes to it. For instance, we use both dilations and convolutions in the definition of $\mathcal{N}^{\varepsilon}_{\lambda}(A)$, so that $\lambda$ and $\varepsilon$ can really be interpreted and treated as ``scales'' in the usual sense. We believe that this approach will also prove useful in relation to other problems in the Euclidean Ramsey theory.

As it has already been mentioned, there is an apparent similarity of our scheme of proof with the outlines of the aforementioned papers \cite{CMP15:roth,DK18:products,DKR17:corner}. However, the difficulties within the field of harmonic analysis that are related to the study of longer arithmetic progressions have not been dealt with elsewhere. Moreover, many details of the approach will be worked out differently here, enabling us to be more quantitative and also to further streamline the corresponding parts of the proof. In comparison with \cite{CMP15:roth,DK18:products,DKR17:corner} the main new ingredient of the present paper is Section~\ref{sec:errorpart}, where a multilinear singular estimate is proven. It is largely influenced by \cite{DKT16:splx}, but we made sure that this section is entirely self-contained. The whole paper does not depend on any difficult outside results, except for the bounds \eqref{eq:roth3}--\eqref{eq:szemeredi} in Szemer\'{e}di's theorem used in the proof of Theorem~\ref{thm:maintheorem}, while the proof of Theorem~\ref{thm:quantifiedlm} does not even depend on any combinatorial results.

\section{Notation}
\label{sec:notation}
Whenever $A$ and $B$ are two nonnegative expressions, we will write $A\lesssim B$ and $B\gtrsim A$ if there exists a constant $C\in[0,\infty)$ such that  $A\leq C B$. We also write $A\sim B$ if both $A\lesssim B$ and $B\lesssim A$ hold. If the implied constant $C$ depends on a set of parameters $P$, we will denote that in a subscript, writing $A\lesssim_{P} B$ and $B\gtrsim_{P} A$. The set of parameters $P$ is allowed to be empty, in which case the constant $C$ is an absolute one, or all dependencies are understood.

The \emph{standard inner product} on $\mathbb{R}^d$ will be written with a dot, i.e.
\[ x\cdot y:=\sum_{i=1}^{d}x_i y_i \]
for vectors $x=(x_1,\ldots,x_d)$ and $y=(y_1,\ldots,y_d)$ in $\mathbb{R}^d$. For any exponent $p\in[1,\infty)$ the \emph{$\ell^p$-norm} of $x$ is given by the formula
\[ \|x\|_{\ell^p} := \Big( \sum_{i=1}^{d} |x_i|^p \Big)^{1/p}. \]

The \emph{indicator function} of a set $A$ will be denoted $\mathbbm{1}_A$, while the ambient space (i.e.\@ the domain of $\mathbbm{1}_A$) will always be understood from context.
The \emph{support} of a continuous function $f\colon\mathbb{R}^d\to\mathbb{C}$ will be written as $\mathop{\textup{supp}}f$.
The greatest integer not exceeding a number $x\in\mathbb{R}$, i.e.\@ \emph{floor} of $x$, will be written $\lfloor x\rfloor$.
The imaginary unit in $\mathbb{C}$ will be denoted $\mathbbm{i}$.
The logarithm function will always have the number $e$ as its base.
%We put a dot in place of a certain variable to denote a single-variable function obtained by freezing all other variables in the expression. For example, if $(x,y)\mapsto f(x,y)$ is a function of two variables, then $f(x,\cdot)$ denotes a function of one variable $y\mapsto f(x,y)$ for each fixed $x$.

The Lebesgue measure of a measurable set $A\subseteq\mathbb{R}^d$ will be written simply as $|A|$, the dimension $d$ being clear from context. Whenever the measure is not specified, it is understood that integrals are evaluated with respect to the Lebesgue measure.
Let us write $\fint_{T}$ for the integral \emph{average} over a measurable set $T$, i.e.\@ for $\frac{1}{|T|}\int_T$.

If $f$ is a complex integrable function on $\mathbb{R}^d$ and $\sigma$ is a finite measure on Borel subsets of $\mathbb{R}^d$, then we write $f_\lambda$ and $\sigma_\lambda$ for their ($\textup{L}^1$-normalized) \emph{dilates} defined as
\[ f_{\lambda}(x) := \lambda^{-d} f(\lambda^{-1}x) \]
for a number $\lambda>0$ and for each $x\in\mathbb{R}^d$, and
\[ \sigma_{\lambda}(A) := \sigma(\lambda^{-1}A) \]
for Borel sets $A\subseteq\mathbb{R}^d$.
Otherwise, lower indices that we use have no predefined meaning; the context should remove any ambiguity.
We also often use upper indices when there is no chance for confusing them with powers.

\emph{Convolution} of two integrable functions $f,g\colon\mathbb{R}^d\to\mathbb{C}$ is a function $f\ast g$ defined as
\[ (f\ast g)(x) := \int_{\mathbb{R}^d} f(x-y) g(y) \,\textup{d}y = \int_{\mathbb{R}^d} g(x-y) f(y) \,\textup{d}y \]
for almost every $x\in\mathbb{R}^d$.
If we are also given a finite Borel measure $\sigma$, then the \emph{convolution} $\sigma\ast g$
is again an almost everywhere defined function on $\mathbb{R}^d$, this time given as
\[ (\sigma\ast g)(x) := \int_{\mathbb{R}^d} g(x-y) \,\textup{d}\sigma(y). \]

The \emph{Fourier transform} of an integrable function $f\colon\mathbb{R}^d\to\mathbb{C}$ is written as $\widehat{f}$ and normalized as
\[ \widehat{f}(\xi) := \int_{\mathbb{R}^d} f(x) e^{-2\pi \mathbbm{i} x\cdot\xi} \,\textup{d}x \]
for $\xi\in\mathbb{R}^d$.
Well-known symmetries of the Fourier transform are
\[ (\widehat{f\ast g})(\xi) = \widehat{f}(\xi) \widehat{g}(\xi), \]
\[ \widehat{f_\lambda}(\xi) = \widehat{f}(\lambda \xi), \]
and
\[ \widehat{f(\cdot+y)}(\xi) = e^{2\pi\mathbbm{i}y\cdot\xi}\widehat{f}(\xi) \]
for integrable functions $f$ and $g$, $\lambda>0$, and $y\in\mathbb{R}^d$.
The \emph{Fourier transform} of a finite Borel measure $\sigma$ is defined as
\[ \widehat{\sigma}(\xi) := \int_{\mathbb{R}^d} e^{-2\pi \mathbbm{i} x\cdot\xi} \,\textup{d}\sigma(x) \]
and the property
\[ (\widehat{\sigma\ast g})(\xi) = \widehat{\sigma}(\xi) \widehat{g}(\xi) \]
remains valid for the convolution of $\sigma$ with an integrable function $g$.

\section{Quantities that detect progressions}
\label{sec:mainproof}
Progression length $n\geq 3$ will be fixed throughout the paper and so will the exponent $p\in[1,\infty)$ different from $1$, $2$, \ldots, $n-1$.
We will be working in $\mathbb{R}^d$ for a suitably large dimension $d$. Requirement $d\geq D(n,p)$, where $D(n,p)$ was given in \eqref{eq:dimthreshhold}, will be used only when needed in the proof.
Objects we are about to introduce all depend on $n$, $p$, and $d$, but, in favor of elegance, we do not always emphasize that in the notation.

Let us first recall the measure $\sigma$ introduced in \cite{CMP15:roth}. Slightly informally, $\sigma$ is given as a delta-measure $\delta(\|x\|_{\ell^p}^p-1)$, so it is a measure on the Borel subsets of $\mathbb{R}^d$ supported on the unit sphere in the $\ell^p$-norm,
\[ S:= \big\{(x_1,x_2,\ldots,x_d)\in\mathbb{R}^d : |x_1|^p+|x_2|^p+\cdots+|x_d|^p=1\big\}. \]
It can also be defined, more rigorously, as an appropriately weighted surface measure of $S$; see \cite[\S{}2.1.1]{CMP15:roth}.
Fix a $\textup{C}^\infty$ function $\psi\colon\mathbb{R}\to[0,\infty)$ supported in $[-1,1]$ and having integral $1$.
The only fact about $\sigma$ that we will need is that it is the vague limit as $\eta\to0^+$ of the measures $\sigma^\eta$ given by
\begin{equation}\label{eq:measuresigmaeta}
\textup{d}\sigma^{\eta}(x)=\psi_{\eta}(\|x\|_{\ell^p}^p-1)\,\textup{d}x.
\end{equation}
This was already commented on and used in \cite{CMP15:roth}.
In other words, for any continuous function $f\colon\mathbb{R}^d\to\mathbb{C}$ we have
\begin{equation}\label{eq:vagueconvergence}
\lim_{\eta\to0^+} \int_{\mathbb{R}^d} f(x) \,\textup{d}\sigma^{\eta}(x) = \int_{\mathbb{R}^d} f(x) \,\textup{d}\sigma(x).
\end{equation}

The most important object associated with a measurable set $A\subseteq[0,1]^d$ and a number $\lambda\in(0,\infty)$ is the following quantity that ``counts'' $n$-term progressions in $A$ with gaps $y$ of length precisely $\|y\|_{\ell^p}=\lambda$. It is defined as
\begin{equation}\label{eq:n0definition}
\mathcal{N}^{0}_{\lambda}(A) := \int_{\mathbb{R}^d} \int_{\mathbb{R}^d} \prod_{i=0}^{n-1} \mathbbm{1}_{A}(x+i y) \,\textup{d}\sigma_{\lambda}(y) \,\textup{d}x;
\end{equation}
number $0$ in the upper index will be justified in a moment.
In analogy with the paper by Cook, Magyar, and Pramanik \cite{CMP15:roth}, we will also need to introduce a certain ``smoothened'' variant of \eqref{eq:n0definition}. Let us fix an even $\textup{C}^\infty$ function $\varphi\colon\mathbb{R}^d\to[0,\infty)$, whose integral is equal to $1$, and which is strictly positive on $[-2,2]^d$ and zero outside $[-3,3]^d$. Dependencies of any future constants and other objects on $\varphi$ and $\psi$ will be notationally suppressed. Now for each $\varepsilon\in(0,1]$ define
\begin{equation}\label{eq:nepsdefinition}
\mathcal{N}^{\varepsilon}_{\lambda}(A) := \int_{\mathbb{R}^d} \int_{\mathbb{R}^d} \Big( \prod_{i=0}^{n-1} \mathbbm{1}_{A}(x+i y) \Big) (\sigma_{\lambda}\ast\varphi_{\varepsilon\lambda})(y) \,\textup{d}y \,\textup{d}x.
\end{equation}
For fixed $A$ and $\lambda$ we have
\begin{equation}\label{eq:nconvergence}
\lim_{\varepsilon\to0^+} \mathcal{N}^{\varepsilon}_{\lambda}(A) = \mathcal{N}^{0}_{\lambda}(A).
\end{equation}
Indeed, denote
\begin{equation}\label{eq:Fdefinedasaproduct}
\mathcal{F}(x,y):=\prod_{i=0}^{n-1}\mathbbm{1}_{A}(x+i y)
\end{equation}
and observe that the integral
\begin{equation}\label{eq:Fdefinedasaproductnewq}
f(y) := \int_{[0,1]^d} \mathcal{F}(x,y) \,\textup{d}x
\end{equation}
defines a continuous function $f\colon\mathbb{R}^d\to[0,1]$. This is an easy consequence of continuity of translation operators on $\textup{L}^1(\mathbb{R}^d)$.
Using Fubini's theorem we can rewrite
\[ \mathcal{N}^{\varepsilon}_{\lambda}(A) = \int_{\mathbb{R}^d} \big(f\ast\varphi_{\varepsilon\lambda}\big)(y) \,\textup{d}\sigma_{\lambda}(y). \]
Since $(f\ast\varphi_{\varepsilon\lambda})(y)$ converges to $f(y)$ as $\varepsilon\to0^+$ for every point $y\in\mathbb{R}^d$, property \eqref{eq:nconvergence} is confirmed by applying the dominated convergence theorem with respect to the finite measure $\sigma_{\lambda}$.

Number $\lambda$ is the scale of ``largeness;'' it simply restricts our attention to arithmetic progressions with gaps of size $\lambda$. On the other hand, we interpret $\varepsilon$ as the scale of ``smoothness,'' the main intuition being that \eqref{eq:nepsdefinition} range from $\mathcal{N}^{1}_{\lambda}(A)$, which is easier to estimate from below, to $\mathcal{N}^{0}_{\lambda}(A)$, which is the actual counting expression that we care about.
In line with that reasoning we will fix a sufficiently small $\varepsilon$ and split
\begin{equation}\label{eq:splitting}
\mathcal{N}^{0}_{\lambda}(A)
= \mathcal{N}^{1}_{\lambda}(A)
+ \big(\mathcal{N}^{\varepsilon}_{\lambda}(A)-\mathcal{N}^{1}_{\lambda}(A)\big)
+ \big(\mathcal{N}^{0}_{\lambda}(A)-\mathcal{N}^{\varepsilon}_{\lambda}(A)\big).
\end{equation}
We will put effort into estimating each of the three terms on the right hand side of \eqref{eq:splitting}. This will be done in the following three propositions.

\begin{proposition}\label{prop:structured}
There exists a constant $E=E(n,p,d)\in[1,\infty)$ with the following property: if $\lambda\in(0,1]$, $\delta\in(0,1/2]$, and $A\subseteq[0,1]^d$ is a measurable set with $|A|\geq\delta$, then
\[ \mathcal{N}^{1}_{\lambda}(A) \geq \begin{cases} \big(\exp(\delta^{-E})\big)^{-1} & \text{when } 3\leq n\leq 4, \\ \big(\exp(\exp(\delta^{-E}))\big)^{-1} & \text{when } n\geq 5. \end{cases} \]
\end{proposition}

\begin{proposition}\label{prop:error}
There exists a constant $F=F(n,p,d)\in[1,\infty)$ with the following property: for a positive integer $J$, any real numbers $\lambda_j\in(2^{-j},2^{-j+1}]$; $j=1,2,\ldots,J$, any $\varepsilon\in(0,1/2]$, and any measurable set $A\subseteq[0,1]^d$ we have
\[ \sum_{j=1}^{J} \big|\mathcal{N}^{\varepsilon}_{\lambda_j}(A)-\mathcal{N}^{1}_{\lambda_j}(A)\big| \leq \varepsilon^{-F} J^{1-2^{-n+2}}. \]
\end{proposition}

\begin{proposition}\label{prop:uniform}
For every integer $d\geq D(n,p)$, every $\lambda,\varepsilon\in(0,1]$, and every measurable set $A\subseteq[0,1]^d$ we have
\begin{equation}\label{eq:propuniformineq}
\big|\mathcal{N}^{0}_{\lambda}(A)-\mathcal{N}^{\varepsilon}_{\lambda}(A)\big| \lesssim_{n,p,d} \varepsilon^{1/3}.
\end{equation}
\end{proposition}

Informally speaking, $\mathcal{N}^{1}_{\lambda}(A)$ constitutes the \emph{structured part} of decomposition \eqref{eq:splitting}. It will be controlled using the known bounds \eqref{eq:roth3}--\eqref{eq:szemeredi} in Szemer\'{e}di's theorem. This is ``the main term'' in the language of \cite{CMP15:roth}. The next summand, $\mathcal{N}^{\varepsilon}_{\lambda}(A)-\mathcal{N}^{1}_{\lambda}(A)$, is an analogue of what is traditionally called the \emph{error part}; see \cite{Gow10:exp,Tao07:exp}. Starting with Bourgain \cite{B86:roth} it has often been handled by certain energy pigeonholing. The last summand $\mathcal{N}^{0}_{\lambda}(A)-\mathcal{N}^{\varepsilon}_{\lambda}(A)$ we call the \emph{uniform part}. While bounding it, one can shift attention completely from the set $A$ to the effectiveness of approximation of $\sigma$ by $\sigma\ast\varphi_{\varepsilon}$ in (the Euclidean analogue of) the Gowers uniformity norm.

Propositions~\ref{prop:structured}--\ref{prop:uniform} will be established in the three subsequent sections, while here we show how together they imply Theorem~\ref{thm:maintheorem}. This argument is far from new, appearing to some extent already in \cite{B86:roth} and written in detail in \cite{CMP15:roth}. We include it for completeness.

\begin{proof}[Proof of Theorem~\ref{thm:maintheorem}]
Take an integer $d\geq D(n,p)$ and recall the constants $E$ and $F$ from Propositions~\ref{prop:structured} and \ref{prop:error}, respectively. Let $G=G(n,p,d)$ be the constant that was implicit in inequality \eqref{eq:propuniformineq}. Take $\delta\in(0,1/2]$, denote
\[ \vartheta := \begin{cases} \big(\exp(\delta^{-E})\big)^{-1} & \text{when } 3\leq n\leq 4, \\ \big(\exp(\exp(\delta^{-E}))\big)^{-1} & \text{when } n\geq 5, \end{cases} \]
choose
\[ \varepsilon := \Big(\frac{\vartheta}{3G}\Big)^{3}, \]
and finally take
\[ J := \big\lfloor \big(3\vartheta^{-1}\varepsilon^{-F}\big)^{2^{n-2}} \big\rfloor + 1. \]
Observe
\[ J \lesssim_{n,p,d} \big(\vartheta^{-1}\big)^{2^{n-2}+3F}, \]
so that
\begin{equation}\label{eq:proofmainthm0}
2^{-J} \geq \begin{cases} \big(\exp(\exp(\delta^{-C(n,p,d)}))\big)^{-1} & \text{when } 3\leq n\leq 4, \\ \big(\exp(\exp(\exp(\delta^{-C(n,p,d)})))\big)^{-1} & \text{when } n\geq 5 \end{cases}
\end{equation}
for a sufficiently large constant $C(n,p,d)\in[1,\infty)$.

Take a measurable set $A\subseteq[0,1]^d$ such that $|A|\geq \delta$. First, we claim that there exists an index $j\in\{1,2,\ldots,J\}$ such that for every $\lambda\in(2^{-j},2^{-j+1}]$ we have
\begin{equation}\label{eq:proofmainthm1}
\big|\mathcal{N}^{\varepsilon}_{\lambda}(A)-\mathcal{N}^{1}_{\lambda}(A)\big| \leq \varepsilon^{-F} J^{-2^{-n+2}}.
\end{equation}
If that was not the case, then for each $j$ we could choose $\lambda_j\in(2^{-j},2^{-j+1}]$ such that the opposite inequality of \eqref{eq:proofmainthm1} holds. Summing in $j$ would then contradict Proposition~\ref{prop:error}.

Next, fix an index $j$ with the above property. Using splitting \eqref{eq:splitting}, Proposition~\ref{prop:structured}, Equation \eqref{eq:proofmainthm1}, Proposition~\ref{prop:uniform}, and the choices of $\varepsilon$ and $J$, for any $\lambda\in(2^{-j},2^{-j+1}]$ we can estimate
\begin{align}
\mathcal{N}^{0}_{\lambda}(A)
& \geq \mathcal{N}^{1}_{\lambda}(A)
- \big|\mathcal{N}^{\varepsilon}_{\lambda}(A)-\mathcal{N}^{1}_{\lambda}(A)\big|
- \big|\mathcal{N}^{0}_{\lambda}(A)-\mathcal{N}^{\varepsilon}_{\lambda}(A)\big| \nonumber \\
& \geq \vartheta - \varepsilon^{-F} J^{-2^{-n+2}} - G \varepsilon^{1/3} \nonumber \\
& \geq \vartheta - \vartheta/3 - \vartheta/3 = \vartheta/3. \label{eq:proofmainthm2}
\end{align}
In particular, the interval $I:=(2^{-j},2^{-j+1}]$ has length at least \eqref{eq:proofmainthm0} and for each $\lambda\in I$ we have $\mathcal{N}^{0}_{\lambda}(A)>0$, which implies that there exists a progression $x,x+y,\ldots,x+(n-1)y$ in $A$ satisfying $\|y\|_{\ell^p}=\lambda$.
\end{proof}

Observe that \eqref{eq:proofmainthm2} above gives slightly more than claimed in the formulation of Theorem~\ref{thm:maintheorem}: one has
\[ \mathcal{N}^{0}_{\lambda}(A) \gtrsim \begin{cases} \big(\exp(\delta^{-E})\big)^{-1} & \text{when } 3\leq n\leq 4, \\ \big(\exp(\exp(\delta^{-E}))\big)^{-1} & \text{when } n\geq 5 \end{cases} \]
for each scale $\lambda$ from an interval $I$ of length at least \eqref{eq:proofmainthm0}. In other words, for each $\lambda\in I$ there are ``quite a few'' progressions in $A$ with gaps of that $\ell^p$-size.

\section{The structured part: proof of Propositions~\ref{prop:contszemeredi} and \ref{prop:structured}}
\label{sec:structuredpart}
We will simultaneously prove Proposition~\ref{prop:contszemeredi}, which was mentioned in the introduction, and Proposition~\ref{prop:structured}, which is the first ingredient needed in the proof of Theorem~\ref{thm:maintheorem}. Both will be consequences of the following, slightly more general result. Recall the definition of the numbers $N(n,\delta)$ stated at the very beginning of the introductory section.

\begin{lemma}\label{lemma:lowerbound}
For $\lambda\in(0,1]$, $\delta\in(0,1/2]$, and a measurable set $A\subseteq[0,1]^d$ satisfying $|A|\geq\delta$ we have
\begin{equation}\label{eq:formlowerbound}
\frac{1}{\lambda^d} \int_{[0,\lambda]^d} \int_{[0,1]^d} \prod_{i=0}^{n-1} \mathbbm{1}_{A}(x+i y) \,\textup{d}x \,\textup{d}y
\gtrsim_d \delta^{d+1} N(n,\delta/4)^{-d-2}.
\end{equation}
\end{lemma}

\begin{proof}
Denote $N:=N(n,\delta/4)$, $\vartheta:=\delta/4d$, and $T := (\vartheta\lambda/2N,\vartheta\lambda/N]^d \times [0,1-\vartheta]^d$.
The basic idea is to apply discrete Szemer\'{e}di's theorem to $N$-point sets $\{u+kt:k=0,1,\ldots,N-1\}$ obtained by fixing $(t,u)\in T$ and then integrate over sufficiently many choices of $(t,u)$. This is a continuous-parameter modification of a trick by Varnavides \cite{Var59:dens}.

Assuming that $A\subseteq[0,1]^d$ has measure $|A|\geq\delta$, we can write
\begin{align*}
& \fint_{T} \Big( \frac{1}{N}\sum_{k=0}^{N-1} \mathbbm{1}_{A}(u+kt) \Big) \,\textup{d}t \,\textup{d}u
= \frac{1}{N}\sum_{k=0}^{N-1} \fint_{(\vartheta\lambda/2N,\vartheta\lambda/N]^d} \fint_{[0,1-\vartheta]^d+kt} \mathbbm{1}_{A}(u) \,\textup{d}u \,\textup{d}t \\
& \geq \frac{|A\cap[\vartheta,1-\vartheta]^d|}{(1-\vartheta)^d}
\geq |A| - \big| [0,1]^d\setminus[\vartheta,1-\vartheta]^d \big| \geq \delta - 2d\vartheta = \frac{\delta}{2}.
\end{align*}
Thus, if we denote
\[ T_{\textup{large}} := \Big\{ (t,u)\in T : \frac{1}{N}\sum_{k=0}^{N-1} \mathbbm{1}_{A}(u+kt) \geq \frac{\delta}{4} \Big\}, \]
then
\begin{align*}
\frac{\delta}{2} & \leq \frac{1}{|T|} \int_{T} \Big( \frac{1}{N}\sum_{k=0}^{N-1} \mathbbm{1}_{A}(u+kt) \Big) \,\textup{d}t \,\textup{d}u \\
& \leq \frac{1}{|T|} \int_{T_{\textup{large}}} 1 \,\textup{d}t \,\textup{d}u + \frac{1}{|T|} \int_{T\setminus T_{\textup{large}}} \frac{\delta}{4} \,\textup{d}t \,\textup{d}u
\leq \frac{|T_{\textup{large}}|}{|T|} + \frac{\delta}{4}
\end{align*}
gives
\begin{equation}\label{eq:tlowerbound}
|T_{\textup{large}}| \geq \frac{\delta}{4} |T|
\geq \frac{\delta}{4} \Big(\frac{\vartheta\lambda}{2N}\Big)^d (1-\vartheta)^d
\geq \frac{\delta}{4} \Big(\frac{\vartheta\lambda}{2N}\Big)^d (1-d\vartheta)
\gtrsim_d \delta^{d+1} N^{-d} \lambda^{d}.
\end{equation}

For each pair $(t,u)\in T_{\textup{large}}$ we consider the set
\[ S_{t,u} := \big\{ k\in\{0,1,2,\ldots,N-1\} : u+k t\in A \big\} \]
and from the mere definition of $T_{\textup{large}}$ we know that it has at least $(\delta/4)N$ elements. Using the fact that we have chosen $N$ to be $N(n,\delta/4)$, we know that $S_{t,u}$ contains a nontrivial $n$-term arithmetic progression, i.e.\@ there exist integers $0\leq k\leq N-1$ and $1\leq l\leq\big\lfloor\frac{N-1-k}{n-1}\big\rfloor$ such that
\[ k, k+l, k+2l,\ldots, k+(n-1)l \in S_{t,u}, \]
i.e.
\[ \prod_{i=0}^{n-1} \mathbbm{1}_{A}\big(u+(k+il)t\big) = 1. \]
Consequently,
\[ |T_{\textup{large}}| \leq \int_{T_{\textup{large}}} \sum_{k=0}^{N-1} \sum_{l=1}^{\lfloor\frac{N-1-k}{n-1}\rfloor} \prod_{i=0}^{n-1} \mathbbm{1}_{A}\big(u+(k+il)t\big) \,\textup{d}t \,\textup{d}u, \]
so, interchanging the integral with the two sums and substituting $x=u+kt$, $y=lt$, we obtain
\[ |T_{\textup{large}}| \leq \sum_{k=0}^{N-1} \sum_{l=1}^{\lfloor\frac{N-1-k}{n-1}\rfloor} \frac{1}{l^d} \int_{(0,\vartheta\lambda]^d} \int_{[0,1]^d} \prod_{i=0}^{n-1} \mathbbm{1}_{A}(x+i y) \,\textup{d}x \,\textup{d}y. \]
Combining the last inequality with \eqref{eq:tlowerbound}, estimating crudely the sums in $k$ and $l$ by $N^2$ times the largest term, and finally using $(0,\vartheta\lambda]\subset[0,\lambda]$, we deduce \eqref{eq:formlowerbound}.
\end{proof}

Specializing $\lambda=1$ in Lemma~\ref{lemma:lowerbound} and using Estimates \eqref{eq:roth3}--\eqref{eq:szemeredi} proves Proposition~\ref{prop:contszemeredi}. Note that both the factor $\delta^{d+1}$ and the implicit constant depending on $d$ are swiped into the double or triple exponential lower bound for $N(n,\delta)^{-d-2}$.

For the proof of Proposition~\ref{prop:structured} we observe that, by the choice of $\varphi$, for each $y\in[0,1]^d$ we have
\[ (\sigma\ast\varphi)(y) = \int_S \varphi(y-z) \,\textup{d}\sigma(z) \geq \sigma(S)\min_{x\in[-2,2]^d} \varphi(x) >0, \]
so
\[ \sigma_{\lambda}\ast\varphi_{\lambda} \gtrsim_{p,d} \lambda^{-d} \mathbbm{1}_{[0,\lambda]^d} \]
and it remains to apply Lemma~\ref{lemma:lowerbound} again.

\begin{remark}\label{rem:szemerediequiv}
Interestingly, Proposition~\ref{prop:contszemeredi} is actually equivalent to having bounds of the form \eqref{eq:roth3}--\eqref{eq:szemeredi} in Szemer\'{e}di's theorem. Let us fix integers $n\geq 3$ and $d\geq 1$ and assume that the estimate from Proposition~\ref{prop:contszemeredi} holds for any $\delta\in(0,1/2]$ and any measurable set $A\subseteq[0,1]^d$ with measure $|A|\geq\delta$.
Fix $\delta\in(0,1/2]$, take a positive integer $N$ such that
\[ N \geq \begin{cases} \exp((\delta/5)^{-C(n,d)}) & \text{when } 3\leq n\leq 4, \\ \exp(\exp((\delta/5)^{-C(n,d)})) & \text{when } n\geq 5, \end{cases} \]
and finally take a set $S\subseteq\{0,1,\ldots,N-1\}$ with cardinality at least $\delta N$.
We claim that $S$ contains a nontrivial progression of length $n$. We want to argue by contradiction, so suppose that this is not the case. For any $k\in S$ consider a ``thin'' rectangular box
\[ T_k := [(5k+2)/5N,(5k+3)/5N] \times [0,1]^{d-1} \]
and define
\[ A := \bigcup_{k\in S} T_k \subseteq [0,1]^d. \]
If $x_1,x_2,\ldots,x_n$ is an arithmetic progression in $A$ such that $x_i\in T_{k_i}$ for $i=1,2,\ldots,n$, then it is easy to see that the corresponding indices of tubes, $k_1,k_2,\ldots,k_n\in S$, must also stand in arithmetic progression. By our assumption that $S$ contains no nontrivial arithmetic progressions of length $n$, this is only possible when $k_1=k_2=\cdots=k_n$. In particular, $x_1$ and $x_2$ must belong to the same tube $T_{k_1}$. Consequently, we can estimate
\begin{align*}
\int_{([0,1]^d)^2} \prod_{i=0}^{n-1} \mathbbm{1}_{A}(x+iy) \,\textup{d}x \,\textup{d}y
& \leq \sum_{k\in S} \int_{([0,1]^d)^2} \mathbbm{1}_{T_k}(x) \mathbbm{1}_{T_k}(x+y) \,\textup{d}x \,\textup{d}y \\
& \leq \sum_{k\in S} \frac{1}{25N^2} < \frac{1}{N}.
\end{align*}
On the other hand, the set $A$ has measure at least $\delta/5$ and Proposition~\ref{prop:contszemeredi} applies, contradicting the last displayed estimate and our choice of $N$.
\end{remark}

\section{The error part: proof of Proposition~\ref{prop:error}}
\label{sec:errorpart}
In this section we need to control the difference between the ``slightly smoothened'' counting quantity $\mathcal{N}^{\varepsilon}_{\lambda}(A)$ and the ``totally smoothened'' counting quantity $\mathcal{N}^{1}_{\lambda}(A)$. We need to say more about this difference than just bounding it for each fixed scale $\lambda$. Proposition~\ref{prop:error} attempts to control the multiscale sum
\begin{equation}\label{eq:multiscaleHilbert1}
\sum_{j=1}^{J} \kappa_j \big(\mathcal{N}^{\varepsilon}_{\lambda_j}(A)-\mathcal{N}^{1}_{\lambda_j}(A)\big)
\end{equation}
for arbitrary scales $\lambda_j\in(2^{-j},2^{-j+1}]$ and arbitrary complex signs $\kappa_j$, with a bound that is better than the trivial one, i.e.\@ grows sub-linearly in the number of scales $J$. To achieve this we will need to exploit some cancellation between different summands of \eqref{eq:multiscaleHilbert1}. In order to recognize the level of complexity of this problem, let us write \eqref{eq:multiscaleHilbert1} in the form
\[ \int_{\mathbb{R}^d} \int_{\mathbb{R}^d} \prod_{i=0}^{n-1} \mathbbm{1}_{A}(x+i y) K(y) \,\textup{d}y \,\textup{d}x, \]
where
\[ K(y) := \sum_{j=1}^{J} \kappa_j \big( (\sigma_{\lambda_j}\ast\varphi_{\varepsilon\lambda_j})(y) - (\sigma_{\lambda_j}\ast\varphi_{\lambda_j})(y) \big) \]
is viewed as an integral kernel. It can be shown that $K$ is indeed a translation-invariant Cal\-de\-r\'{o}n--Zyg\-mund kernel (see the definition in \cite[Chapter~VII]{St93:book}), but we will not need that fact, since we will take a slightly different route anyway.
Various multilinear generalizations of Cal\-de\-r\'{o}n--Zyg\-mund operators are still far from being categorized and understood.
If we turn our attention to dimension $d=1$, replace $K(y)$ with $\mathbbm{1}_{[-R,-r]\cup[r,R]}(y)/y$, and replace the $n$ appearances of $\mathbbm{1}_{A}$ with generic functions $f_0,f_1,\ldots,f_{n-1}$, we arrive at
\begin{equation}\label{eq:multiscaleHilbert2}
\int_{\mathbb{R}} \int_{[-R,-r]\cup[r,R]} \prod_{i=0}^{n-1} f_{i}(x+i y) \frac{\textup{d}y}{y} \,\textup{d}x,
\end{equation}
which is the (dualized and truncated) \emph{multilinear Hilbert transform}. It is useful to think of $J$ as being roughly $\log(R/r)$. First $\textup{L}^p$-bounds for \eqref{eq:multiscaleHilbert2} that are uniform in $r$, $R$, and the functions $f_i$ were shown by Lacey and Thiele \cite{LT97,LT99} in the case $n=3$, while it is only conjectured that any such bounds hold when $n\geq 4$. On the other hand, any bound for \eqref{eq:multiscaleHilbert2} that grows slower than $\log(R/r)$ as $r\to0^+$ and $R\to\infty$ is a nontrivial result. That type of bound was first established by Tao \cite{T16:mht}, while a quantitative improvement to a bound of the form $(\log(R/r))^{1-\gamma}$ for some $\gamma>0$ was later given by Thiele and the authors in \cite{DKT16:splx}.
Consequently, a large part of this section draws ideas from the paper \cite{DKT16:splx}. Notable differences are that we need more general kernels, we are working in higher dimensions $d$, and we are dealing with arithmetic progressions, rather than with (higher-dimensional) corners.

For the purpose of bounding \eqref{eq:multiscaleHilbert1}, we will find useful to decompose it into sums involving Gaussian functions. Let $\mathbbm{g}\colon\mathbb{R}^d\to\mathbb{R}$ be the \emph{standard Gaussian} on $\mathbb{R}^d$, defined as
\begin{equation}\label{eq:standardgaussian}
\mathbbm{g}(x) := e^{-\pi\|x\|_{\ell^2}^2}
\end{equation}
for each $x\in\mathbb{R}^d$, or, in Cartesian coordinates,
\[ \mathbbm{g}(x_1,x_2,\ldots,x_d) = e^{-\pi(x_1^2+x_2^2+\cdots+x_d^2)}. \]
Let us introduce the following notation for its partial derivatives,
\[ \mathbbm{h}^{(l)} := \partial_l \mathbbm{g}, \quad \mathbbm{k}^{(l)} := \partial_l^2 \mathbbm{g} \]
for $l=1,2,\ldots,d$, and also write
\[ \mathbbm{k} := \sum_{l=1}^{d}\mathbbm{k}^{(l)} = \Delta\mathbbm{g}. \]
It is well-known that
\begin{equation}\label{eq:ftofgaussians0}
\widehat{\mathbbm{g}}(\xi) = \mathbbm{g}(\xi).
\end{equation}
Consequently,
\begin{equation}\label{eq:ftofgaussians}
\widehat{\mathbbm{h}^{(l)}}(\xi) = 2\pi\mathbbm{i}\xi_l \,\mathbbm{g}(\xi), \quad
\widehat{\mathbbm{k}^{(l)}}(\xi) = -(2\pi\xi_l)^2 \,\mathbbm{g}(\xi),
\end{equation}
and
\begin{equation}\label{eq:ftofk}
\widehat{\mathbbm{k}}(\xi) = - 4 \pi^2 \|\xi\|_{\ell^2}^2 e^{-\pi \|\xi\|_{\ell^2}^2}
\end{equation}
for each $\xi=(\xi_1,\ldots,\xi_d)\in\mathbb{R}^d$. Identities
\begin{equation}\label{eq:teleconv0}
\mathbbm{g}_{a} \ast \mathbbm{g}_{b} = \mathbbm{g}_{\sqrt{a^2+b^2}}, \quad
\mathbbm{h}^{(l)}_{a} \ast \mathbbm{h}^{(l)}_{b} = \frac{ab}{a^2+b^2} \mathbbm{k}^{(l)}_{\sqrt{a^2+b^2}}
\end{equation}
and
\begin{equation}\label{eq:teleconv}
\mathbbm{h}^{(l)}_{a} \ast \mathbbm{g}_{b} = \frac{a}{\sqrt{a^2+b^2}} \mathbbm{h}^{(l)}_{\sqrt{a^2+b^2}}, \quad
\mathbbm{k}^{(l)}_{a} \ast \mathbbm{g}_{b} = \frac{a^2}{a^2+b^2} \mathbbm{k}^{(l)}_{\sqrt{a^2+b^2}}
\end{equation}
are immediate consequences of \eqref{eq:ftofgaussians0}, \eqref{eq:ftofgaussians}, and the fact that the Fourier transform changes convolution to the pointwise product.
Straightforward differentiation easily gives
\begin{equation}\label{eq:heatequation}
\frac{\partial}{\partial t} \big( \mathbbm{g}_{t}(x) \big) = \frac{1}{2\pi t} \mathbbm{k}_{t}(x).
\end{equation}
One can notice that \eqref{eq:heatequation} is actually the \emph{heat equation}, rewritten using a quadratic change of the time variable.

Let us define several auxiliary quantities involving these Gaussian functions. They will depend on quite a few parameters: $k\in\{1,2,\ldots,n-1\}$, $l\in\{1,2,\ldots,d\}$, $m\in\{k,k+1,\ldots,n-1\}$, $\alpha,\alpha_1,\ldots,\alpha_{n-1}\in(0,\infty)$, and $t\in(0,\infty)$. The quantities will all be functions of $y,u_1,\ldots,u_{n-1}\in\mathbb{R}^d$, but in favor of brevity we omit writing their variables and some of the parameters. These quantities are:
\begin{align*}
K_{k,l,m,t} := - \int_{(\mathbb{R}^d)^{n-k}}
& \mathbbm{g}_{t\alpha}(y+p_k+\cdots +p_{n-1}) \mathbbm{h}^{(l)}_{t\alpha_m}(p_m) \mathbbm{h}^{(l)}_{t\alpha_m}(u_m-p_m) \\
& \times \Big( \prod_{\substack{k\leq j\leq n-1\\ j\neq m}} \mathbbm{g}_{t\alpha_j}(p_j) \mathbbm{g}_{t\alpha_j}(u_j-p_j) \Big)
\,\textup{d}p_k\cdots \textup{d}p_{n-1},
\end{align*}
\begin{align*}
L_{k,l,t} := -\frac{1}{2}\Big(1+ \alpha^{-2} \sum_{m=k}^{n-1} \alpha_m^2 \Big) \int_{(\mathbb{R}^d)^{n-k}}
& \mathbbm{k}^{(l)}_{t\alpha}(y+p_k+\cdots +p_{n-1}) \\
& \times \Big( \prod_{j=k}^{n-1} \mathbbm{g}_{t\alpha_j}(p_j) \mathbbm{g}_{t\alpha_j}(u_j-p_j) \Big)
\,\textup{d}p_k\cdots \textup{d}p_{n-1},
\end{align*}
and
\[ M_{k,t} := \int_{(\mathbb{R}^d)^{n-k}} \mathbbm{g}_{t\alpha}(y+p_k+\cdots +p_{n-1})
\Big ( \prod_{j=k}^{n-1} \mathbbm{g}_{t\alpha_j}(p_j) \mathbbm{g}_{t\alpha_j}(u_j-p_j) \Big ) \,\textup{d}p_k\cdots \textup{d}p_{n-1}. \]
We will also find useful to define:
\begin{align}
\widetilde{K}_{l,m,t} := - \int_{(\mathbb{R}^d)^{n-1}} & \mathbbm{h}^{(l)}_{t\alpha_m}(p_m) \mathbbm{h}^{(l)}_{t\alpha_m}(u_m-p_m) \nonumber \\
& \Big( \prod_{\substack{1\leq j\leq n-1\\ j\neq m}} \mathbbm{g}_{t\alpha_j}(p_j) \mathbbm{g}_{t\alpha_j}(u_j-p_j) \Big) \,\textup{d}p_1\cdots \textup{d}p_{n-1} \label{eq:Ktildeisdefined}
\end{align}
and
\[ \widetilde{M}_{t} := \int_{(\mathbb{R}^d)^{n-1}} \prod_{j=1}^{n-1} \mathbbm{g}_{t\alpha_j}(p_j) \mathbbm{g}_{t\alpha_j}(u_j-p_j) \,\textup{d}p_1\cdots \textup{d}p_{n-1}. \]
Observe that some of the above quantities are defined with negative signs before the integrals. These signs matter, since the following proof crucially depends on the fact that certain expressions are nonnegative.

\begin{lemma}\label{lemma:Gaussians}
For real numbers $a$ and $b$ such that $0<a<b$, one has
\begin{equation}\label{eq:Gaussianidentity1}
\int_{a}^{b} \sum_{l=1}^d \Big( L_{k,l,t} + \sum_{m=k}^{n-1} K_{k,l,m,t} \Big) \frac{\textup{d}t}{t} = \pi \big( M_{k,a} - M_{k,b} \big)
\end{equation}
and
\begin{equation}\label{eq:Gaussianidentity2}
\int_{a}^{b} \sum_{l=1}^d \sum_{m=1}^{n-1} \widetilde{K}_{l,m,t} \,\frac{\textup{d}t}{t} = \pi \big( \widetilde{M}_{a} - \widetilde{M}_{b} \big).
\end{equation}
\end{lemma}

\begin{proof}
Identity \eqref{eq:Gaussianidentity1} could, in principle, be proven by evaluating several Gaussian integrals. In practice this would lead to very messy computations, so we approach it differently, analogously as it was done in \cite{DKT16:splx}.
Observe that $K_{k,l,m,t}$ is the integral of the function $H\colon(\mathbb{R}^d)^{2n-2k+1}\to\mathbb{C}$,
\begin{align*}
H(q,q^0_k,q_k^1,\ldots, q^0_{n-1},q_{n-1}^1) :=\,
& \mathbbm{g}_{t\alpha}(q+y) \mathbbm{h}^{(l)}_{t\alpha_m}(q^0_m) \mathbbm{h}^{(l)}_{t\alpha_m}(q^1_m-u_m) \\
& \times \Big( \prod_{\substack{k\leq j\leq n-1\\ j\neq m}}
\mathbbm{g}_{t\alpha_j}(q^0_j) \mathbbm{g}_{t\alpha_j}(q^1_j-u_j) \Big)
\end{align*}
over the linear space
\[ \big\{(p_k+\cdots +p_{n-1},\, p_k ,\, p_k,\, p_{k+1},\, p_{k+1},\ldots, p_{n-1},\,p_{n-1}) : p_k,\ldots, p_{n-1}\in \mathbb{R}^d\big\}. \]
Therefore it also equals the integral of the Fourier transform of $H$ over its orthogonal complement,
\[ \big\{(\eta,\, \xi_k ,\, -\xi_k-\eta,\, \xi_{k+1},\, -\xi_{k+1}-\eta,\ldots, \xi_{n-1},\,-\xi_{n-1}-\eta) : \eta, \xi_k,\ldots, \xi_{n-1}\in \mathbb{R}^d\big\}. \]
Using this fact and symmetries of the Fourier transform, we obtain
\begin{align*}
K_{k,l,m,t} =\, & - \int_{(\mathbb{R}^d)^{n-k+1}} \widehat{\mathbbm{g}}(t\alpha\eta) \widehat{\mathbbm{h}^{(l)}}(t\alpha_m\xi_m)\widehat{\mathbbm{h}^{(l)}}(t\alpha_m(\xi_m+\eta)) \\
& \times \Big(\prod_{\substack{k\leq j\leq n-1\\ j\neq m}} \widehat{\mathbbm{g}}(t\alpha_j\xi_j)\widehat{\mathbbm{g}}(t\alpha_j(\xi_j+\eta))\Big)
\,e^{2\pi \mathbbm{i} \big( y\cdot\eta + \sum_{j=k}^{n-1}u_j\cdot(\xi_j+\eta) \big )} \,\textup{d}\eta \,\textup{d}\xi_k\cdots \textup{d}\xi_{n-1}.
\end{align*}
Formula \eqref{eq:ftofgaussians} finally gives
\begin{align*}
K_{k,l,m,t} = \int_{(\mathbb{R}^d)^{n-k+1}} & 4\pi^2 t^2 \alpha_m^2 \xi_m^l (\xi_m^l+\eta^l) \,\widehat{\mathbbm{g}}(t\alpha\eta)
\Big(\prod_{j=k}^{n-1} \widehat{\mathbbm{g}}(t\alpha_j\xi_j)\widehat{\mathbbm{g}}(t\alpha_j(\xi_j+\eta))\Big) \\
& \times e^{2\pi \mathbbm{i} \big( y\cdot\eta + \sum_{j=k}^{n-1}u_j\cdot(\xi_j+\eta) \big )} \,\textup{d}\eta \,\textup{d}\xi_k\cdots \textup{d}\xi_{n-1},
\end{align*}
where $\xi_j=(\xi_j^1,\ldots,\xi_j^d)$ and $\eta=(\eta^1,\ldots,\eta^d)$.
Similarly, $L_{k,l,t}$ can be expressed as
\begin{align*}
L_{k,l,t} = \int_{(\mathbb{R}^d)^{n-k+1}} & 2\pi^2 t^2 \Big(\alpha^{2} + \sum_{m=k}^{n-1} \alpha_m^2 \Big) (\eta^l)^2 \,\widehat{\mathbbm{g}}(t\alpha\eta) \Big(\prod_{j=k}^{n-1} \widehat{\mathbbm{g}}(t\alpha_j\xi_j)\widehat{\mathbbm{g}}(t\alpha_j(\xi_j+\eta))\Big) \\
& \times e^{2\pi \mathbbm{i} \big( y\cdot\eta + \sum_{j=k}^{n-1}u_j\cdot(\xi_j+\eta) \big )} \,\textup{d}\eta \,\textup{d}\xi_k\cdots \textup{d}\xi_{n-1}
\end{align*}
and $M_{k,t}$ can be expressed simply as
\begin{align*}
M_{k,t} = \int_{(\mathbb{R}^d)^{n-k+1}} & \widehat{\mathbbm{g}}(t\alpha\eta)
\Big(\prod_{j=k}^{n-1} \widehat{\mathbbm{g}}(t\alpha_j\xi_j)\widehat{\mathbbm{g}}(t\alpha_j(\xi_j+\eta))\Big) \\
& \times e^{2\pi \mathbbm{i} \big( y\cdot\eta + \sum_{j=k}^{n-1}u_j\cdot(\xi_j+\eta) \big )} \,\textup{d}\eta \,\textup{d}\xi_k\cdots \textup{d}\xi_{n-1}.
\end{align*}
We see that the identity from the statement of Lemma~\ref{lemma:Gaussians} will follow from
\begin{align*}
& \int_{a}^{b} \bigg( 2\pi t \Big( \alpha^2 + \sum_{m=k}^{n-1} \alpha_m^2 \Big) \|\eta\|_{\ell^2}^2 +  4\pi t \sum_{m=k}^{n-1} \alpha_m^2 \xi_m\cdot (\xi_m+\eta) \bigg) \\
& \qquad \times \widehat{\mathbbm{g}}(t\alpha\eta) \Big(\prod_{j=k}^{n-1} \widehat{\mathbbm{g}}(t\alpha_j\xi_j)\widehat{\mathbbm{g}}(t\alpha_j(\xi_j+\eta))\Big) \,\textup{d}t \\
& = - \bigg( \widehat{\mathbbm{g}}(t\alpha\eta)
\Big(\prod_{j=k}^{n-1} \widehat{\mathbbm{g}}(t\alpha_j\xi_j)\widehat{\mathbbm{g}}(t\alpha_j(\xi_j+\eta))\Big) \bigg) \bigg|_{t=a}^{t=b}
\end{align*}
after multiplication with the complex exponential and integration in the variables $\eta$, $\xi_k$, \ldots, $\xi_{n-1}$.
Via the fundamental theorem of calculus, the last equality is equivalent to its differential formulation,
\begin{align*}
& \frac{\partial}{\partial t} \bigg( \widehat{\mathbbm{g}}(t\alpha\eta)
\Big(\prod_{j=k}^{n-1} \widehat{\mathbbm{g}}(t\alpha_j\xi_j)\widehat{\mathbbm{g}}(t\alpha_j(\xi_j+\eta))\Big) \bigg) \\
& = -2\pi t \Big( \alpha^2 \|\eta\|_{\ell^2}^2 + \sum_{j=k}^{n-1} \alpha_j^2 \big( 2\|\xi_j\|_{\ell^2}^2 + 2\xi_j\cdot\eta + \|\eta\|_{\ell^2}^2 \big) \Big) \\
& \qquad \times \widehat{\mathbbm{g}}(t\alpha\eta) \Big( \prod_{j=k}^{n-1} \widehat{\mathbbm{g}}(t\alpha_j\xi_j) \widehat{\mathbbm{g}}(t\alpha_j(\xi_j+\eta)) \Big),
\end{align*}
while this identity is readily verified by recalling \eqref{eq:ftofgaussians0}, writing
\[ \widehat{\mathbbm{g}}(t\alpha\eta) \Big(\prod_{j=k}^{n-1} \widehat{\mathbbm{g}}(t\alpha_j\xi_j)\widehat{\mathbbm{g}}(t\alpha_j(\xi_j+\eta))\Big)
= e^{-\pi t^2 \left( \alpha^2\|\eta\|_{\ell^2}^2 + \sum_{j=k}^{n-1} \alpha_j^2 ( \|\xi_j\|_{\ell^2}^2 + \|\xi_j+\eta\|_{\ell^2}^2 ) \right)}, \]
and using the chain rule. This proves Identity \eqref{eq:Gaussianidentity1}.

Identity \eqref{eq:Gaussianidentity2} can now be deduced simply by specializing \eqref{eq:Gaussianidentity1} to $k=1$, integrating it in $y$ over $\mathbb{R}^d$, and observing $\int_{\mathbb{R}^d}\mathbbm{k}^{(l)}_{t\alpha}=0$, which makes $L_{k,l,t}$ vanish.
However, \eqref{eq:Gaussianidentity2} is actually much simpler to prove, as this time the corresponding Gaussian integrals can be evaluated quite easily. For an alternative proof of \eqref{eq:Gaussianidentity2} we use Fubini's theorem and convolution identities \eqref{eq:teleconv0} to rewrite $\widetilde{K}_{l,m,t}$ and $\widetilde{M}_{t}$ as
\[ \widetilde{K}_{l,m,t} = - \frac{1}{2} \mathbbm{k}^{(l)}_{2^{1/2}t\alpha_m}(u_m)
\prod_{\substack{1\leq j\leq n-1\\ j\neq m}} \mathbbm{g}_{2^{1/2}t\alpha_j}(u_j) \]
and
\[ \widetilde{M}_{t} = \prod_{j=1}^{n-1} \mathbbm{g}_{2^{1/2}t\alpha_j}(u_j). \]
Now the product rule for differentiation and the heat equation \eqref{eq:heatequation} together give
\begin{align*}
\frac{\partial}{\partial t}\widetilde{M}_{t}
& = \sum_{m=1}^{n-1} \frac{1}{2\pi t} \mathbbm{k}_{2^{1/2}t\alpha_m}(u_m) \prod_{\substack{1\leq j\leq n-1\\ j\neq m}} \mathbbm{g}_{2^{1/2}t\alpha_j}(u_j) \\
& = \frac{1}{2\pi t} \sum_{m=1}^{n-1} \sum_{l=1}^{d} \mathbbm{k}^{(l)}_{2^{1/2}t\alpha_m}(u_m) \prod_{\substack{1\leq j\leq n-1\\ j\neq m}} \mathbbm{g}_{2^{1/2}t\alpha_j}(u_j)
= -\frac{1}{\pi t} \sum_{m=1}^{n-1} \sum_{l=1}^{d} \widetilde{K}_{l,m,t},
\end{align*}
which is precisely the differential formulation of Identity \eqref{eq:Gaussianidentity2}.
\end{proof}

For $k\in\{0,1,\ldots,n-1\}$ denote
\[ \mathcal{F}_k := \prod_{i=0}^{k} \prod_{(r_{k+1},\ldots, r_{n-1})\in\{0,1\}^{n-k-1}} \mathbbm{1}_{A}\Big( x+i y + \sum_{s={k+1}}^{n-1} r_s (i+s-k) u_s \Big). \]
Note that this is a function of $x,y,u_{k+1},\ldots, u_{n-1}\in \mathbb{R}^d$.
For $k=n-1$ the sum in $s$ is interpreted as $0$ and the product over $(r_{k+1},\ldots,r_{n-1})$ is interpreted simply as $\mathbbm{1}_{A}( x+i y)$, so $\mathcal{F}_{n-1}=\mathcal{F}_{n-1}(x,y)$ becomes exactly $\mathcal{F}(x,y)$, which was defined previously in \eqref{eq:Fdefinedasaproduct}.
Let us also agree to write
\begin{equation}\label{eq:Gisdefined}
\mathcal{G}^{\beta_1,\beta_2,\ldots,\beta_{n-1}} := \prod_{(r_{1},\ldots, r_{n-1})\in\{0,1\}^{n-1}} \mathbbm{1}_{A}\Big( x + \sum_{s={1}}^{n-1} r_s \beta_s u_s \Big)
\end{equation}
for parameters $\beta_1,\ldots,\beta_{n-1}\in\mathbb{R}\setminus\{0\}$. This is a function of $x,u_{1},\ldots, u_{n-1}\in \mathbb{R}^d$.

In order to control \eqref{eq:multiscaleHilbert1} we will introduce two finite sequences of quantities with increasing degrees of complexity. Take $k=\{1,\ldots,n-1\}$, $l\in\{1,\ldots,d\}$, $a,b\in(0,\infty)$ such that $a<b$, and $\alpha,\alpha_k,\ldots, \alpha_{n-1}\in (0,\infty)$.
We define
\begin{align*}
\widetilde{\Lambda}_{k,l,a,b}^{\alpha,\alpha_k,\ldots, \alpha_{n-1}} :=\, & \int_{a}^{b} \int_{(\mathbb{R}^d)^{2(n-k-1)}} \Big| \int_{(\mathbb{R}^d)^3}\mathcal{F}_k \, \mathbbm{h}^{(l)}_{t\alpha_k}(y-p_{k}) \mathbbm{h}^{(l)}_{t\alpha}(p_{k}+\cdots +p_{n-1}) \,\textup{d}p_k\,\textup{d}x\,\textup{d}y \Big| \\
& \times \Big ( \prod_{j=k+1}^{n-1} \mathbbm{g}_{t\alpha_j}(p_j) \mathbbm{g}_{t\alpha_j}(u_j-p_j) \Big ) \,\textup{d}p_{k+1}\cdots \textup{d}p_{n-1} \,\textup{d}u_{k+1}\cdots \textup{d}u_{n-1} \,\frac{\textup{d}t}{t}
\end{align*}
and, if $k\geq 2$, then we also define
\begin{align*}
\Lambda_{k,l,a,b}^{\alpha,\alpha_k,\ldots, \alpha_{n-1}} :=\, & \int_{a}^{b} \int_{(\mathbb{R}^d)^{2(n-k)}} \Big| \int_{\mathbb{R}^d}\mathcal{F}_k \, \mathbbm{h}^{(l)}_{t\alpha_k}(y-p_k) \,\textup{d}y \Big| \,\mathbbm{g}_{t\alpha}(p_k+\cdots +p_{n-1}) \\
& \times \Big ( \prod_{j=k+1}^{n-1} \mathbbm{g}_{t\alpha_j}(p_j) \mathbbm{g}_{t\alpha_j}(u_j-p_j) \Big ) \,\textup{d}p_k\cdots \textup{d}p_{n-1} \,\textup{d}u_{k+1}\cdots \textup{d}u_{n-1} \,\textup{d}x \,\frac{\textup{d}t}{t} .
\end{align*} 	
Note that, for $k=n-1$ the products in $j$ are interpreted simply as being equal to $1$.

\begin{lemma}\label{lemma:telescoping}
For every $l\in\{1,\ldots,d\}$, $a,b\in(0,\infty)$ such that $a<b$, and $\alpha, \alpha_1,\ldots, \alpha_{n-1}\in (0,\infty)$ we have
\begin{equation}\label{eq:lemmateleineq1}
\widetilde{\Lambda}_{1,l,a,b}^{\alpha,\alpha_1,\ldots, \alpha_{n-1}} \lesssim 1.
\end{equation}
For every $k=\{2,\ldots,n-1\}$, $l\in\{1,\ldots,d\}$, $a,b\in(0,\infty)$ such that $b\geq 3a$, and $\alpha,\alpha_k,\ldots,\alpha_{n-1}$
$\in [2^{-(n-k)/2},\infty)$ %artificial formatting
we have
\begin{equation}\label{eq:lemmateleineq2}
\Lambda_{k,l,a,b}^{\alpha,\alpha_k,\ldots, \alpha_{n-1}},\, \widetilde{\Lambda}_{k,l,a,b}^{\alpha,\alpha_k,\ldots, \alpha_{n-1}} \lesssim_{n,d}  (\alpha\alpha_k\cdots \alpha_{n-1})^2 \Big(\log\frac{b}{a}\Big)^{1-2^{-k+1}}.
\end{equation}
\end{lemma}

\begin{proof}
To prove the lemma we induct on $k=\{1,\ldots,n-1\}$. More precisely, by the mathematical induction on $k$ we simultaneously prove that quantities $\Lambda_{k,l,a,b}^{\alpha,\alpha_k,\ldots, \alpha_{n-1}}$ and $\widetilde{\Lambda}_{k,l,a,b}^{\alpha,\alpha_k,\ldots, \alpha_{n-1}}$ satisfy the claimed Estimates \eqref{eq:lemmateleineq1} and \eqref{eq:lemmateleineq2}.

For the induction basis $k=1$ we only have one of the two quantities defined, i.e.\@ we need to show only one estimate, which is \eqref{eq:lemmateleineq1}.
Inserting the definition of $\mathcal{F}_1$, changing the variable $y$ by introducing $z=x+y$, and also changing $p_1\rightarrow p_1-x$, we see that $\widetilde{\Lambda}_{1,l,a,b}^{\alpha,\alpha_1,\ldots,\alpha_{n-1}}$ equals
\begin{align*}
\int_{a}^{b} \int_{(\mathbb{R}^d)^{2(n-2)}} & \bigg| \int_{\mathbb{R}^d} \bigg( \int_{\mathbb{R}^d} \Big( \prod_{(r_{2},\ldots, r_{n-1})\in\{0,1\}^{n-2}} \mathbbm{1}_{A}\Big(x + \sum_{s=2}^{n-1} r_s (s-1) u_s\Big) \Big) \\
& \qquad\qquad\qquad\qquad\qquad\qquad\qquad \times \mathbbm{h}^{(l)}_{t\alpha}(-x+p_{1}+\cdots +p_{n-1}) \,\textup{d}x \bigg) \\
& \times \bigg( \int_{\mathbb{R}^d} \Big ( \prod_{(r_{2},\ldots, r_{n-1})\in\{0,1\}^{n-2}} \mathbbm{1}_{A}\Big(z + \sum_{s={2}}^{n-1} r_s s u_s\Big)\Big)
\, \mathbbm{h}^{(l)}_{t\alpha_1}(z-p_{1}) \,\textup{d}z \bigg) \,\textup{d}p_1 \bigg| \\
& \times \Big ( \prod_{j=2}^{n-1} \mathbbm{g}_{t\alpha_j}(p_j) \mathbbm{g}_{t\alpha_j}(u_j-p_j) \Big ) \,\textup{d}p_2\cdots \textup{d}p_{n-1} \,\textup{d}u_2\cdots \textup{d}u_{n-1} \frac{\textup{d}t}{t}.
\end{align*}
We use the triangle inequality for the integral in $p_1$ and the Cauchy--Schwarz inequality in $p_1,\ldots, p_{n-1}$, $u_2,\ldots, u_{n-1}$, and $t$. Then we also shift $p_1 \rightarrow p_1-p_2-\cdots-p_{n-1}$ in the first factor. That way we bound
\begin{equation}\label{eq:iljlaux1}
\widetilde{\Lambda}_{1,l,a,b}^{\alpha,\alpha_1,\ldots,\alpha_{n-1}} \leq \mathcal{I}^{1/2} \,\mathcal{J}^{1/2},
\end{equation}
where
\begin{align*}
\mathcal{I} := \int_{a}^{b} \int_{(\mathbb{R}^d)^{2n-3}} & \bigg( \int_{\mathbb{R}^d} \Big( \prod_{(r_{2},\ldots, r_{n-1})\in\{0,1\}^{n-2}} \mathbbm{1}_{A}\Big(x + \sum_{s=2}^{n-1} r_s (s-1) u_s\Big)\Big) \,\mathbbm{h}^{(l)}_{t\alpha}(x-p_{1}) \,\textup{d}x \bigg)^2 \\
& \times \Big( \prod_{j=2}^{n-1} \mathbbm{g}_{t\alpha_j}(p_j) \mathbbm{g}_{t\alpha_j}(u_j-p_j) \Big) \,\textup{d}p_1\cdots \textup{d}p_{n-1} \,\textup{d}u_2\cdots \textup{d}u_{n-1} \frac{\textup{d}t}{t}
\end{align*}
and
\begin{align*}
\mathcal{J} := \int_{a}^{b} \int_{(\mathbb{R}^d)^{2n-3}} & \bigg( \int_{\mathbb{R}^d} \Big ( \prod_{(r_{2},\ldots, r_{n-1})\in\{0,1\}^{n-2}} \mathbbm{1}_{A}\Big(x + \sum_{s={2}}^{n-1} r_s s u_s\Big)\Big) \,\mathbbm{h}^{(l)}_{t\alpha_1}(x-p_{1}) \,\textup{d}x \bigg)^2 \\
& \times \Big( \prod_{j=2}^{n-1} \mathbbm{g}_{t\alpha_j}(p_j) \mathbbm{g}_{t\alpha_j}(u_j-p_j) \Big) \,\textup{d}p_1\cdots \textup{d}p_{n-1} \,\textup{d}u_2\cdots \textup{d}u_{n-1} \frac{\textup{d}t}{t}.
\end{align*}
In each of the two expressions above we write the square of the integral in $x$ as a double integral in $x$ and $x'$, then we change the variable $x'$ to $x+u_1$, and finally shift $p_1 \rightarrow p_1+x$.
That way we can recognize $\mathcal{I}$ and $\mathcal{J}$ respectively as
\begin{equation}\label{eq:iljlaux2}
\mathcal{I} = \widetilde{\Theta}_{l,1,a,b}^{\alpha,\alpha_2,\ldots, \alpha_{n-1},1,1,2,\ldots,n-2},\quad
\mathcal{J} = \widetilde{\Theta}_{l,1,a,b}^{\alpha_1,\alpha_2,\ldots, \alpha_{n-1},1,2,3,\ldots,n-1},
\end{equation}
where we define, quite generally,
\[ \widetilde{\Theta}_{l,m,a,b}^{\alpha_1,\alpha_2,\ldots, \alpha_{n-1},\beta_1,\beta_2,\beta_3,\ldots,\beta_{n-1}}
:= \int_{a}^{b} \int_{(\mathbb{R}^d)^{n}} \mathcal{G}^{\beta_1,\ldots,\beta_{n-1}} \widetilde{K}_{l,m,t}
\,\textup{d}u_{1}\cdots \textup{d}u_{n-1} \,\textup{d}x \,\frac{\textup{d}t}{t} \]
for $m\in\{1,2,\ldots,n-1\}$, $\alpha_1,\alpha_2,\ldots, \alpha_{n-1}\in (0,\infty)$, and $\beta_1,\beta_2,\ldots,\beta_{n-1}\in\mathbb{R}\setminus\{0\}$,
and we recall that $\mathcal{G}^{\beta_1,\ldots,\beta_{n-1}}$ and $\widetilde{K}_{l,m,t}$ were introduced in \eqref{eq:Gisdefined} and \eqref{eq:Ktildeisdefined}, respectively.
Recalling \eqref{eq:iljlaux1} and \eqref{eq:iljlaux2} we see that \eqref{eq:lemmateleineq1} will follow from the bound
\begin{equation}\label{eq:iljlaux3}
\widetilde{\Theta}_{l,m,a,b}^{\alpha_1,\alpha_2,\ldots, \alpha_{n-1},\beta_1,\beta_2,\beta_3,\ldots,\beta_{n-1}} \lesssim 1,
\end{equation}
which will be shown to hold uniformly in all of the parameters: $l$, $m$, $a$, $b$, $\alpha_1,\ldots, \alpha_{n-1}$, $\beta_1,\ldots,\beta_{n-1}$ and even in $n$ and $d$.

On $\widetilde{\Theta}_{l,m,a,b}^{\alpha_1,\ldots,\beta_{n-1}}$ we perform the reverse of the procedure applied to $\mathcal{I}$ and $\mathcal{J}$. Introducing the variable $x'=x+\beta_m u_m$, shifting $p_m\rightarrow p_m-x/\beta_m$, and observing that the integrals in $x$ and $x'$ are the same, we obtain
\begin{align*}
\widetilde{\Theta}_{l,m,a,b}^{\alpha_1,\ldots,\beta_{n-1}} = \int_{a}^{b} \int_{(\mathbb{R}^d)^{2n-3}}
& \bigg( \int_{\mathbb{R}^d} \prod_{(r_1,\ldots,r_{m-1},r_{m+1},\ldots,r_{n-1})\in\{0,1\}^{n-2}}
\mathbbm{1}_{A}\Big( x + \sum_{\substack{1\leq s\leq n-1 \\s\neq m}} r_s \beta_s u_s \Big) \\
& \times \mathbbm{h}^{(l)}_{t\alpha_m}\Big(\frac{x}{\beta_m}-p_{m}\Big) \,\textup{d}x \bigg)^2
\Big( \prod_{\substack{1\leq j\leq n-1\\ j\neq m}} \mathbbm{g}_{t\alpha_j}(p_j) \mathbbm{g}_{t\alpha_j}(u_j-p_j) \Big) \\
& \,\textup{d}p_1\cdots \textup{d}p_{n-1} \,\textup{d}u_{1}\cdots\textup{d}u_{m-1}\,\textup{d}u_{m+1}\cdots\textup{d}u_{n-1} \,\frac{\textup{d}t}{t},
\end{align*}
which clearly shows that
\begin{equation}\label{eq:telebasis1}
\widetilde{\Theta}_{l,m,a,b}^{\alpha_1,\ldots,\beta_{n-1}}\geq 0
\end{equation}
for every $l$ and $m$. We also find it convenient to define
\[ \widetilde{\Xi}_{t}^{\alpha_1,\ldots,\beta_{n-1}} := \int_{(\mathbb{R}^d)^{n}} \mathcal{G}^{\beta_1,\ldots,\beta_{n-1}} \widetilde{M}_{t}
\,\textup{d}u_{1}\cdots \textup{d}u_{n-1} \,\textup{d}x \]
for $t>0$, so that Identity~\eqref{eq:Gaussianidentity2} gives
\[ \sum_{l=1}^d \sum_{m=1}^{n-1} \widetilde{\Theta}_{l,m,a,b}^{\alpha_1,\ldots,\beta_{n-1}}
= \pi \big( \widetilde{\Xi}_{a}^{\alpha_1,\ldots,\beta_{n-1}} - \widetilde{\Xi}_{b}^{\alpha_1,\ldots,\beta_{n-1}} \big). \]
Because of $\widetilde{M}_{t}\geq 0$ we also have
\[ \widetilde{\Xi}_{b}^{\alpha_1,\ldots,\beta_{n-1}}\geq 0, \]
while making a rough estimate
\[ \mathcal{G}^{\beta_1,\ldots,\beta_{n-1}} \leq \mathbbm{1}_{A}(x) \leq \mathbbm{1}_{[0,1]}(x) \]
and integrating in the order $x$, $u_1,\ldots,u_{n-1}$, $p_1,\ldots,p_{n-1}$ we get
\[ \widetilde{\Xi}_{a}^{\alpha_1,\ldots,\beta_{n-1}} \leq 1. \]
Therefore,
\begin{equation}\label{eq:telebasis2}
\sum_{l=1}^d \sum_{m=1}^{n-1} \widetilde{\Theta}_{l,m,a,b}^{\alpha_1,\ldots,\beta_{n-1}} \leq \pi.
\end{equation}
In remains to note that the left hand side of \eqref{eq:telebasis2} is a sum of nonnegative terms \eqref{eq:telebasis1}, so each of these terms is also bounded from above by $\pi$. This proves Estimate \eqref{eq:iljlaux3}.

Now we turn to the inductive step: for a fixed $k\in\{2,\ldots,n-1\}$ we need to deduce \eqref{eq:lemmateleineq2} assuming that the estimates (either \eqref{eq:lemmateleineq1}, or \eqref{eq:lemmateleineq2}, depending on $k$) hold for $k-1$ in place of $k$.
We first handle ${\Lambda}_{k,l,a,b}^{\alpha,\alpha_k,\ldots, \alpha_{n-1}}$, with all of its parameters as in the statement of the lemma.
In the definition of $\mathcal{F}_k$ we split the product in $i$ into the factor corresponding to $i=0$ and the factors corresponding to $i\geq 1$.
Applying the Cauchy--Schwarz inequality in all variables but $y$, we obtain the bound
\begin{equation}\label{eq:telestep1}
{\Lambda}_{k,l,a,b}^{\alpha,\alpha_k,\ldots, \alpha_{n-1}} \leq \mathcal{S}^{1/2}\,\mathcal{T}^{1/2} ,
\end{equation}
where
\begin{align*}
\mathcal{S} := \int_{a}^{b} \int_{(\mathbb{R}^d)^{2(n-k)}} & \bigg( \prod_{(r_{k+1},\ldots, r_{n-1})\in\{0,1\}^{n-k-1}} \mathbbm{1}_{A} \Big(x + \sum_{s={k+1}}^{n-1} r_s(s-k)u_s \Big) \bigg) \\
& \times \mathbbm{g}_{t\alpha}(p_k+\cdots +p_{n-1}) \Big( \prod_{j=k+1}^{n-1} \mathbbm{g}_{t\alpha_j}(p_j) \mathbbm{g}_{t\alpha_j}(u_j-p_j) \Big) \\
& \,\textup{d}p_k\cdots \textup{d}p_{n-1} \,\textup{d}u_{k+1}\cdots \textup{d}u_{n-1} \,\textup{d}x \,\frac{\textup{d}t}{t}
\end{align*}
and
\begin{align*}
\mathcal{T} := \int_{a}^{b} \int_{(\mathbb{R}^d)^{2(n-k)}} & \bigg( \int_{\mathbb{R}^d} \prod_{i=1}^{k} \prod_{(r_{k+1},\ldots, r_{n-1})\in\{0,1\}^{n-k-1}} \mathbbm{1}_{A} \Big( x+i y + \sum_{s={k+1}}^{n-1} r_s (i+s-k) u_s \Big) \\
& \times \mathbbm{h}^{(l)}_{t\alpha_k}(y-p_k) \,\textup{d}y \bigg)^2
\mathbbm{g}_{t\alpha}(p_k+\cdots +p_{n-1}) \Big( \prod_{j=k+1}^{n-1} \mathbbm{g}_{t\alpha_j}(p_j) \mathbbm{g}_{t\alpha_j}(u_j-p_j) \Big) \\
& \,\textup{d}p_k\cdots \textup{d}p_{n-1} \,\textup{d}u_{k+1}\cdots \textup{d}u_{n-1} \,\textup{d}x \,\frac{\textup{d}t}{t} .
\end{align*}

To bound $\mathcal{S}$, we estimate the product involving $x$ simply by $\mathbbm{1}_{[0,1]}(x)$. Then we integrate in the order
$x$, $u_{k+1},\ldots,u_{n-1}$, $p_k,\ldots,p_{n-1}$, $t$, which yields
\begin{equation}\label{eq:telestep2}
\mathcal{S} \leq \log\frac{b}{a}.
\end{equation}
To bound $\mathcal{T}$ we expand out the square, for which we introduce the variable $y'$ as a copy of the variable $y$. Introducing the variable $u_k=y'-y$, changing the order of integration, and shifting $x\rightarrow x-y$ we obtain
{\allowdisplaybreaks\begin{align*}
\mathcal{T} = \int_{a}^{b} \int_{(\mathbb{R}^d)^{2(n-k+1)}} &
\bigg( \prod_{i=1}^{k} \prod_{(r_{k+1},\ldots, r_{n-1})\in\{0,1\}^{n-k-1}} \\
& \mathbbm{1}_{A} \Big( x+(i-1)y + \sum_{s={k+1}}^{n-1} r_s (i+s-k) u_s \Big) \,\mathbbm{h}^{(l)}_{t\alpha_k}(y-p_k) \bigg) \\
& \times \bigg( \prod_{i=1}^{k} \prod_{(r_{k+1},\ldots, r_{n-1})\in\{0,1\}^{n-k-1}} \\
& \qquad \mathbbm{1}_{A} \Big( x+(i-1)y + i u_k + \sum_{s={k+1}}^{n-1} r_s (i+s-k) u_s \Big) \, \mathbbm{h}^{(l)}_{t\alpha_k}(y+u_k-p_k) \bigg) \\
& \times \mathbbm{g}_{t\alpha}(p_k+\cdots +p_{n-1}) \Big( \prod_{j=k+1}^{n-1} \mathbbm{g}_{t\alpha_j}(p_j) \mathbbm{g}_{t\alpha_j}(u_j-p_j) \Big) \\
& \,\textup{d}p_k\cdots \textup{d}p_{n-1} \,\textup{d}u_{k}\cdots \textup{d}u_{n-1} \,\textup{d}x \,\textup{d}y \,\frac{\textup{d}t}{t}.
\end{align*}}
If we also change $i\rightarrow i+1$ and $p_k\rightarrow p_k+y$, then we can recognize the last display as the special case $m=k$ of the following expression defined for $l\in\{1,\ldots,d\}$ and $m\in\{k,\ldots,n-1\}$:
\begin{equation}\label{eq:Thetaisdefined}
\Theta_{k,l,m,a,b}^{\alpha,\alpha_k,\ldots, \alpha_{n-1}} := \int_{a}^{b} \int_{(\mathbb{R}^d)^{n-k+2}} \mathcal{F}_{k-1} K_{k,l,m,t} \,\textup{d}u_{k}\cdots \textup{d}u_{n-1} \,\textup{d}x \,\textup{d}y \frac{\textup{d}t}{t},
\end{equation}
i.e.,
\begin{equation}\label{eq:telestep2a}
\mathcal{T} = \Theta_{k,l,k,a,b}^{\alpha,\alpha_k,\ldots, \alpha_{n-1}}.
\end{equation}
Let us also define
\[ \Psi_{k,l,a,b}^{\alpha,\alpha_k,\ldots, \alpha_{n-1}} := \int_{a}^{b} \int_{(\mathbb{R}^d)^{n-k+2}} \mathcal{F}_{k-1} L_{k,l,t} \,\textup{d}u_{k}\cdots \textup{d}u_{n-1} \,\textup{d}x \,\textup{d}y \frac{\textup{d}t}{t} \]
and
\[ \Xi_{k,t}^{\alpha,\alpha_k,\ldots, \alpha_{n-1}} := \int_{(\mathbb{R}^d)^{n-k+2}} \mathcal{F}_{k-1} M_{k,t} \,\textup{d}u_{k}\cdots \textup{d}u_{n-1} \,\textup{d}x \,\textup{d}y \]
for $t>0$.
Identity \eqref{eq:Gaussianidentity1} implies
\begin{equation}\label{eq:telestep3}
\sum_{l=1}^d \Big( \Psi_{k,l,a,b}^{\alpha,\alpha_k,\ldots, \alpha_{n-1}} + \sum_{m=k}^{n-1} \Theta_{k,l,m,a,b}^{\alpha,\alpha_k,\ldots, \alpha_{n-1}} \Big ) = \pi \big( \Xi_{k,a}^{\alpha,\alpha_k,\ldots, \alpha_{n-1}} - \Xi_{k,b}^{\alpha,\alpha_k,\ldots, \alpha_{n-1}} \big).
\end{equation}

Now we want to show that for each $l$ and $m$ the expression $\Theta_{k,l,m,a,b}^{\alpha,\alpha_k,\ldots, \alpha_{n-1}}$ is non-negative. This is clear for $m=k$ as it resulted from an application of the Cauchy--Schwarz inequality.
To see this in general, one performs a change of variables
$x\rightarrow x+(m-k+1)y$ and $p_m\rightarrow p_m-y$ in \eqref{eq:Thetaisdefined}, yielding
\begin{align*}
\int_{a}^{b} \int_{(\mathbb{R}^d)^{2(n-k+1)}} & \prod_{i=0}^{k-1} \prod_{(r_{k},\ldots, r_{n-1})\in\{0,1\}^{n-k}} \\
& \mathbbm{1}_{A} \Big(x+(i+m-k+1)(y+r_mu_m) + \sum_{\substack{k\leq s\leq n-1\\ s\neq m}} r_s (i+s-k+1) u_s \Big) \\
& \times \mathbbm{h}^{(l)}_{t\alpha_m}(y-p_m) \mathbbm{h}^{(l)}_{t\alpha_m}(y+u_m-p_m) \mathbbm{g}_{t\alpha}(p_k+\cdots +p_{n-1}) \\
& \times \Big( \prod_{\substack{k\leq j\leq n-1\\ j\neq m}} \mathbbm{g}_{t\alpha_j}(p_j) \mathbbm{g}_{t\alpha_j}(u_j-p_j) \Big) \,\textup{d}p_k\cdots \textup{d}p_{n-1} \,\textup{d}u_{k}\cdots \textup{d}u_{n-1} \,\textup{d}x \,\textup{d}y \,\frac{\textup{d}t}{t}.
\end{align*}
After the change of variable $u_m$ to $y'=y+u_m$, we observe that the only terms involving variables $y$ and $y'$ are those with functions $\mathbbm{1}_A$ and $\mathbbm{h}^{(l)}_{t\alpha_m}$ and that the integrals in $y$ and $y'$ can be assembled into a square, so
\begin{equation}\label{eq:telestep4}
\Theta_{k,l,m,a,b}^{\alpha,\alpha_k,\ldots, \alpha_{n-1}} \geq 0.
\end{equation}

Next, we turn to controlling $\Psi_{k,l,a,b}^{\alpha,\alpha_k,\ldots, \alpha_{n-1}}$. Using the second convolution identity from \eqref{eq:teleconv} in the form
\begin{align*}
\mathbbm{k}^{(l)}_{t\alpha}(y+p_k+\cdots +p_{n-1}) = 2\int_{\mathbb{R}^d} \mathbbm{h}^{(l)}_{t 2^{-1/2}\alpha} (y- p_{k-1}) \mathbbm{h}^{(l)}_{t 2^{-1/2}\alpha}(p_{k-1}+\cdots + p_{n-1} ) \,\textup{d}p_{k-1},
\end{align*}
we can bound
\[ \big|\Psi_{k,l,a,b}^{\alpha,\alpha_k,\ldots, \alpha_{n-1}}\big|
\leq \Big(1+ \alpha^{-2} \sum_{m=k}^{n-1} \alpha_m^2 \Big)
\,\widetilde{\Lambda}_{k-1,l,a,b}^{2^{-1/2}\alpha, 2^{-1/2}\alpha, \alpha_k,\ldots,\alpha_{n-1}}. \]
By the induction hypothesis (i.e., the statement for $k-1$) applied to each fixed $l$ and using $\alpha\gtrsim_n 1$, $\alpha_m\gtrsim_n 1$ we may therefore estimate
\begin{align}
\big|\Psi_{k,l,a,b}^{\alpha,\alpha_k,\ldots, \alpha_{n-1}}\big|
& \lesssim_{n,d} \alpha^{-2} \Big( \alpha^2 + \sum_{m=k}^{n-1} \alpha_m^2 \Big) (\alpha^2\alpha_k\cdots \alpha_{n-1})^2 \Big(\log\frac{b}{a}\Big)^{1-2^{-k+2}} \nonumber \\
& \lesssim_n (\alpha \alpha_k \cdots \alpha_{n-1})^{4} \Big(\log\frac{b}{a}\Big)^{1-2^{-k+2}}. \label{eq:telestep5}
\end{align}

Finally, quantities $\Xi_{k,t}^{\alpha,\alpha_k,\ldots, \alpha_{n-1}}$ are clearly nonnegative and they can be bounded from above by simply estimating
\[ \mathcal{F}_{k-1} \leq \mathbbm{1}_{A}(x) \leq \mathbbm{1}_{[0,1]}(x) \]
and integrating in the order $x,y$, $u_k,\ldots,u_{n-1}$, $p_k,\ldots,p_{n-1}$, which gives
\begin{equation}\label{eq:telestep6}
\Xi_{k,t}^{\alpha,\alpha_k,\ldots, \alpha_{n-1}} \leq 1.
\end{equation}

Combining \eqref{eq:telestep3}, \eqref{eq:telestep5}, and \eqref{eq:telestep6} we see
\begin{align*}
\sum_{l=1}^d \sum_{m=k}^{n-1} \Theta_{k,l,m,a,b}^{\alpha,\alpha_k,\ldots,\alpha_{n-1}}
& \leq \pi + \sum_{l=1}^d \big|\Psi_{k,l,a,b}^{\alpha,\alpha_k,\ldots,\alpha_{n-1}}\big| \\
& \lesssim_{n,d} (\alpha \alpha_k \cdots \alpha_{n-1})^{4} \Big(\log\frac{b}{a}\Big)^{1-2^{-k+2}}.
\end{align*}
Because of \eqref{eq:telestep4} each individual $\Theta_{k,l,m,a,b}^{\alpha,\alpha_k,\ldots, \alpha_{n-1}}$ satisfies the same bound.
Together with \eqref{eq:telestep1}--\eqref{eq:telestep2a} this proves
\begin{equation}\label{eq:telesteplambda}
\Lambda_{k,l,a,b}^{\alpha,\alpha_k,\ldots, \alpha_{n-1}} \lesssim_{n,d} (\alpha\alpha_k\cdots \alpha_{n-1})^2 \Big(\log\frac{b}{a}\Big)^{1-2^{-k+1}},
\end{equation}
just as we wanted.

At last, we deduce the desired bound for $\widetilde{\Lambda}_{k,l,a,b}^{\alpha,\alpha_k,\ldots, \alpha_{n-1}}$ from \eqref{eq:telesteplambda}.
Substituting $\gamma=\pi\|z\|_{\ell^2}^2\beta^{-2}$ and using the definition of the gamma-function $\Gamma$, it is easy to compute
\begin{align*}
\lim_{\substack{z\in\mathbb{R}^d\\ \|z\|_{\ell^2}\to\infty}} \|z\|_{\ell^2}^{d+3} \int_{1}^{\infty} \mathbbm{g}_\beta(z) \frac{\textup{d}\beta}{\beta^{4}}
& = \frac{1}{2}\pi^{-(d+3)/2} \int_{0}^{\infty} \gamma^{(d+1)/2} e^{-\gamma} \,\textup{d}\gamma \\
& = \frac{1}{2}\pi^{-(d+3)/2} \Gamma\Big(\frac{d+3}{2}\Big) \gtrsim_d 1.
\end{align*}
Therefore, for $z\in \mathbb{R}^d$ we can dominate
\begin{equation}\label{eq:Gaussiandom}
(1+\|z\|_{\ell^2})^{-d-3} \lesssim_d \int_1^\infty \mathbbm{g}_\beta(z) \,\frac{\textup{d}\beta}{\beta^{4}}.
\end{equation}
Since Gaussian tails decay super-polynomially, in particular we have
\begin{equation}\label{eq:teledom}
\big|\mathbbm{h}^{(l)}(z)\big| \lesssim_d \int_1^\infty \mathbbm{g}_\beta(z) \,\frac{\textup{d}\beta}{\beta^{4}}
\end{equation}
for each $l\in\{1,\ldots,d\}$.
By the triangle inequality and \eqref{eq:teledom} applied with $z=p_k+\cdots+p_{n-1}$ we can then bound
\[ \widetilde{\Lambda}_{k,l,a,b}^{\alpha,\alpha_k,\ldots,\alpha_{n-1}} \lesssim_d \int_1^\infty \Lambda_{k,l,a,b}^{\alpha\beta,\alpha_k,\ldots, \alpha_{n-1} } \frac{\textup{d}\beta}{\beta^{4}}. \]
Using \eqref{eq:telesteplambda} and integrating in $\beta$ we get
\begin{align*}
\widetilde{\Lambda}_{k,l,a,b}^{\alpha,\alpha_k,\ldots,\alpha_{n-1}}
& \lesssim_{n,d} \int_1^\infty (\alpha\beta\alpha_k\cdots \alpha_{n-1})^2 \Big(\log\frac{b}{a}\Big)^{1-2^{-k+1}} \frac{\textup{d}\beta}{\beta^{4}} \\
& = (\alpha\alpha_k\cdots \alpha_{n-1})^2 \Big(\log\frac{b}{a}\Big)^{1-2^{-k+1}}.
\end{align*}
This completes the induction step and also finishes the proof of Lemma~\ref{lemma:telescoping}.
\end{proof}

We conjecture that the second estimate in the formulation of Lemma~\ref{lemma:telescoping} holds without the factor containing $\log(b/a)$ on the right hand side. However, removing this factor seems to be a difficult task; see the related discussion at the beginning of this section. Fortunately, Lemma~\ref{lemma:telescoping} will be sufficient for the intended application.

Remember that a function $\varphi$ was chosen in Section~\ref{sec:mainproof} in order to define quantities \eqref{eq:nepsdefinition}.
If we use the same $\varphi$ to define $\varrho\colon\mathbb{R}^d\to\mathbb{R}$ by
\begin{equation}\label{eq:defofrho}
\varrho(x) := d \,\varphi(x) + (\nabla\varphi)(x) \cdot x,
\end{equation}
then
\[ -t \frac{\partial}{\partial t} \big( \varphi_t(x) \big) = \frac{1}{t^d} \Big( d \varphi\Big(\frac{x}{t}\Big) + \sum_{i=1}^{d} (\partial_i\varphi)\Big(\frac{x}{t}\Big) \frac{x_i}{t} \Big) = \varrho_t(x). \]
Consequently,
\[ (\sigma\ast\varphi_{a})(y) - (\sigma\ast\varphi_{b})(y)
= \int_{a}^{b} (\sigma\ast\varrho_{t})(y) \,\frac{\textup{d}t}{t} \]
for $0<a<b$ and $y\in\mathbb{R}^d$.
Dilating the last identity we end up with a more general one,
\begin{equation}\label{eq:uniformscales}
(\sigma_{\lambda}\ast\varphi_{a\lambda})(y) - (\sigma_{\lambda}\ast\varphi_{b\lambda})(y)
= \int_{a}^{b} (\sigma_{\lambda}\ast\varrho_{t\lambda})(y) \,\frac{\textup{d}t}{t},
\end{equation}
which holds for any $\lambda>0$.
Observe that $\varrho$ is an even $\textup{C}^\infty$ function (because $\varphi$ was chosen to be even) and that it is still supported in $[-3,3]^d$.
In the next section we will also find convenient to define the vector field $v\colon\mathbb{R}^d\to\mathbb{R}^d$ as
\begin{equation}\label{eq:defofvectorv}
v(x) := \varphi(x)x.
\end{equation}
Observe that
\begin{equation}\label{eq:divofvectorv}
\mathop{\textup{div}}v(x) = \sum_{i=1}^{d} \frac{\partial}{\partial x_i} \big(\varphi(x)x_i\big)
= \sum_{i=1}^{d} \big( (\partial_i\varphi)(x)x_i + \varphi(x) \big) = \varrho(x)
\end{equation}
for each $x=(x_1,\ldots,x_d)\in\mathbb{R}^d$, simply by the definition of $\varrho$.
In particular, \eqref{eq:divofvectorv} gives $\widehat{\varrho}(0) =\int_{\mathbb{R}^d}\varrho =0$.

Now we are in position to prove Proposition~\ref{prop:error}.

\begin{proof}[Proof of Proposition~\ref{prop:error}]
Let $\theta$ be a Schwartz function on $\mathbb{R}^d$ such that $\widehat{\theta}$ is nonnegative, radial, supported in the annulus $1/2\leq\|\xi\|_{\ell^2}\leq 1$, and not identically zero. Without further notice it will be understood that all constants are allowed to depend on $\theta$.
From \eqref{eq:ftofk} we see that the function $-\widehat{\theta}\,\widehat{\mathbbm{k}}$ is also nonnegative and radial, so the integral
\[ - \int_{0}^{\infty} (\widehat{\theta_u\ast\mathbbm{k}_u})(\xi) \,\frac{\textup{d}u}{u}
= - \int_{0}^{\infty} \widehat{\theta}(u\xi) \,\widehat{\mathbbm{k}}(u\xi) \,\frac{\textup{d}u}{u} \]
is identically equal to a positive constant, except at the origin $\xi=0$. Multiplying $\theta$ with an appropriate positive factor we can achieve that the last integral is equal to $1$ for each $\xi\in\mathbb{R}^d\setminus\{0\}$.
We use dilates of $-\widehat{\theta}\,\widehat{\mathbbm{k}}$ to construct a continuous-parameter Littlewood--Paley partition of unity,
i.e., for any $f\in\textup{L}^2(\mathbb{R}^d)$ we have
\[ f = \lim_{\substack{v\to0^+\\V\to\infty}} \int_{v}^{V} f \ast \theta_u \ast \mathbbm{k}_u \,\frac{\textup{d}u}{u}, \]
where the convergence holds in the norm of $\textup{L}^2(\mathbb{R}^d)$.
That way we can write
\begin{equation}\label{eq:lpeq2a}
\sigma_{\lambda_j}\ast\varphi_{\varepsilon\lambda_j} - \sigma_{\lambda_j}\ast\varphi_{\lambda_j}
= \lim_{\substack{v\to0^+\\V\to\infty}} \int_{v}^{V} \big( \sigma_{\lambda_j} \ast \varphi_{\varepsilon\lambda_j} - \sigma_{\lambda_j} \ast \varphi_{\lambda_j} \big) \ast \theta_u \ast \mathbbm{k}_u \,\frac{\textup{d}u}{u},
\end{equation}
still in the $\textup{L}^2$-sense.
On the other hand, \eqref{eq:multiscaleHilbert1}, which is the quantity that we need to bound, equals
\[ \sum_{j=1}^{J} \kappa_j \int_{(\mathbb{R}^d)^2} \mathcal{F}_{n-1}(x,y) \,(\sigma_{\lambda_j}\ast\varphi_{\varepsilon\lambda_j} - \sigma_{\lambda_j}\ast\varphi_{\lambda_j})(y) \,\textup{d}y \,\textup{d}x, \]
where we recall that
\begin{equation}\label{eq:lpeq0}
2^{-j}<\lambda_j\leq 2^{-j+1}
\end{equation}
and that $\kappa_j$ are arbitrary complex signs.
Passing to appropriate (sub)sequences in order to obtain a.e.\@ convergence in \eqref{eq:lpeq2a} for each $j$ and using Fatou's lemma, we see that the desired bound for \eqref{eq:multiscaleHilbert1} will follow from the same bound for
\begin{equation}\label{eq:lpeq3a}
\sum_{j=1}^{J} \int_{0}^{\infty} \Big| \int_{(\mathbb{R}^d)^2} \mathcal{F}_{n-1}(x,y) \,\big( (\sigma_{\lambda_j}\ast\varphi_{\varepsilon\lambda_j} - \sigma_{\lambda_j}\ast\varphi_{\lambda_j}) \ast \theta_u \ast \mathbbm{k}_u \big)(y) \,\textup{d}y \,\textup{d}x \Big| \,\frac{\textup{d}u}{u} .
\end{equation}
From \eqref{eq:uniformscales} we get
\begin{equation}\label{eq:lpeq1}
\sigma_{\lambda_j}\ast\varphi_{\varepsilon\lambda_j} - \sigma_{\lambda_j}\ast\varphi_{\lambda_j}
= \int_{\varepsilon}^{1} \sigma_{\lambda_j}\ast\varrho_{t\lambda_j} \,\frac{\textup{d}t}{t}
\end{equation}
for every $j$.
We insert \eqref{eq:lpeq1} into \eqref{eq:lpeq3a}, use the triangle inequality for the integral in $t$, and change the variables $(t,u)\rightarrow(t,tu\lambda_j)$. That way we see that \eqref{eq:lpeq3a} is at most
\begin{equation}\label{eq:lpeq3b}
\sum_{j=1}^{J} \int_{0}^{\infty} \int_{\varepsilon}^{1} \Big| \int_{(\mathbb{R}^d)^2} \mathcal{F}_{n-1}(x,y) \,(\sigma_{\lambda_j} \ast \varrho_{t\lambda_j} \ast \theta_{tu\lambda_j} \ast \mathbbm{k}_{tu\lambda_j})(y) \,\textup{d}y \,\textup{d}x \Big| \,\frac{\textup{d}t}{t} \,\frac{\textup{d}u}{u} .
\end{equation}
The rest of the proof is concerned with bounding \eqref{eq:lpeq3b}.

Now we perform yet another decomposition on \eqref{eq:lpeq3b}, which will allow us to replace $t\lambda_j$ with a single parameter $s$.
Observe that
\begin{equation}\label{eq:lpeq3}
\frac{1}{\log 2} \int_{2^{-j-5}t}^{2^{-j-4}t} \,\frac{\textup{d}s}{s} = 1.
\end{equation}
For $j\in\{1,\ldots,J\}$, $t\in[\varepsilon,1]$, and $s\in[2^{-j-5}t,2^{-j-4}t]$ let us denote
\[ r_j(s,t) := \sqrt{t^2\lambda_j^2-s^2},\quad c_j(s,t) := \frac{t^2\lambda_j^2}{s\sqrt{t^2\lambda_j^2-s^2}}. \]
Assumption \eqref{eq:lpeq0} implies
\begin{equation}\label{eq:lpeq4}
s\sim t \lambda_j, \quad r_j(s,t)\sim t \lambda_j, \quad c_j(s,t)\sim 1
\end{equation}
for $j,t,s$ as above. Using the second convolution identity from \eqref{eq:teleconv0} we get
\begin{equation}\label{eq:lpeq5}
\mathbbm{k}_{tu\lambda_j} = c_j(s,t) \sum_{l=1}^{d} \mathbbm{h}^{(l)}_{r_j(s,t)u} \ast \mathbbm{h}^{(l)}_{su}
\end{equation}
for $l\in\{1,\ldots,d\}$.
We multiply the expression inside the absolute value signs in \eqref{eq:lpeq3b} by \eqref{eq:lpeq3}, interchange the integrals, and also substitute \eqref{eq:lpeq5} for $\mathbbm{k}_{tu\lambda_j}$. Using the triangle inequality for the sum in $l$ and the integral in $s$ and controlling $c_j(s,t)$ by \eqref{eq:lpeq4}, we see that \eqref{eq:lpeq3b} is at most a constant times
\begin{align*}
\sum_{j=1}^J \sum_{l=1}^{d} & \int_{0}^{\infty} \int_{\varepsilon}^{1} \int_{2^{-j-5}t}^{2^{-j-4}t} \Big| \int_{(\mathbb{R}^{d})^2} \mathcal{F}_{n-1}(x,y) \\
& \times \big(\sigma_{\lambda_j} \ast \varrho_{t\lambda_j} \ast \theta_{tu\lambda_j} \ast \mathbbm{h}^{(l)}_{r_j(s,t)u} \ast \mathbbm{h}^{(l)}_{su}\big)(y)
\,\textup{d}y \,\textup{d}x \Big| \,\frac{\textup{d}s}{s} \,\frac{\textup{d}t}{t} \,\frac{\textup{d}u}{u} .
\end{align*}
Expanding out the last convolution we dominate the whole display by
\begin{align}
\sum_{j=1}^J \sum_{l=1}^{d} & \int_{0}^{\infty} \int_{\varepsilon}^{1} \int_{2^{-j-5}t}^{2^{-j-4}t}
\int_{(\mathbb{R}^{d})^2} \Big| \int_{\mathbb{R}^d} \mathcal{F}_{n-1}(x,y) \,\mathbbm{h}^{(l)}_{su}(y-q) \,\textup{d}y \Big| \nonumber \\
& \times \big| (\sigma_{\lambda_j}\ast\varrho_{t\lambda_j}\ast\theta_{tu\lambda_j} \ast \mathbbm{h}^{(l)}_{r_j(s,t)u})(q) \big|
\,\textup{d}q \,\textup{d}x \,\frac{\textup{d}s}{s} \,\frac{\textup{d}t}{t} \,\frac{\textup{d}u}{u}. \label{eq:lpeq7}
\end{align}

We would like to dominate the four-tuple convolution inside the second pair of absolute value signs in \eqref{eq:lpeq7} by a superposition of Gaussians at scale $su$.
First, note that for any positive integers $M$ and $N$ one has
\begin{equation}\label{eq:lpdecay}
\big| \big(\varrho_{u^{-1}}\ast\theta\ast\mathbbm{h}^{(l)}_{r_j(s,t)t^{-1}\lambda_j^{-1}}\big)(q) \big| \lesssim_{d,M,N} \min\{u^M,u^{-1}\} (1 + \|q\|_{\ell^2})^{-N}
\end{equation}
for each $q\in\mathbb{R}^d$.
Indeed, \eqref{eq:lpdecay} can be seen with the aid of the Fourier inversion formula,
\[ \big(\varrho_{u^{-1}}\ast\theta\ast\mathbbm{h}^{(l)}_{r_j(s,t)t^{-1}\lambda_j^{-1}}\big)(q)
= \int_{\mathbb{R}^d} \widehat{\varrho}(u^{-1}\xi) \,\widehat{\theta}(\xi)
\,\widehat{\mathbbm{h}^{(l)}}\Big(\frac{r_j(s,t)}{t\lambda_j}\xi\Big) \,e^{2\pi \mathbbm{i} q\cdot \xi} \,\textup{d}\xi, \]
as follows.
Recalling $\widehat{\varrho}(0)=0$ and using $|\widehat{\varrho}(\xi)| \lesssim \min\{|\xi|^{-M},|\xi|\}$ we get
\[ \big| \big(\varrho_{u^{-1}}\ast\theta\ast\mathbbm{h}^{(l)}_{r_j(s,t)t^{-1}\lambda_j^{-1}}\big)(q) \big| \lesssim_{d,M} \min\{u^M,u^{-1}\}. \]
Also, for any $m\in\{1,\ldots,d\}$, writing $\xi=(\xi_1,\ldots,\xi_d)$ and integrating by parts $N$ times in variable $\xi_m$ we get
\[ \big(\varrho_{u^{-1}}\ast\theta\ast\mathbbm{h}^{(l)}_{r_j(s,t)t^{-1}\lambda_j^{-1}}\big)(q)
= \int_{\mathbb{R}^d} \frac{\partial^N}{\partial\xi_m^N} \bigg( \widehat{\varrho}(u^{-1}\xi) \,\widehat{\theta}(\xi) \,\widehat{\mathbbm{h}^{(l)}}\Big(\frac{r_j(s,t)}{t\lambda_j}\xi\Big) \bigg)
\,(2\pi \mathbbm{i} q_m)^{-N} e^{2\pi \mathbbm{i} q\cdot \xi} \,\textup{d}\xi. \]
If we also control the derivatives of $\widehat{\varrho}(\xi)$ by $\min\{|\xi|^{-M},1\}$ and take \eqref{eq:lpeq4} into account, we obtain
\[ \big| \big(\varrho_{u^{-1}}\ast\theta\ast\mathbbm{h}^{(l)}_{r_j(s,t)t^{-1}\lambda_j^{-1}}\big)(q) \big| \lesssim_{d,M,N} \min\{u^M,u^{-1}\} |q_m|^{-N}. \]
Repeating this for each $m$, we derive \eqref{eq:lpdecay}.
Next, adding one more convolution to \eqref{eq:lpdecay} we obtain, for $t\in[\varepsilon,1]$,
{\allowdisplaybreaks
\begin{align*}
& \big| \big(\sigma_{t^{-1}u^{-1}}\ast\varrho_{u^{-1}}\ast\theta\ast\mathbbm{h}^{(l)}_{r_j(s,t)t^{-1}\lambda_j^{-1}}\big)(q) \big| \\
& \lesssim_{d,M,N} \min\{u^M,u^{-1}\} \int_{\mathbb{R}^d} \Big(1 + \Big\|q-\frac{w}{tu}\Big\|_{\ell^2}\Big)^{-N} \,\textup{d}\sigma(w) \\
& \leq \min\{u^M,u^{-1}\} \,(1 + \|q\|_{\ell^2})^{-N} \int_{\mathbb{R}^d} \Big(1 + \frac{\|w\|_{\ell^2}}{tu}\Big)^{N} \,\textup{d}\sigma(w) \\
& \lesssim_{N,\sigma} \varepsilon^{-N} \min\{u^M,u^{-1}\} \max\{u^{-N},1\} \,(1 + \|q\|_{\ell^2})^{-N} .
\end{align*}
}
For $t$ and $s$ as before, we choose $N=d+3$, $M=N+1$ and rescale in $q$ by $t\lambda_j/s\sim 1$:
\[ \big| \big(\sigma_{s^{-1}u^{-1}\lambda_j}\ast\varrho_{t s^{-1}u^{-1}\lambda_j}\ast\theta_{t s^{-1}\lambda_j}\ast\mathbbm{h}^{(l)}_{r_j(s,t)s^{-1}}\big)(q) \big|
\lesssim_{p,d} \varepsilon^{-d-3} \min\{u,u^{-1}\} \,(1 + \|q\|_{\ell^2})^{-d-3} . \]
We use the Gaussian domination estimate \eqref{eq:Gaussiandom} and perform another rescaling in $q$, this time by $su$, to finally obtain
\begin{equation}\label{eq:lpeq6}
\big| \big(\sigma_{\lambda_j}\ast\varrho_{t\lambda_j}\ast\theta_{tu\lambda_j} \ast \mathbbm{h}^{(l)}_{r_j(s,t)u}\big)(q) \big|
\lesssim_{p,d} \varepsilon^{-d-3} \min\{u,u^{-1}\} \int_{1}^{\infty} \mathbbm{g}_{\beta su}(q) \,\frac{\textup{d}\beta}{\beta^4} .
\end{equation}

Using \eqref{eq:lpeq6}, interchanging the integrals, summing in $j$, and rescaling $s\rightarrow u^{-1}s$, we see that \eqref{eq:lpeq7} is at most a constant multiple of
\begin{align}
\varepsilon^{-d-3} \sum_{l=1}^d \int_1^\infty \int_0^\infty \min\{u,u^{-1}\} \int_{\varepsilon}^{1} \int_{2^{-J-5}tu}^{2^{-5}tu}    \int_{(\mathbb{R}^{d})^2} \Big| \int_{\mathbb{R}^d} \mathcal{F}_{n-1}(x,y) \,\mathbbm{h}^{(l)}_{s}(y-q) \,\textup{d}y \Big| & \nonumber \\
\times \,\mathbbm{g}_{\beta s}(q) \,\textup{d}q \,\textup{d}x \,\frac{\textup{d}s}{s} \,\frac{\textup{d}t}{t} \,\frac{\textup{d}u}{u} \,\frac{\textup{d}\beta}{\beta^4} & . \label{eq:adisplay}
\end{align}
Recall that
\[ \Lambda_{n-1,l,a,b}^{\alpha,\gamma} = \int_{a}^{b} \int_{(\mathbb{R}^d)^{2}} \Big| \int_{\mathbb{R}^d}\mathcal{F}_{n-1}(x,y) \, \mathbbm{h}^{(l)}_{s\gamma}(y-q) \,\textup{d}y \Big| \,\mathbbm{g}_{s\alpha}(q)
\,\textup{d}q \,\textup{d}x \,\frac{\textup{d}s}{s} , \]
by the definition before the statement of Lemma~\ref{lemma:telescoping} specialized to $k=n-1$, and that, by Estimate \eqref{eq:lemmateleineq2}, we have
\begin{equation}\label{eq:adisplay2}
\Lambda_{n-1,l,a,b}^{\alpha,\gamma} \lesssim_{n,d} (\alpha\gamma)^2 \Big(\log\frac{b}{a}\Big)^{1-2^{-n+2}}
\end{equation}
for $\alpha,\gamma\geq 2^{-1/2}$.
We can recognize \eqref{eq:adisplay} as
\[ \varepsilon^{-d-3} \sum_{l=1}^d \int_1^\infty \int_0^\infty \min\{u,u^{-1}\} \int_{\varepsilon}^{1}
\Lambda_{n-1,l,2^{-J-5}tu,2^{-5}tu}^{\beta,1}
\,\frac{\textup{d}t}{t} \,\frac{\textup{d}u}{u} \,\frac{\textup{d}\beta}{\beta^4}, \]
so \eqref{eq:adisplay2} bounds it by a constant times
\begin{align*}
& \varepsilon^{-d-3} \sum_{l=1}^d \int_1^\infty \int_0^\infty \min\{u,u^{-1}\} \int_{\varepsilon}^{1}
\beta^2 J^{1-2^{-n+2}}
\,\frac{\textup{d}t}{t} \,\frac{\textup{d}u}{u} \,\frac{\textup{d}\beta}{\beta^4} \\
& = d \,\varepsilon^{-d-3} \Big(\int_1^\infty \,\frac{\textup{d}\beta}{\beta^2}\Big)
\Big(\int_0^\infty \min\{1,u^{-2}\} \,\textup{d}u\Big) \Big(\int_{\varepsilon}^{1} \,\frac{\textup{d}t}{t}\Big) J^{1-2^{-n+2}} \\
& \lesssim_d \varepsilon^{-d-4} J^{1-2^{-n+2}}.
\end{align*}
This completes the proof. Note that all implicit constants throughout the proof can be pushed to the exponent of $\varepsilon$, yielding the new power $\varepsilon^{-F}$ for some constant $F$.
\end{proof}

\section{The uniform part: proof of Proposition~\ref{prop:uniform}}
\label{sec:uniformpart}
This section follows the corresponding proof scheme by Cook, Magyar, and Pramanik \cite{CMP15:roth}.
We prefer to work out all details, even those that are easy modifications of the results from \cite{CMP15:roth}.
We justify that by the fact that we have a different definition of ``smoothened'' progression counting quantity \eqref{eq:nepsdefinition} and by desire to keep the paper self-contained.

The \emph{uniformity norms}, also known simply as the \emph{$\textup{U}^n$-norms}, were introduced by Gowers \cite{Gow98:4,Gow01:n}. Even though they were initially used for functions defined on the set of integers, their generalizations to locally compact abelian groups also proved to be interesting and useful; see for instance the paper \cite{ET12:gowers} by Eisner and Tao. We are working on the Euclidean space $\mathbb{R}^d$, so we define
\begin{align*}
\|f\|_{\textup{U}^n(\mathbb{R}^d)} & := \Big( \int_{(\mathbb{R}^d)^{n+1}} (\Delta_{h_n} \cdots \Delta_{h_2} \Delta_{h_1} f) (x) \,\textup{d}x \,\textup{d}h_1 \textup{d}h_2 \cdots \textup{d}h_n \Big)^{2^{-n}} \\
& = \Big( \int_{(\mathbb{R}^d)^{n-1}} \Big| \int_{\mathbb{R}^d} (\Delta_{h_{n-1}} \cdots \Delta_{h_1} f) (x) \,\textup{d}x \Big|^2 \,\textup{d}h_1 \cdots \textup{d}h_{n-1} \Big)^{2^{-n}}
\end{align*}
for any measurable function $f\colon\mathbb{R}^d\to\mathbb{C}$.
Here $\Delta_{h}f$ is a function given by
\[ (\Delta_{h}f)(x) := f(x) \overline{f(x+h)} \]
for $h,x\in\mathbb{R}^d$.
Iterated operator $\Delta$ will sometimes be written using the product sign, i.e.,
\[ \Big[ \prod_{j=1}^{n} \Delta_{h_j} \Big] f : = \Delta_{h_1} \Delta_{h_2} \cdots \Delta_{h_n} f. \]
Note that the order of transformations in the above composition is not important since they commute.

Besides comparison of the $\textup{U}^n$-norm to an appropriate $\textup{L}^p$-norm, namely
\begin{equation}\label{eq:gowersLp}
\|f\|_{\textup{U}^n(\mathbb{R}^d)} \leq \|f\|_{\textup{L}^{2^n/(n+1)}(\mathbb{R}^d)},
\end{equation}
we also have its scaling property,
\begin{equation}\label{eq:gowersscaling}
\|f_{\lambda}\|_{\textup{U}^n(\mathbb{R}^d)} = \lambda^{-d(1-(n+1)2^{-n})} \|f\|_{\textup{U}^n(\mathbb{R}^d)}
\end{equation}
for any $\lambda\in(0,\infty)$. These are easy consequences of Young's inequality and a change of variables; see \cite{ET12:gowers} for details. It is also well-known that $\|\cdot\|_{\textup{U}^n(\mathbb{R}^d)}$ satisfy the triangle inequality; see \cite{Gow01:n}.
%When working with an expression $f(x,y)$ depending on variables $x$ and $y$, we write $\|f(x,y)\|_{\textup{U}^n_x(\mathbb{R}^d)}$ for the $\textup{U}^n$-norm of the function $f(\cdot,y)$; that way we emphasize that the $\textup{U}^n$-norm is taken with respect to the variable $x$.

The following lemma is a generalization of \cite[Lemma~4.2]{CMP15:roth} to longer progressions.

\begin{lemma}\label{lemma:CMPinduction}
If $A\subseteq[0,1]^d$ is a measurable set and $g\colon\mathbb{R}^d\to\mathbb{R}$ is a measurable function supported in $[-5\lambda,5\lambda]^d$, then
\[ \bigg|  \int_{(\mathbb{R}^d)^{2}} \Big( \prod_{i=0}^{n-1} \mathbbm{1}_{A}(x+iy)\Big ) g(y) \,\textup{d}y \,\textup{d}x \bigg| \lesssim_{n,d} \lambda^{d(1-(n+1)2^{-n})} \|g\|_{\textup{U}^n(\mathbb{R}^d)}. \]
\end{lemma}

\begin{proof}
The proof is performed by induction. To formulate the inductive statement,
for any $k\in\{1,2,\ldots,n\}$ we define auxiliary expressions
\begin{align*}
\Lambda_k := \int_{(\mathbb{R}^d)^{n-k+2}}
& \bigg( \prod_{i=0}^{k-1}\Big( \Big[ \prod_{m=1}^{n-k}\Delta_{(i+m)h_m} \Big]\mathbbm{1}_{A}(x+iy)\Big) \bigg) \\
& \times \bigg( \Big[ \prod_{m=1}^{n-k}\Delta_{h_m} \Big]g(y) \bigg) \,\textup{d}y \,\textup{d}x \,\textup{d}h_1\cdots \textup{d}h_{n-k} .
\end{align*}
Here we remark that we interpret the product $\prod_{m=1}^{0}$ as the identity operator.
We claim that for any $k\in\{1,2,\ldots,n\}$ we have the estimate
\begin{align}\label{eq:inductive}
|\Lambda_k| \lesssim_{n,d} \lambda^{d\sum_{j=1}^{k}(n-j)2^{-k+j-1}} \|g\|^{2^{n-k}}_{\textup{U}^n(\mathbb{R}^d)}.
\end{align}
Note that the lemma will follow simply by specifying $k=n$ in \eqref{eq:inductive} and using a simple identity
\[ \sum_{j=1}^{n} \frac{n-j}{2^{n-j+1}} = 1 - \frac{n+1}{2^n}. \]

In order to prove the claim \eqref{eq:inductive} we induct on $k$.
For any $k$ we can write
\begin{equation}\label{eq:lemmaCMPsplitting}
\Lambda_k = \int_{(\mathbb{R}^d)^{n-k+1}} L_k \, M_k \,\textup{d}x \,\textup{d}h_1\cdots \textup{d}h_{n-k},
\end{equation}
where $L_k = L_k(x,h_1,\ldots,h_{n-k})$ and $M_k = M_k(x,h_1,\ldots,h_{n-k})$ are expressions given by
\[ L_k := \Big[ \prod_{m=1}^{n-k}\Delta_{mh_m} \Big]\mathbbm{1}_{A}(x) \]
and
\[ M_k := \int_{\mathbb{R}^{d}}  \bigg( \prod_{i=1}^{k-1}\Big( \Big[ \prod_{m=1}^{n-k}\Delta_{(i+m)h_m} \Big]\mathbbm{1}_{A}(x+iy)\Big) \bigg) \bigg( \Big[ \prod_{m=1}^{n-k}\Delta_{h_m} \Big]g(y) \bigg) \,\textup{d}y . \]
Also note that, by the support information of $g$, the integration in $y,h_1,\ldots,h_{n-k}$ in the definition of $\Lambda_k$ can be performed just over balls $\|y\|_{\ell^2}\lesssim_{n,d}\lambda$, $|h_1|\lesssim_{n,d}\lambda$,\ldots, $|h_{n-k}|\lesssim_{n,d}\lambda$.

To handle the induction basis, $k=1$, we note that $M_1$ does not depend on $x$, apply the Cauchy--Schwarz inequality in $h_1,\ldots, h_{n-1}$ to \eqref{eq:lemmaCMPsplitting}, and estimate
\begin{align*}
|\Lambda_1|
& \leq \bigg( \int_{(\mathbb{R}^d)^{n-1}} \Big( \int_{\mathbb{R}^d} L_1 \,\textup{d}x  \Big)^2 \,\textup{d}h_1\cdots \textup{d}h_{n-1} \bigg)^{1/2}
\bigg(  \int_{(\mathbb{R}^d)^{n-1}}  M_1^2 \,\textup{d}h_1\cdots \textup{d}h_{n-1} \bigg)^{1/2} \\
& \lesssim_{n,d} \lambda^{d(n-1)2^{-1}} \|g\|^{2^{n-1}}_{\textup{U}^n(\mathbb{R}^d)}.
\end{align*}
The first factor was controlled using the aforementioned support consideration, while in the second factor we only needed to recall definition of the $\textup{U}^n$-norm.
This is the desired Estimate \eqref{eq:inductive} in the case $k=1$.

Now take $k\in\{2,\ldots,n\}$ and assume that \eqref{eq:inductive} holds for $k-1$ in place of $k$.
Applying the Cauchy--Schwarz in $x,h_1,\ldots, h_{n-k}$ to \eqref{eq:lemmaCMPsplitting} yields
\begin{equation}\label{eq:auxlm}
|\Lambda_k| \leq \mathcal{L}_k^{1/2} \mathcal{M}_k^{1/2},
\end{equation}
where
\begin{align*}
\mathcal{L}_k & := \int_{(\mathbb{R}^d)^{n-k+1}} L_k^2 \,\textup{d}x \,\textup{d}h_1\cdots \textup{d}h_{n-k}, \\
\mathcal{M}_k & := \int_{(\mathbb{R}^d)^{n-k+1}} M_k^2 \,\textup{d}x \,\textup{d}h_1\cdots \textup{d}h_{n-k}.
\end{align*}
Using the same support considerations as before we bound
\begin{equation}\label{eq:auxlk}
\mathcal{L}_k\lesssim_{n,d} \lambda^{d(n-k)}.
\end{equation}
To estimate $\mathcal{M}_k$ we expand the square $M_k^2$, which gives
\begin{align*}
\mathcal{M}_k = \int_{(\mathbb{R}^d)^{n-k+3}}
& \bigg( \prod_{i=1}^{k-1}\Big( \Big[ \prod_{m=1}^{n-k}\Delta_{(i+m)h_m} \Big]\mathbbm{1}_{A}(x+iy)\Big) \bigg) \Big[ \prod_{m=1}^{n-k}\Delta_{h_m} \Big]g(y) \\
& \times \bigg( \prod_{i=1}^{k-1}\Big( \Big [ \prod_{m=1}^{n-k}\Delta_{(i+m)h_m} \Big]\mathbbm{1}_{A}(x+i(y+z)) \Big) \bigg) \Big[ \prod_{m=1}^{n-k}\Delta_{h_m} \Big]g(y+z) \\
& \,\textup{d}x \,\textup{d}y \,\textup{d}z \,\textup{d}h_1\cdots \textup{d}h_{n-k}.
\end{align*}
Changing variables $x\rightarrow x-y$ and shifting $i\rightarrow i+1$ we obtain
\begin{align*}
\int_{(\mathbb{R}^d)^{n-k+3}}
& \bigg( \prod_{i=0}^{k-2}\Big( \Big[ \prod_{m=1}^{n-k}\Delta_{(i+1+m)h_m} \Big]\mathbbm{1}_{A}(x+iy)\Big) \bigg) \Big[ \prod_{m=1}^{n-k}\Delta_{h_m} \Big]g(y) \\
& \times \bigg( \prod_{i=0}^{k-2}\Big( \Big[ \prod_{m=1}^{n-k}\Delta_{(i+1+m)h_m} \Big]\mathbbm{1}_{A}(x+iy+(i+1)z)\Big) \bigg) \Big[ \prod_{m=1}^{n-k}\Delta_{h_m} \Big]g(y+z) \\
& \,\textup{d}x \,\textup{d}y \,\textup{d}z \,\textup{d}h_1\cdots \textup{d}h_{n-k},
\end{align*}
which  further equals, after shifting $m\rightarrow m-1$,
\begin{align*}
\int_{(\mathbb{R}^d)^{n-k+3}}
& \bigg( \prod_{i=0}^{k-2}\Big( \Big[ \prod_{m=2}^{n-k+1}\Delta_{(i+m)h_{m-1}} \Big] \Delta_{(i+1)z} \mathbbm{1}_{A}(x+iy)\Big) \bigg) \\
& \times \Big[ \prod_{m=2}^{n-k+1}\Delta_{h_{m-1}} \Big]\Delta_{z} g(y) \,\textup{d}x \,\textup{d}y \,\textup{d}z \,\textup{d}h_1\cdots \textup{d}h_{n-k}.
\end{align*}
Relabelling
\[ (h_1,h_2,\ldots,h_{n-k}, z)\rightarrow (h_2,h_3,\ldots, h_{n-k+1},h_1) \]
we arrive precisely at quantity $\Lambda_{k-1}$.
By the induction hypothesis we obtain
\begin{equation}\label{eq:auxmk}
\mathcal{M}_k = \Lambda_{k-1} \lesssim_{n,d} \lambda^{d\sum_{j=1}^{k-1}(n-j)2^{-k+j}} \|g\|_{\textup{U}^n(\mathbb{R}^d)}^{2^{n-k+1}}.
\end{equation}
Combining Estimates \eqref{eq:auxlm}, \eqref{eq:auxlk}, \eqref{eq:auxmk} and simplifying we obtain exactly \eqref{eq:inductive}.
This completes the induction and hence it also finishes the proof of Lemma~\ref{lemma:CMPinduction}.
\end{proof}

Lemmata~\ref{lemma:Unnorm1D} and \ref{lemma:Unnormsigma} below could be viewed as modifications and generalizations of \cite[Lemma~4.3]{CMP15:roth} and \cite[Lemma~2.4]{CMP15:roth}, respectively. They handle higher $\textup{U}^n$-norms, since the paper by Cook, Magyar, and Pramanik \cite{CMP15:roth} only needed to deal with the case $n=3$.

\begin{lemma}\label{lemma:Unnorm1D}
Suppose that $p\in[1,\infty)\setminus\{1,2,\ldots,n-1\}$ and recall $D=D(n,p)$ from \eqref{eq:dimthreshhold}. For any $u\in\mathbb{R}$ we have
\begin{equation}\label{eq:Unnormest}
\big\|\mathbbm{1}_{[-3,3]}(x)e^{2\pi\mathbbm{i} u |x|^p}\big\|_{\textup{U}^n_x(\mathbb{R})} \lesssim_{n,p} (1+|u|)^{-2/D}.
\end{equation}
\end{lemma}

\begin{proof}
For $|u|\leq1$ there is nothing to prove: we can even disregard any cancellation coming from the complex exponential by using \eqref{eq:gowersLp}.
Therefore, we assume that $|u|>1$. Fix a number $0<\eta<1$; its value will be chosen later.
Using the fact that the $\textup{U}^n$-norm satisfies the triangle inequality, we can bound the left hand side of \eqref{eq:Unnormest} by a sum of three terms:
\begin{equation}\label{eq:Unterm1}
\big\|\mathbbm{1}_{(-\eta,\eta)}(x)e^{2\pi\mathbbm{i} u |x|^p}\big\|_{\textup{U}^n_x(\mathbb{R})},
\end{equation}
\begin{equation}\label{eq:Unterm2}
\big\|\mathbbm{1}_{[\eta,3]}(x)e^{2\pi\mathbbm{i} u |x|^p}\big\|_{\textup{U}^n_x(\mathbb{R})},
\end{equation}
and
\begin{equation}\label{eq:Unterm3}
\big\|\mathbbm{1}_{[-3,-\eta]}(x)e^{2\pi\mathbbm{i} u |x|^p}\big\|_{\textup{U}^n_x(\mathbb{R})}.
\end{equation}
Term \eqref{eq:Unterm1} is again controlled by inequality \eqref{eq:gowersLp}; it is at most
\[ \|\mathbbm{1}_{(-\eta,\eta)}\|_{\textup{L}^{2^n/(n+1)}(\mathbb{R})} \lesssim \eta^{(n+1)2^{-n}}. \]
In the rest of the proof we will be bounding term \eqref{eq:Unterm2}, while term \eqref{eq:Unterm3} will then clearly satisfy the same estimate.

Using the definition of the $\textup{U}^n$-norm we can expand out the $2^n$-th power of \eqref{eq:Unterm2} as an expression of the form
\[ \int_{\mathbb{R}^{n-1}} \big| \mathcal{I}_{h_1,\ldots,h_{n-1}}(u) \big|^2 \,\textup{d}h_1 \cdots \textup{d}h_{n-1}, \]
where $\mathcal{I}_{h_1,\ldots,h_{n-1}}$ is defined as
\begin{equation}\label{eq:oscilintegral}
\mathcal{I}_{h_1,\ldots,h_{n-1}}(u) := \int_{[a,b]} e^{2\pi\mathbbm{i} u \phi(x)} \,\textup{d}x .
\end{equation}
Here, the limits of integration are given by
\[ a := \max_{(r_1,\ldots,r_{n-1})\in\{0,1\}^{n-1}} (\eta-r_1 h_1-\cdots-r_{n-1} h_{n-1}) \]
and
\[ b := \min_{(r_1,\ldots,r_{n-1})\in\{0,1\}^{n-1}} (3-r_1 h_1-\cdots-r_{n-1} h_{n-1}), \]
while the phase function is given as
\begin{equation}\label{eq:phasefn}
\phi(x) := \sum_{(r_1,\ldots,r_{n-1})\in\{0,1\}^{n-1}} (-1)^{r_1+\cdots+r_{n-1}} |x+r_1 h_1+\cdots+r_{n-1} h_{n-1}|^p.
\end{equation}
We will treat \eqref{eq:oscilintegral} as an oscillatory integral, see \cite[Chapter~VIII]{St93:book}, but we will not need any advanced results from the literature.
Note that $a\geq\eta$ and $b\leq 3$. We interpret \eqref{eq:oscilintegral} as zero if $a\geq b$. In order for \eqref{eq:oscilintegral} to be nonzero there must also be an $x\in[\eta,3]$ such that $x+h_i\in[\eta,3]$ for $i=1,2,\ldots,n-1$, which implies
\begin{equation}\label{eq:suppconditandh}
|h_i| \leq 3
\end{equation}
for each index $i$. We proceed by estimating \eqref{eq:oscilintegral} for each individual choice of $h_1,\ldots,h_{n-1}$ satisfying \eqref{eq:suppconditandh}.
We distinguish two cases.

\emph{Case 1.} $|h_i|\leq\eta$ for at least one $i\in\{1,\ldots,n-1\}$.
In this case we disregard any oscillation and simply bound
\begin{equation}\label{eq:estforICase1}
| \mathcal{I}_{h_1,\ldots,h_{n-1}}(u) | \leq |[a,b]| \leq 3.
\end{equation}

\emph{Case 2.} $|h_i|>\eta$ for each $i\in\{1,\ldots,n-1\}$.
For $x>a-\eta$ all sums inside the absolute values in Formula \eqref{eq:phasefn} are strictly positive, so multiple application of the fundamental theorem of calculus allows us to write
\begin{align*}
\phi(x) = & (-1)^{n-1} p(p-1)\cdots(p-n+2) \, h_1 \cdots h_{n-1} \\
& \times \int_{[0,1]^{n-1}} (x+t_1 h_1+\cdots+t_{n-1} h_{n-1})^{p-n+1} \,\textup{d}t_1 \cdots \textup{d}t_{n-1}.
\end{align*}
Consequently,
\begin{align*}
\phi'(x) = & (-1)^{n-1} p(p-1)\cdots(p-n+2)(p-n+1) \, h_1 \cdots h_{n-1} \\
& \times \int_{[0,1]^{n-1}} (x+t_1 h_1+\cdots+t_{n-1} h_{n-1})^{p-n} \,\textup{d}t_1 \cdots \textup{d}t_{n-1}
\end{align*}
and
\begin{align*}
\phi''(x) = & (-1)^{n-1} p(p-1)\cdots(p-n+2)(p-n+1)(p-n) \, h_1 \cdots h_{n-1} \\
& \times \int_{[0,1]^{n-1}} (x+t_1 h_1+\cdots+t_{n-1} h_{n-1})^{p-n-1} \,\textup{d}t_1 \cdots \textup{d}t_{n-1}.
\end{align*}
Thus, for $p\in[1,\infty)$ other than $1,2,\ldots,n-1$ and for $a\leq x\leq b$ we have
\[ |\phi'(x)| \gtrsim_{n,p} \eta^{n-1} \min\{\eta^{p-n},3^{p-n}\} \gtrsim_{n,p} \min\{\eta^{p-1},\eta^{n-1}\} \]
and
\[ |\phi''(x)| \lesssim_{n,p} \max\{\eta^{p-n-1},1\}. \]
Integrating \eqref{eq:oscilintegral} by parts as long as $a<b$, we obtain
\begin{align*}
\mathcal{I}_{h_1,\ldots,h_{n-1}}(u)
& = \frac{1}{2\pi\mathbbm{i}u} \int_{a}^{b} \frac{1}{\phi'(x)} \Big( \frac{\partial}{\partial x} e^{2\pi\mathbbm{i} u \phi(x)} \Big) \,\textup{d}x \\
& = \frac{1}{2\pi\mathbbm{i}u} \bigg( \frac{e^{2\pi\mathbbm{i} u \phi(b)}}{\phi'(b)} - \frac{e^{2\pi\mathbbm{i} u \phi(a)}}{\phi'(a)}
+ \int_{a}^{b} \frac{\phi''(x)}{\phi'(x)^2} e^{2\pi\mathbbm{i} u \phi(x)} \,\textup{d}x \bigg).
\end{align*}
The above bounds for $\phi'$ and $\phi''$ give
\begin{equation}\label{eq:estforICase2}
| \mathcal{I}_{h_1,\ldots,h_{n-1}}(u) | \lesssim_{n,p} \frac{\max\{\eta^{p-n-1},1\}}{|u|\min\{\eta^{2p-2},\eta^{2n-2}\}}
\leq \eta^{-4(n+p)+1}|u|^{-1}.
\end{equation}

Combining Estimates \eqref{eq:estforICase1} and \eqref{eq:estforICase2} from each of the two cases, squaring, and integrating over all $h_1,\ldots,h_{n-1}$ that satisfy \eqref{eq:suppconditandh}, we see that \eqref{eq:Unterm2} is at most a constant (depending on $n$ and $p$) times
\[ \big( \eta + \eta^{-8(n+p)+1} |u|^{-2} \big)^{2^{-n}}. \]

Now, when we have estimated each of the terms \eqref{eq:Unterm1}--\eqref{eq:Unterm3}, we can finally choose $\eta=|u|^{-1/4(n+p)}$ and conclude that \eqref{eq:Unnormest} holds.
\end{proof}

Recall that we have chosen $\varphi$ in Section~\ref{sec:mainproof} and defined $\varrho$ by \eqref{eq:defofrho} in Section~\ref{sec:errorpart}.
Also recall that a function $\psi$ was fixed in Section~\ref{sec:mainproof} and used to define measures $\sigma^\eta$ via Formula \eqref{eq:measuresigmaeta}. These measures are absolutely continuous with respect to the Lebesgue measure on $\mathbb{R}^d$, so, by a slight abuse of notation, we also write $\sigma^{\eta}$ for their densities,
\[ \sigma^{\eta}(x)=\psi_{\eta}(\|x\|_{\ell^p}^p-1). \]

\begin{lemma}\label{lemma:Unnormsigma}
For $d\geq D=D(n,p)$ and $0<\eta<t<1$ we have
\[ \| \sigma^{\eta}\ast\varrho_{t} \|_{\textup{U}^n(\mathbb{R}^d)} \lesssim_{n,p,d} t^{1/3}. \]
\end{lemma}

\begin{proof}
A few times in the proof we will find the following inequality useful,
\begin{equation}\label{eq:Gowersinttriangle}
\|f\ast g\|_{\textup{U}^n(\mathbb{R}^d)} \leq \|f\|_{\textup{U}^n(\mathbb{R}^d)} \|g\|_{\textup{L}^1(\mathbb{R}^d)},
\end{equation}
and we will only need it for continuous compactly supported functions $f,g\colon\mathbb{R}^d\to\mathbb{C}$.
It easily follows from the triangle inequality for the Gowers norms, by writing the convolution as an integral and then the integral as a limit of its Riemann sums.
We split
\begin{equation}\label{eq:Unnormsplitting}
\sigma^{\eta}\ast\varrho_{t} = \sigma^{\tau}\ast\varrho_{t} + (\sigma^{\eta}-\sigma^{\tau})\ast\varrho_{t},
\end{equation}
where $\tau$ satisfying $t<\tau<1$ will be chosen later.

In order to handle the first term in \eqref{eq:Unnormsplitting} we write the vector field $v$ from \eqref{eq:defofvectorv} in Cartesian coordinates as $v=(v^{(1)},\ldots,v^{(d)})$ and scale Equation \eqref{eq:divofvectorv} by $t$, to obtain
\[ \varrho_t = t \sum_{i=1}^{d} \partial_i v^{(i)}_t. \]
Using this, \eqref{eq:Gowersinttriangle}, and \eqref{eq:gowersLp} we can estimate
\begin{align*}
\|\sigma^{\tau}\ast\varrho_{t}\|_{\textup{U}^n(\mathbb{R}^d)}
& \leq t \sum_{i=1}^{d} \big\|\sigma^{\tau}\ast\partial_i v^{(i)}_t\big\|_{\textup{U}^n(\mathbb{R}^d)}
= t \sum_{i=1}^{d} \big\|\partial_i \sigma^{\tau}\ast v^{(i)}_t\big\|_{\textup{U}^n(\mathbb{R}^d)} \\
& \leq t \sum_{i=1}^{d} \|\partial_i \sigma^{\tau}\|_{\textup{U}^n(\mathbb{R}^d)} \big\|v^{(i)}_t\big\|_{\textup{L}^{1}(\mathbb{R}^d)}
\lesssim t \sum_{i=1}^{d} \|\partial_i \sigma^{\tau}\|_{\textup{L}^{2^n/(n+1)}(\mathbb{R}^d)}.
\end{align*}
Using
\[ |(\partial_i \sigma^{\tau})(x)|
= \Big| \psi'\Big(\frac{|x_1|^p+\cdots+|x_d|^p-1}{\tau}\Big) \Big| \ \frac{p|x_i|^{p-1}}{\tau^2} \]
and a very crude observation
\begin{equation}\label{eq:sigmasupport}
\mathop{\textup{supp}}\sigma^{\tau} \subseteq [-2,2]^d,
\end{equation}
we finally deduce
\begin{equation}\label{eq:Unsplitterm1}
\|\sigma^{\tau}\ast\varrho_{t}\|_{\textup{U}^n(\mathbb{R}^d)} \lesssim_{p,d} \tau^{-2}t.
\end{equation}

Now we turn to the second term in \eqref{eq:Unnormsplitting}. Using \eqref{eq:Gowersinttriangle} again,
\begin{equation}\label{eq:Unsplitterm2a}
\big\|(\sigma^{\eta}-\sigma^{\tau})\ast\varrho_{t}\big\|_{\textup{U}^n(\mathbb{R}^d)}
\lesssim \|\sigma^{\eta}-\sigma^{\tau}\|_{\textup{U}^n(\mathbb{R}^d)},
\end{equation}
we are lead to estimate the $\textup{U}^n$-norm of $\sigma^{\eta}-\sigma^{\tau}$.
Using the Fourier inversion formula we get
\[ (\sigma^{\eta}-\sigma^{\tau})(x) = \int_{\mathbb{R}} \big( \widehat{\psi}(\eta u) - \widehat{\psi}(\tau u) \big)
e^{2\pi\mathbbm{i} u (\|x\|_{\ell^p}^{p}-1)} \,\textup{d}u. \]
Due to the support observation \eqref{eq:sigmasupport}, multiplication of the last equality by
\[ \mathbbm{1}_{[-3,3]}(x_1) \cdots \mathbbm{1}_{[-3,3]}(x_d) \]
gives
\[ (\sigma^{\eta}-\sigma^{\tau})(x) = \int_{\mathbb{R}} \big( \widehat{\psi}(\eta u) - \widehat{\psi}(\tau u) \big)
e^{-2\pi\mathbbm{i} u} \Big( \prod_{i=1}^{d} \mathbbm{1}_{[-3,3]}(x_i) e^{2\pi\mathbbm{i} u |x_i|^p} \Big) \,\textup{d}u, \]
so the triangle inequality for the Gowers norm yields
\[ \| \sigma^{\eta}-\sigma^{\tau} \|_{\textup{U}^n(\mathbb{R}^d)}
\leq \int_{\mathbb{R}} \big| \widehat{\psi}(\eta u) - \widehat{\psi}(\tau u) \big|
\big\| \mathbbm{1}_{[-3,3]}(x) e^{2\pi\mathbbm{i} u |x|^p} \big\|_{\textup{U}^{n}_{x}(\mathbb{R})}^d \,\textup{d}u. \]
Lemma~\ref{lemma:Unnorm1D} applies, giving us
\[ \| \sigma^{\eta}-\sigma^{\tau} \|_{\textup{U}^n(\mathbb{R}^d)}
\lesssim_{n,p} \int_{\mathbb{R}} \big| \widehat{\psi}(\eta u) - \widehat{\psi}(\tau u) \big| (1+|u|)^{-2d/D} \,\textup{d}u. \]
The last integral is estimated by splitting it into regions $\{|u|<1\}$, $\{1\leq|u|<1/\tau\}$, and $\{|u|\geq1/\tau\}$, which gives
\begin{equation}\label{eq:Unsplitterm2}
\| \sigma^{\eta}-\sigma^{\tau} \|_{\textup{U}^n(\mathbb{R}^d)} \lesssim_{n,p,d} \tau + \tau^{2d/D-1} \lesssim \tau,
\end{equation}
as soon as $d\geq D$.

Recall the splitting \eqref{eq:Unnormsplitting} and combine \eqref{eq:Unsplitterm1}, \eqref{eq:Unsplitterm2a}, and \eqref{eq:Unsplitterm2}. That way we finally obtain
\[ \| \sigma^{\eta}\ast\varrho_{t} \|_{\textup{U}^n(\mathbb{R}^d)} \lesssim_{n,p,d} \tau^{-2}t + \tau, \]
so it remains to choose $\tau=t^{1/3}$.
\end{proof}

We are finally in position to prove Proposition~\ref{prop:uniform}.

\begin{proof}[Proof of Proposition~\ref{prop:uniform}]
We are working in dimensions $d\geq D(n,p)$ in order to be able to use Lemma~\ref{lemma:Unnormsigma}.
Let us fix numbers $\lambda,\varepsilon\in(0,1]$.
Using property \eqref{eq:nconvergence}, shorthand notation \eqref{eq:Fdefinedasaproduct} and \eqref{eq:Fdefinedasaproductnewq}, and Formula \eqref{eq:uniformscales}, we can write
\begin{align*}
\mathcal{N}^{0}_{\lambda}(A) - \mathcal{N}^{\varepsilon}_{\lambda}(A)
& = \lim_{\vartheta\to0^+} \big( \mathcal{N}^{\vartheta}_{\lambda}(A) - \mathcal{N}^{\varepsilon}_{\lambda}(A) \big) \\
& = \lim_{\vartheta\to0^+} \int_{\vartheta}^{\varepsilon} \int_{\mathbb{R}^d} (f\ast\varrho_{t\lambda})(z) \,\textup{d}\sigma_{\lambda}(z) \,\frac{\textup{d}t}{t}.
\end{align*}
Then we apply \eqref{eq:vagueconvergence} to get
\begin{align}
\mathcal{N}^{0}_{\lambda}(A) - \mathcal{N}^{\varepsilon}_{\lambda}(A)
& = \lim_{\vartheta\to0^+} \int_{\vartheta}^{\varepsilon} \Big( \lim_{\eta\to0^+} \int_{\mathbb{R}^d} (f\ast\varrho_{t\lambda})(z) \,\textup{d}\sigma_{\lambda}^{\eta}(z) \Big) \frac{\textup{d}t}{t} \nonumber \\
& = \lim_{\vartheta\to0^+} \int_{\vartheta}^{\varepsilon} \Big( \lim_{\eta\to0^+} \int_{(\mathbb{R}^d)^2} \mathcal{F}(x,y) (\sigma_{\lambda}^{\eta}\ast\varrho_{t\lambda})(y) \,\textup{d}y \,\textup{d}x \Big) \frac{\textup{d}t}{t}. \label{eq:Ndifferencelimit}
\end{align}
From the scaling property of the Gowers norms \eqref{eq:gowersscaling} and Lemma~\ref{lemma:Unnormsigma} we know that
\[ \|\sigma_{\lambda}^{\eta}\ast\varrho_{t\lambda}\|_{\textup{U}^n(\mathbb{R}^d)}
= \|(\sigma^{\eta}\ast\varrho_{t})_{\lambda}\|_{\textup{U}^n(\mathbb{R}^d)}
\lesssim_{n,p,d} \lambda^{-d(1-(n+1)2^{-n})} t^{1/3} \]
for $0<\eta<t<1$. Combining this with an application of Lemma~\ref{lemma:CMPinduction} for $g = \sigma_{\lambda}^{\eta}\ast\varrho_{t\lambda}$, we get
\[ \Big| \int_{(\mathbb{R}^d)^2} \mathcal{F}(x,y) (\sigma_{\lambda}^{\eta}\ast\varrho_{t\lambda})(y) \,\textup{d}y \,\textup{d}x \Big|
\lesssim_{n,p,d} t^{1/3}. \]
Observe that the right hand side of the last estimate no longer depends on $\lambda$, since the factor coming from scaling and the factor coming from Lemma~\ref{lemma:CMPinduction} cancelled each other. Also, this estimate is uniform in $\eta\in(0,t)$ for each fixed $t$.
Taking the limit as $\eta\to0^+$ and using \eqref{eq:Ndifferencelimit}, we can finally control the difference between the two progression-counting expressions as
\[ \big| \mathcal{N}^{0}_{\lambda}(A) - \mathcal{N}^{\varepsilon}_{\lambda}(A) \big|
\lesssim_{n,p,d} \int_{0}^{\varepsilon} \frac{\textup{d}t}{t^{2/3}} \lesssim \varepsilon^{1/3}. \]
This is precisely what we needed to prove.
\end{proof}

\section{Proof of Theorem~\ref{thm:quantifiedlm}}
\label{sec:quantifiedlm}
This section establishes Theorem~\ref{thm:quantifiedlm} from the introductory section. We will be following the same steps as in the proof of Theorem~\ref{thm:maintheorem}, which spanned over Sections~\ref{sec:mainproof}--\ref{sec:uniformpart}, but each step will collapse into a several times shorter argument. For instance, there will be no need for the whole structural induction in the proof of Lemma~\ref{lemma:telescoping}; in order to control the error part we will only have to perform a computation very similar to the induction basis. Also, the oscillatory control of the uniform part done in Lemmata~\ref{lemma:Unnorm1D} and \ref{lemma:Unnormsigma} will be replaced merely by the well-known decay of the Fourier transform of the circle measure. Apart from being quite short, the proof presented in this section is essentially self-contained: it will only use basic properties of the Gaussian functions \eqref{eq:ftofgaussians0}--\eqref{eq:heatequation}, the Gaussian domination trick \eqref{eq:Gaussiandom}, and the aforementioned decay, stated in \eqref{eq:FTofsigma} below. This gives us a great opportunity to illustrate the largeness--smoothness multiscale approach on a simpler problem.

In this section $\sigma$ will denote the arc-length measure of the standard unit circle $\mathbb{S}^1\subseteq\mathbb{R}^2$, i.e., the $1$-dimensional spherical measure in the plane. Normalization of $\sigma$ is not very important, but we can choose it so that $\sigma$ is a probability measure, i.e., $\sigma(\mathbb{S}^1)=1$. Its Fourier transform satisfies
\begin{equation}\label{eq:FTofsigma}
\big|\widehat{\sigma}(\xi)\big| \lesssim (1+\|\xi\|_{\ell^2})^{-1/2}
\end{equation}
for $\xi\in\mathbb{R}^2$. Indeed, it is well-known that this Fourier transform is $2\pi J_0(2\pi\|\xi\|_{\ell^2})$, where $J_\alpha$ are the Bessel functions of the first kind and their asymptotic properties can be found, for instance, in \cite{AS92:hmf}.

This time the pattern-counting quantities are defined, for a measurable set $A\subseteq([0,1]^2)^n$, as
\begin{align*}
\mathcal{N}^{0}_{\lambda}(A) := \int_{(\mathbb{R}^2)^{2n}}
& \prod_{(r_1,\ldots,r_n)\in\{0,1\}^n} \mathbbm{1}_{A}(x_1 + r_1 y_1, x_2 + r_2 y_2, \ldots, x_n + r_n y_n) \\
& \,\textup{d}\sigma_{\lambda}(y_1)\cdots\textup{d}\sigma_{\lambda}(y_n) \,\textup{d}x_1 \cdots \textup{d}x_n
\end{align*}
and
\begin{align*}
\mathcal{N}^{\varepsilon}_{\lambda}(A) := \int_{(\mathbb{R}^2)^{2n}}
& \prod_{(r_1,\ldots,r_n)\in\{0,1\}^n} \mathbbm{1}_{A}(x_1 + r_1 y_1, x_2 + r_2 y_2, \ldots, x_n + r_n y_n) \\
& \times (\sigma_{\lambda}\ast\mathbbm{g}_{\varepsilon\lambda})(y_1)\cdots(\sigma_{\lambda}\ast\mathbbm{g}_{\varepsilon\lambda})(y_n)
\,\textup{d}y_1\cdots\textup{d}y_n \,\textup{d}x_1 \cdots \textup{d}x_n
\end{align*}
for $\lambda,\varepsilon\in(0,1]$. Here $\mathbbm{g}$ stands for the standard Gaussian function \eqref{eq:standardgaussian} on $\mathbb{R}^2$.
One can rewrite
\[ \mathcal{N}^{\varepsilon}_{\lambda}(A) = \int_{(\mathbb{R}^2)^n} (f\ast\mathbbm{G}_{\varepsilon\lambda})(y_1,\ldots,y_n) \,\textup{d}\sigma_{\lambda}(y_1) \cdots \textup{d}\sigma_{\lambda}(y_n), \]
where
\[ f(y_1,\ldots,y_n) := \int_{([0,1]^2)^n} \prod_{(r_1,\ldots,r_n)\in\{0,1\}^n} \mathbbm{1}_{A}(x_1 + r_1 y_1, \ldots, x_n + r_n y_n) \,\textup{d}x_1 \cdots \textup{d}x_n \]
and $\mathbbm{G}$ now denotes a $2n$-dimensional standard Gaussian. As in Section~\ref{sec:mainproof}, it is easy to argue that \eqref{eq:nconvergence} holds by an application of the dominated convergence theorem.
Changing variables $y_i$ to $x_i^{1}=x_i+y_i$ and writing $x_i^{0}$ in place of $x_i$ we can rewrite $\mathcal{N}^{\varepsilon}_{\lambda}(A)$ even more conveniently as
\begin{align}
\mathcal{N}^{\varepsilon}_{\lambda}(A) = \int_{(\mathbb{R}^2)^{2n}}
& \prod_{(r_1,\ldots,r_n)\in\{0,1\}^n} \mathbbm{1}_{A}(x_1^{r_1}, x_2^{r_2}, \ldots, x_n^{r_n}) \nonumber \\
& \times (\sigma_{\lambda}\ast\mathbbm{g}_{\varepsilon\lambda})(x_1^0-x_1^1)\cdots(\sigma_{\lambda}\ast\mathbbm{g}_{\varepsilon\lambda})(x_n^0-x_n^1)
\,\textup{d}x_1^0 \,\textup{d}x_1^1 \cdots \textup{d}x_n^0 \,\textup{d}x_n^1. \label{eq:nepslm}
\end{align}
The above $0$s and $1$s are upper indices and they should not be confused with powers.
The reader can notice that there is no arithmetic structure left in \eqref{eq:nepslm}, except for the differences of two variables appearing as arguments of $\sigma_{\lambda}\ast\mathbbm{g}_{\varepsilon\lambda}$.

For the proof of Theorem~\ref{thm:quantifiedlm} it is sufficient to show the following three estimates, which will be established for every measurable set $A\subseteq([0,1]^2)^d$ with $|A|\geq \delta$, every $\delta,\lambda,\varepsilon\in(0,1]$, every positive integer $J$, and every choice of numbers $\lambda_j$; $j=1,2,\ldots,J$ from $(0,1]$ satisfying $\lambda_j\in(2^{-j},2^{-j+1}]$ for each index $j$. The desired estimates are
\begin{equation}\label{eq:proofoflm1}
\mathcal{N}^{1}_{\lambda}(A) \gtrsim_n \delta^{2^n},
\end{equation}
\begin{equation}\label{eq:proofoflm2}
\sum_{j=1}^{J} \big|\mathcal{N}^{\varepsilon}_{\lambda_j}(A)-\mathcal{N}^{1}_{\lambda_j}(A)\big| \lesssim_{n} \varepsilon^{-5n-1},
\end{equation}
and
\begin{equation}\label{eq:proofoflm3}
\big|\mathcal{N}^{0}_{\lambda}(A)-\mathcal{N}^{\varepsilon}_{\lambda}(A)\big| \lesssim_{n} \varepsilon^{1/2}.
\end{equation}
They are counterparts of Propositions~\ref{prop:structured}--\ref{prop:uniform}.
Once these estimates are shown, the proof of Theorem~\ref{thm:quantifiedlm} concludes as follows.
Take a set $A\subseteq([0,1]^2)^n$ with measure $|A|\geq \delta$.
Let $c_1\in(0,1]$, $C_2\in[1,\infty)$, $C_3\in[1,\infty)$, respectively, be the implied constants in \eqref{eq:proofoflm1}, \eqref{eq:proofoflm2}, \eqref{eq:proofoflm3}; they all depend on $n$.
Choose
\[ \varepsilon := \Big(\frac{c_1 \delta^{2^n}}{3C_3}\Big)^2, \quad J := \Big\lfloor \frac{3 C_2}{c_1 \varepsilon^{5n+1}\delta^{2^n}} \Big\rfloor + 1, \]
so that
\[ J\lesssim_n \delta^{-(10n+3)2^n}. \]
Consider the intervals $(2^{-j},2^{-j+1}]$ for $j=1,2,\ldots,J$. Estimate \eqref{eq:proofoflm2} enables us to choose an index $j\in\{1,\ldots,J\}$ such that each $\lambda\in(2^{-j},2^{-j+1}]$ satisfies
\[ \big|\mathcal{N}^{\varepsilon}_{\lambda}(A)-\mathcal{N}^{1}_{\lambda}(A)\big| \leq C_2 \varepsilon^{-5n-1} J^{-1} \leq \frac{1}{3}c_1\delta^{2^n}. \]
Moreover, $\varepsilon>0$ was chosen so that
\[ \big|\mathcal{N}^{0}_{\lambda}(A)-\mathcal{N}^{\varepsilon}_{\lambda}(A)\big| \leq \frac{1}{3}c_1\delta^{2^n}. \]
Recalling the splitting \eqref{eq:splitting}, one can conclude that $\mathcal{N}^{0}_{\lambda}(A)>0$ for each $\lambda\in(2^{-j},2^{-j+1}]$ and the last interval has desired length.

\subsection{The structured part: proof of \eqref{eq:proofoflm1}}
If $Q_1,Q_2,\ldots,Q_n\subseteq\mathbb{R}^2$ are arbitrary measurable sets with positive finite Lebesgue measures, then for any measurable set $B\subseteq(\mathbb{R}^2)^n$ we have
\begin{align}
\fint_{Q_1\times Q_1\times\cdots\times Q_n\times Q_n}
\prod_{(r_1,\ldots,r_n)\in\{0,1\}^n} \mathbbm{1}_{B}(x_1^{r_1}, \ldots, x_n^{r_n})
\,\textup{d}x_1^0 \,\textup{d}x_1^1 \cdots \textup{d}x_n^0 \,\textup{d}x_n^1 \nonumber & \\
\geq \bigg( \fint_{Q_1\times\cdots\times Q_n} \mathbbm{1}_{B}(x_1, \ldots, x_n)
\,\textup{d}x_1 \cdots \textup{d}x_n \bigg)^{2^n} & . \label{eq:boxnorm}
\end{align}
Inequality \eqref{eq:boxnorm} is easily shown by the induction on the positive integer $n$. The induction basis is trivial, since \eqref{eq:boxnorm} becomes an equality. For the induction step we rewrite the left hand side of \eqref{eq:boxnorm} as
\begin{align*}
\fint_{Q_1\times Q_1\times\cdots\times Q_{n-1}\times Q_{n-1}}
\bigg( \fint_{Q_n} \prod_{(r_1,\ldots,r_{n-1})\in\{0,1\}^{n-1}} \mathbbm{1}_{B}(x_1^{r_1}, \ldots, x_{n-1}^{r_{n-1}}, x_n) \,\textup{d}x_n \bigg)^2 & \\
\textup{d}x_1^0 \,\textup{d}x_1^1 \cdots \textup{d}x_{n-1}^0 \,\textup{d}x_{n-1}^1 &
\end{align*}
and use the Cauchy--Schwarz inequality in the variables $x_1^0,x_1^1,\ldots,x_{n-1}^0,x_{n-1}^1$ to bound it from below by
\begin{align*}
\bigg( \fint_{Q_n} \fint_{Q_1\times Q_1\times\cdots\times Q_{n-1}\times Q_{n-1}}
\prod_{(r_1,\ldots,r_{n-1})\in\{0,1\}^{n-1}} \mathbbm{1}_{B}(x_1^{r_1}, \ldots, x_{n-1}^{r_{n-1}}, x_n) & \\
\textup{d}x_1^0 \,\textup{d}x_1^1 \cdots \textup{d}x_{n-1}^0 \,\textup{d}x_{n-1}^1 \,\textup{d}x_n & \bigg)^2 .
\end{align*}
Then we apply the induction hypothesis to the set
\[ \{ (x_1,\ldots,x_{n-1})\in(\mathbb{R}^2)^{n-1} : (x_1,\ldots,x_{n-1},x_n)\in B \} \]
for each fixed $x_n\in Q_n$, which completes the proof.
Alternatively, inequality \eqref{eq:boxnorm} can be viewed as a particular case of the Cauchy--Schwarz inequality for $n$-dimensional variants of the so-called box-inner products and box-norms; see \cite{GT08:primes,S06:corners,Tao07:exp}.

Let $m$ be the unique positive integer such that $\lambda\in(2^{-m},2^{-m+1}]$.
For every $y\in[-1,1]^2$ we have $(\sigma\ast\mathbbm{g})(y)\geq e^{-8\pi}$, which implies
\begin{equation}\label{eq:circlelower}
\sigma_{\lambda}\ast\mathbbm{g}_{\lambda} \gtrsim 2^{2m} \mathbbm{1}_{[-2^{-m},2^{-m}]^2}.
\end{equation}
Turning back to \eqref{eq:nepslm}, we observe that $\mathcal{N}^{1}_{\lambda}(A)$ can only decrease if we partition each unit square $[0,1]^2$ into the collection $\mathcal{Q}_m$ of $2^{2m}$ congruent squares of sidelength $2^{-m}$ and restrict the domain of integration by imposing that each pair of variables $x_i^0$, $x_i^1$ has to lie in the same square from $\mathcal{Q}_m$. If we also use \eqref{eq:circlelower}, then we obtain
\begin{align*}
\mathcal{N}^{1}_{\lambda}(A) \gtrsim_n (2^{2m})^n \sum_{Q_1,\ldots,Q_n\in\mathcal{Q}_m}
\int_{Q_1\times Q_1\times\cdots\times Q_n\times Q_n}
\prod_{(r_1,\ldots,r_n)\in\{0,1\}^n} \mathbbm{1}_{A}(x_1^{r_1}, \ldots, x_n^{r_n}) & \\
\,\textup{d}x_1^0 \,\textup{d}x_1^1 \cdots \textup{d}x_n^0 \,\textup{d}x_n^1 & .
\end{align*}
Applying \eqref{eq:boxnorm} for each choice of $Q_1,\ldots,Q_n$, we see that the last expression is at least
\[ 2^{-2mn} \sum_{Q_1,\ldots,Q_n\in\mathcal{Q}_m} \Big(\fint_{Q_1\times\ldots\times Q_n} \mathbbm{1}_A\Big)^{2^n}. \]
Using Jensen's inequality for the power function and the average of $2^{2mn}$ real numbers, we arrive at
\[ \mathcal{N}^{1}_{\lambda}(A) \gtrsim_n \Big(\int_{([0,1]^2)^n} \mathbbm{1}_A \Big)^{2^n} \geq \delta^{2^n}. \]

\subsection{The error part: proof of \eqref{eq:proofoflm2}}
Let us denote
\[ \mathcal{F} := \prod_{(r_1,\ldots,r_n)\in\{0,1\}^n} \mathbbm{1}_{A}(x_1^{r_1}, \ldots, x_n^{r_n}), \]
so that this is a function of variables $x_1^0,x_1^1,\ldots,x_n^0,x_n^1\in\mathbb{R}^2$.
For convenience we also write
\[ \mathcal{F}'(x) := \prod_{(r_2,\ldots,r_n)\in\{0,1\}^{n-1}} \mathbbm{1}_{A}(x, x_2^{r_2}, \ldots, x_{n}^{r_{n}}) \]
for $x\in\mathbb{R}^2$, keeping in mind that $\mathcal{F}'(x)$ also depends on $x_2^0,x_2^1,\ldots,x_{n}^0,x_{n}^1$.
This time we define
\[ \Theta_{m,a,b}^{\alpha_1,\ldots,\alpha_n} := - \int_{a}^{b} \int_{(\mathbb{R}^2)^{2n}} \mathcal{F}
\ \mathbbm{k}_{s\alpha_m}(x_m^0-x_m^1) \Big(\prod_{\substack{1\leq i\leq n\\ i\neq m}}\mathbbm{g}_{s\alpha_i}(x_i^0-x_i^1)\Big)
\,\textup{d}x_1^0 \,\textup{d}x_1^1 \cdots \textup{d}x_n^0 \,\textup{d}x_n^1 \,\frac{\textup{d}s}{s} \]
for $m\in\{1,2,\ldots,n\}$, $0<a<b<\infty$, and $\alpha_1,\ldots,\alpha_n\in(0,\infty)$, and
\[ \Xi_{s}^{\alpha_1,\ldots,\alpha_n} := \int_{(\mathbb{R}^2)^{2n}} \mathcal{F}
\ \Big(\prod_{i=1}^{n}\mathbbm{g}_{s\alpha_i}(x_i^0-x_i^1)\Big)
\,\textup{d}x_1^0 \,\textup{d}x_1^1 \cdots \textup{d}x_n^0 \,\textup{d}x_n^1 \]
for $s,\alpha_1,\ldots,\alpha_n\in(0,\infty)$.
We want to prove the estimate
\begin{equation}\label{eq:lmerr3}
\Theta_{m,a,b}^{\alpha_1,\ldots,\alpha_n} \leq 2\pi
\end{equation}
for $m,a,b,\alpha_1,\ldots,\alpha_n$ as above.

First, by the product rule for differentiation and the heat equation \eqref{eq:heatequation} we can write
\begin{align*}
& \frac{\partial}{\partial s} \big( \mathbbm{g}_{s\alpha_1}(y_1) \mathbbm{g}_{s\alpha_2}(y_2) \mathbbm{g}_{s\alpha_3}(y_3) \cdots \mathbbm{g}_{s\alpha_n}(y_n) \big) \\
& = \frac{1}{2\pi s} \big( \mathbbm{k}_{s\alpha_1}(y_1) \mathbbm{g}_{s\alpha_2}(y_2) \mathbbm{g}_{s\alpha_3}(y_3) \cdots \mathbbm{g}_{s\alpha_n}(y_n) \\
& \qquad\ + \mathbbm{g}_{s\alpha_1}(y_1) \mathbbm{k}_{s\alpha_2}(y_2) \mathbbm{g}_{s\alpha_3}(y_3) \cdots \mathbbm{g}_{s\alpha_n}(y_n) + \cdots \\
& \qquad\ + \mathbbm{g}_{s\alpha_1}(y_1) \mathbbm{g}_{s\alpha_2}(y_2) \mathbbm{g}_{s\alpha_3}(y_3) \cdots \mathbbm{k}_{s\alpha_n}(y_n) \big)
\end{align*}
for $y_1,y_2,y_3,\ldots,y_n\in\mathbb{R}^2$. Substituting $y_i=x_i^0-x_i^1$, using the fundamental theorem of calculus in the variable $s$, multiplying by $\mathcal{F}$, and integrating in $x_1^0, x_1^1, \ldots, x_n^0, x_n^1$ we conclude
\begin{equation}\label{eq:lmerr2}
\sum_{m=1}^{n} \Theta_{m,a,b}^{\alpha_1,\ldots,\alpha_n} = 2\pi \big( \Xi_{a}^{\alpha_1,\ldots,\alpha_n} - \Xi_{b}^{\alpha_1,\ldots,\alpha_n} \big).
\end{equation}
Next, using the second equality from \eqref{eq:teleconv0} one can write
\begin{align*}
- \int_{(\mathbb{R}^2)^2} \mathcal{F} &
\ \mathbbm{k}_{s\alpha_1}(x_1^0-x_1^1) \,\textup{d}x_1^0 \,\textup{d}x_1^1 \\
& = 2 \sum_{l=1}^{2} \int_{(\mathbb{R}^2)^3} \mathcal{F}'(x_1^0) \mathcal{F}'(x_1^1)
\,\mathbbm{h}_{2^{-1/2}s\alpha_1}^{(l)}(x_1^0-q) \,\mathbbm{h}_{2^{-1/2}s\alpha_1}^{(l)}(x_1^1-q) \,\textup{d}x_1^0 \,\textup{d}x_1^1 \,\textup{d}q \\
& = 2 \sum_{l=1}^{2} \int_{\mathbb{R}^2} \Big( \int_{\mathbb{R}^2} \mathcal{F}'(x) \,\mathbbm{h}_{2^{-1/2}s\alpha_1}^{(l)}(x-q) \,\textup{d}x \Big)^2 \,\textup{d}q \geq 0,
\end{align*}
so $\Theta_{1,a,b}^{\alpha_1,\ldots,\alpha_n} \geq 0$ and, completely analogously, we also prove that
\begin{equation}\label{eq:lmerr5}
\Theta_{m,a,b}^{\alpha_1,\ldots,\alpha_n} \geq 0
\end{equation}
for each $m\in\{1,\ldots,n\}$.
Finally, making a crude estimate
\[ \mathcal{F} \leq \mathbbm{1}_{([0,1]^2)^n}(x_1^1, \ldots, x_n^1), \]
integrating in $x_1^0, \ldots, x_n^0\in\mathbb{R}^2$, and then integrating in $x_1^1, \ldots, x_n^1\in[0,1]^2$, we obtain
\begin{equation}\label{eq:lmerr6}
0 \leq \Xi_{s}^{\alpha_1,\ldots,\alpha_n} \leq 1
\end{equation}
for $s,\alpha_1,\ldots,\alpha_n\in(0,\infty)$.
From \eqref{eq:lmerr5} and \eqref{eq:lmerr6} we see that the left hand side of \eqref{eq:lmerr2} is a sum of $n$ nonnegative terms that add up to at most $2\pi$. Thus, each of these terms is bounded individually by $2\pi$, which finalizes the proof of \eqref{eq:lmerr3}.

Now we turn to estimation of the left hand side of \eqref{eq:proofoflm2}.
Using the product rule for differentiation and the heat equation \eqref{eq:heatequation}, similarly as before we get
\begin{align*}
& \frac{\partial}{\partial t} \Big( (\sigma_{\lambda}\ast\mathbbm{g}_{t\lambda})(y_1) (\sigma_{\lambda}\ast\mathbbm{g}_{t\lambda})(y_2) (\sigma_{\lambda}\ast\mathbbm{g}_{t\lambda})(y_3) \cdots (\sigma_{\lambda}\ast\mathbbm{g}_{t\lambda})(y_n) \Big) \\
& = \frac{1}{2\pi t} \Big(
(\sigma_{\lambda}\ast\mathbbm{k}_{t\lambda})(y_1) (\sigma_{\lambda}\ast\mathbbm{g}_{t\lambda})(y_2) (\sigma_{\lambda}\ast\mathbbm{g}_{t\lambda})(y_3) \cdots (\sigma_{\lambda}\ast\mathbbm{g}_{t\lambda})(y_n) \\
& \qquad\ + (\sigma_{\lambda}\ast\mathbbm{g}_{t\lambda})(y_1) (\sigma_{\lambda}\ast\mathbbm{k}_{t\lambda})(y_2) (\sigma_{\lambda}\ast\mathbbm{g}_{t\lambda})(y_3) \cdots (\sigma_{\lambda}\ast\mathbbm{g}_{t\lambda})(y_n) + \cdots \\
& \qquad\ + (\sigma_{\lambda}\ast\mathbbm{g}_{t\lambda})(y_1) (\sigma_{\lambda}\ast\mathbbm{g}_{t\lambda})(y_2) (\sigma_{\lambda}\ast\mathbbm{g}_{t\lambda})(y_3) \cdots (\sigma_{\lambda}\ast\mathbbm{k}_{t\lambda})(y_n) \Big)
\end{align*}
for $\lambda\in(0,\infty)$ and $y_1,\ldots,y_n\in\mathbb{R}^2$.
By the fundamental theorem of calculus we now see that $\mathcal{N}^{\varepsilon}_{\lambda_j}(A)-\mathcal{N}^{1}_{\lambda_j}(A)$ is the sum of
\begin{align}
-\frac{1}{2\pi} \int_{\varepsilon}^{1} \int_{(\mathbb{R}^2)^{2n}} \mathcal{F} \
(\sigma_{\lambda_j}\ast\mathbbm{k}_{t\lambda_j})(x_1^0-x_1^1) (\sigma_{\lambda_j}\ast\mathbbm{g}_{t\lambda_j})(x_2^0-x_2^1) \cdots (\sigma_{\lambda_j}\ast\mathbbm{g}_{t\lambda_j})(x_n^0-x_n^1) & \nonumber \\
\,\textup{d}x_1^0 \,\textup{d}x_1^1 \,\textup{d}x_2^0 \,\textup{d}x_2^1 \cdots \textup{d}x_n^0 \,\textup{d}x_n^1 \,\frac{\textup{d}t}{t} & \label{eq:lmerr1}
\end{align}
and $n-1$ analogous terms.
For a positive integer $j$, number $t\in(0,\infty)$, and $s\in[2^{-j-5}t,2^{-j-4}t]$ this time we denote
\[ r_j(s,t) := \sqrt{t^2\lambda_j^2-2s^2},\quad c_j(s,t) := \frac{t^2\lambda_j^2}{s^2} \]
and observe
\begin{equation}\label{eq:lmerr7}
s\sim 2^{-j}t \sim t \lambda_j, \quad r_j(s,t)\sim t \lambda_j, \quad c_j(s,t)\sim 1.
\end{equation}
Convolution identities \eqref{eq:teleconv0} and \eqref{eq:teleconv} imply
\[ \sigma_{\lambda_j}\ast\mathbbm{k}_{t\lambda_j}
= c_j(s,t) \sum_{l=1}^2 \sigma_{\lambda_j} \ast \mathbbm{g}_{r_j(s,t)} \ast \mathbbm{h}_{s}^{(l)} \ast \mathbbm{h}_{s}^{(l)}. \]
We multiply the integrand in \eqref{eq:lmerr1} by \eqref{eq:lpeq3} and substitute the equality from the last display.
Taking \eqref{eq:lmerr7} into account we conclude that the left hand side of \eqref{eq:proofoflm2} is at most a constant times
\begin{align}
\sum_{j=1}^{J} \sum_{l=1}^2 \int_{\varepsilon}^{1} \int_{2^{-j-5}t}^{2^{-j-4}t} \int_{(\mathbb{R}^2)^{2n}}
& \Big| \int_{\mathbb{R}^2} \mathcal{F}'(x_1^0) \,\mathbbm{h}_{s}^{(l)}(x_1^0-q^0) \,\textup{d}x_1^0 \Big|
\Big| \int_{\mathbb{R}^2} \mathcal{F}'(x_1^1) \,\mathbbm{h}_{s}^{(l)}(x_1^1-q^1) \,\textup{d}x_1^1 \Big| \nonumber \\
& \times (\sigma_{\lambda_j} \ast \mathbbm{g}_{r_j(s,t)})(q^0-q^1) \,\Big( \prod_{i=2}^{n} (\sigma_{\lambda_j} \ast \mathbbm{g}_{t\lambda_j})(x_i^0-x_i^1) \Big) \nonumber \\
& \,\textup{d}q^0 \,\textup{d}q^1 \,\textup{d}x_2^0 \,\textup{d}x_2^1 \cdots \textup{d}x_n^0 \,\textup{d}x_n^1 \,\frac{\textup{d}s}{s} \,\frac{\textup{d}t}{t} \label{eq:lmerr8}
\end{align}
Since Gaussian tails decay faster than any polynomial, we trivially have
\[ (\sigma_{t^{-1}}\ast\mathbbm{g})(x) \lesssim \int_{\mathbb{R}^2} \Big(1+\Big\|x-\frac{y}{t}\Big\|_{\ell^2}\Big)^{-5} \,\textup{d}\sigma(y)
\lesssim \varepsilon^{-5} (1+\|x\|_{\ell^2})^{-5} \]
for $t\in[\varepsilon,1]$ and $x\in\mathbb{R}^2$. In the same way, by also using \eqref{eq:lmerr7}, we obtain
\[ (\sigma_{t^{-1}}\ast\mathbbm{g}_{r_j(s,t)t^{-1}\lambda_j^{-1}})(x) \lesssim \varepsilon^{-5} (1+\|x\|_{\ell^2})^{-5} \]
for $j,t,s$ as above.
Rescaling by $t\lambda_j s^{-1}$ taking \eqref{eq:lmerr7} into account again, we end up with
\[ (\sigma_{\lambda_j s^{-1}}\ast\mathbbm{g}_{t\lambda_j s^{-1}})(x) \lesssim \varepsilon^{-5} (1+\|x\|_{\ell^2})^{-5},\quad
(\sigma_{\lambda_j s^{-1}}\ast\mathbbm{g}_{r_j(s,t) s^{-1}})(x) \lesssim \varepsilon^{-5} (1+\|x\|_{\ell^2})^{-5}. \]
Now we use Estimate \eqref{eq:Gaussiandom} in dimension $d=2$, which dominates the Schwartz tails by a superpositions of dilated Gaussians.
Yet another rescaling, this time by $s$, yields
\[ (\sigma_{\lambda_j}\ast\mathbbm{g}_{t\lambda_j})(x) \lesssim \varepsilon^{-5}\int_1^\infty \mathbbm{g}_{\beta s}(x) \,\frac{\textup{d}\beta}{\beta^4},
\quad (\sigma_{\lambda_j}\ast\mathbbm{g}_{r_j(s,t)})(x) \lesssim \varepsilon^{-5}\int_1^\infty \mathbbm{g}_{\beta s}(x) \,\frac{\textup{d}\beta}{\beta^4}. \]
Using these estimates we can dominate \eqref{eq:lmerr8} by a constant multiple of
\begin{align*}
\varepsilon^{-5n} \sum_{j=1}^{J} \sum_{l=1}^2 & \int_{[1,\infty)^n} \int_{\varepsilon}^{1} \int_{2^{-j-5}t}^{2^{-j-4}t} \int_{(\mathbb{R}^2)^{2n}}
\Big| \int_{\mathbb{R}^2} \mathcal{F}'(x_1^0) \,\mathbbm{h}_{s}^{(l)}(x_1^0-q^0) \,\textup{d}x_1^0 \Big| \\
& \times \Big| \int_{\mathbb{R}^2} \mathcal{F}'(x_1^1) \,\mathbbm{h}_{s}^{(l)}(x_1^1-q^1) \,\textup{d}x_1^1 \Big|
\,\mathbbm{g}_{\beta_1 s}(q^0-q^1) \,\Big( \prod_{i=2}^{n} \mathbbm{g}_{\beta_i s}(x_i^0-x_i^1) \Big) \\
& \,\textup{d}q^0 \,\textup{d}q^1 \,\textup{d}x_2^0 \,\textup{d}x_2^1 \cdots \textup{d}x_n^0 \,\textup{d}x_n^1 \,\frac{\textup{d}s}{s} \,\frac{\textup{d}t}{t} \,\frac{\textup{d}\beta_1}{\beta_1^4} \cdots \frac{\textup{d}\beta_n}{\beta_n^4}.
\end{align*}
Using the Cauchy--Schwarz inequality in $q^0, q^1, x_2^0, x_2^1, \ldots, x_n^0, x_n^1, s, t, \beta_1, \ldots, \beta_n, l, j$ and observing mutually symmetric roles of $x_1^0$ and $x_1^1$, we bound this by
\begin{align*}
\mathcal{I} := \varepsilon^{-5n} \sum_{j=1}^{J} \sum_{l=1}^2 & \int_{[1,\infty)^n} \int_{\varepsilon}^{1} \int_{2^{-j-5}t}^{2^{-j-4}t} \int_{(\mathbb{R}^2)^{2n}}
\Big( \int_{\mathbb{R}^2} \mathcal{F}'(x_1^0) \,\mathbbm{h}_{s}^{(l)}(x_1^0-q^0) \,\textup{d}x_1^0 \Big)^2 \mathbbm{g}_{\beta_1 s}(q^0-q^1) \\
& \times \Big( \prod_{i=2}^{n} \mathbbm{g}_{\beta_i s}(x_i^0-x_i^1) \Big)
\,\textup{d}q^0 \,\textup{d}q^1 \,\textup{d}x_2^0 \,\textup{d}x_2^1 \cdots \textup{d}x_n^0 \,\textup{d}x_n^1 \,\frac{\textup{d}s}{s} \,\frac{\textup{d}t}{t} \,\frac{\textup{d}\beta_1}{\beta_1^4} \cdots \frac{\textup{d}\beta_n}{\beta_n^4}.
\end{align*}
Note that in the expression defining $\mathcal{I}$ we can easily integrate in the variable $q^1$.
Then we expand out the square of the integral in $x_1^0$ denoting the second copy of that variable conveniently by $x_1^1$ again.
If we also use the convolution identity
\[ \sum_{l=1}^2 \int_{\mathbb{R}^2} \mathbbm{h}_{s}^{(l)}(x_1^0-q^0) \,\mathbbm{h}_{s}^{(l)}(x_1^1-q^0) \,\textup{d}q^0
= - \frac{1}{2} \mathbbm{k}_{2^{1/2}s}(x_1^0-x_1^1), \]
which follows from \eqref{eq:teleconv0}, and sum in $j$, then we can recognize $\mathcal{I}$ as
\[ \mathcal{I}
= \frac{1}{2} \varepsilon^{-5n} \int_{[1,\infty)^n} \int_{\varepsilon}^{1} \Theta_{1,2^{-J-5}t,2^{-5}t}^{2^{1/2},\beta_2,\ldots,\beta_n} \,\frac{\textup{d}t}{t} \,\frac{\textup{d}\beta_1}{\beta_1^4} \cdots \frac{\textup{d}\beta_n}{\beta_n^4}. \]
Estimate \eqref{eq:proofoflm2} follows simply by using \eqref{eq:lmerr3} and integrating in all of the remaining variables.

\subsection{The uniform part: proof of \eqref{eq:proofoflm3}}
Take $0<\vartheta<\varepsilon$. Exactly as in the previous subsection, we see that $\mathcal{N}^{\vartheta}_{\lambda}(A)-\mathcal{N}^{\varepsilon}_{\lambda}(A)$ is the sum of
\begin{align}
-\frac{1}{2\pi} \int_{\vartheta}^{\varepsilon} \int_{(\mathbb{R}^2)^{2n}} \mathcal{F} \
(\sigma_{\lambda}\ast\mathbbm{k}_{t\lambda})(x_1^0-x_1^1) (\sigma_{\lambda}\ast\mathbbm{g}_{t\lambda})(x_2^0-x_2^1) \cdots (\sigma_{\lambda}\ast\mathbbm{g}_{t\lambda})(x_n^0-x_n^1) & \nonumber \\
\,\textup{d}x_1^0 \,\textup{d}x_1^1 \,\textup{d}x_2^0 \,\textup{d}x_2^1 \cdots \textup{d}x_n^0 \,\textup{d}x_n^1 \,\frac{\textup{d}t}{t} & \label{eq:lmunif1}
\end{align}
and $n-1$ analogous terms. We begin by working with a fixed $t\in[\vartheta,\varepsilon]$ and split $\mathcal{F}$ as $\mathcal{F}'(x_1^0)\mathcal{F}'(x_1^1)$.
For fixed $x_2^0,x_2^1,\ldots,x_n^0,x_n^1\in\mathbb{R}^2$, Plancherel's theorem (i.e., unitarity of the Fourier transform on $\textup{L}^2(\mathbb{R}^d)$) gives
\begin{align}
& \int_{(\mathbb{R}^2)^{2}} \mathcal{F} \, (\sigma_{\lambda}\ast\mathbbm{k}_{t\lambda})(x_1^0-x_1^1) \,\textup{d}x_1^0 \,\textup{d}x_1^1
= \int_{\mathbb{R}^2} (\mathcal{F}'\ast\sigma_{\lambda}\ast\mathbbm{k}_{t\lambda})(x_1^0) \,\mathcal{F}'(x_1^0) \,\textup{d}x_1^0 \nonumber \\
& = \int_{\mathbb{R}^2} (\widehat{\mathcal{F}'\ast\sigma_{\lambda}\ast\mathbbm{k}_{t\lambda}})(\xi) \,\overline{\widehat{\mathcal{F}'}(\xi)} \,\textup{d}\xi
= \int_{\mathbb{R}^2} \big|\widehat{\mathcal{F}'}(\xi)\big|^2 \,\widehat{\sigma}(\lambda\xi) \,\widehat{\mathbbm{k}}(t\lambda\xi) \,\textup{d}\xi. \label{eq:lmunif2}
\end{align}
If $\|\xi\|_{\ell^2}\leq \lambda^{-1}$, then Estimate \eqref{eq:FTofsigma} and Equality \eqref{eq:ftofk} give
\[ \big| \widehat{\sigma}(\lambda\xi) \,\widehat{\mathbbm{k}}(t\lambda\xi) \big|
\leq t^2\lambda^2\|\xi\|_{\ell^2}^2 e^{-\pi t^2\lambda^2\|\xi\|_{\ell^2}^2} \leq t^2 \leq t^{1/2}. \]
On the other hand, if $\|\xi\|_{\ell^2}>\lambda^{-1}$, then, by Estimate \eqref{eq:FTofsigma}, Equality \eqref{eq:ftofk}, and an easy observation
\[ \sup_{s\in[0,\infty)} s^{3/2} e^{-\pi s^2} < \infty, \]
we have
\[ \big| \widehat{\sigma}(\lambda\xi) \,\widehat{\mathbbm{k}}(t\lambda\xi) \big|
\lesssim t^2 \lambda^{3/2} \|\xi\|_{\ell^2}^{3/2} e^{-\pi t^2\lambda^2\|\xi\|_{\ell^2}^2} \lesssim t^{1/2}. \]
Consequently, \eqref{eq:lmunif2} is bounded in the absolute value by a constant times
\[ t^{1/2} \int_{\mathbb{R}^2} \big|\widehat{\mathcal{F}'}(\xi)\big|^2 \,\textup{d}\xi
= t^{1/2} \big\|\widehat{\mathcal{F}'}\big\|_{\textup{L}^2(\mathbb{R}^2)}^2
= t^{1/2} \|\mathcal{F}'\|_{\textup{L}^2(\mathbb{R}^2)}^2
\leq t^{1/2} \|\mathbbm{1}_{[0,1]^2}\|_{\textup{L}^2(\mathbb{R}^2)}^2 = t^{1/2}. \]
Multiplying \eqref{eq:lmunif2} by $n-1$ convolutions of the form $\sigma_{\lambda}\ast\mathbbm{g}_{t\lambda}$ and integrating, first in $x_2^0,\ldots,x_n^0\in\mathbb{R}^2$ and then in $x_2^1,\ldots,x_n^1\in[0,1]^2$, we estimate the absolute value of \eqref{eq:lmunif1} by a constant times
\[ \int_{\vartheta}^{\varepsilon} t^{1/2} \,\frac{\textup{d}t}{t} \lesssim \varepsilon^{1/2}. \]
Thus,
\[ \big| \mathcal{N}^{\vartheta}_{\lambda}(A)-\mathcal{N}^{\varepsilon}_{\lambda}(A) \big| \lesssim_n \varepsilon^{1/2}. \]
Note that this bound is uniform in $\vartheta\in(0,\varepsilon)$, so letting $\vartheta\to0^+$ we complete the proof of \eqref{eq:proofoflm3}.

\begin{remark}\label{rem:reprovinglm}
It is worth noting that the above proof gives slightly more than stated in Theorem~\ref{thm:quantifiedlm}, since Estimate \eqref{eq:proofoflm2} is uniform in the number of scales $J$.
As a consequence, we also reprove \cite[Theorem~1.1~(i)]{LM19:hypergraphs} by Lyall and Magyar:
for every measurable set $A\subseteq(\mathbb{R}^2)^n$ with $\overline{\delta}(A)>0$ there exists $\lambda_0=\lambda_0(n,A)\in(0,\infty)$ such that for every $\lambda\in[\lambda_0,\infty)$ one can find $x_1,\ldots,x_n, y_1,\ldots,y_n \in\mathbb{R}^2$ satisfying \eqref{eq:thquantlm1} and \eqref{eq:thquantlm2}.
The reduction of this statement to \eqref{eq:proofoflm1}--\eqref{eq:proofoflm3} is easy and standard; the details can be found for instance in \cite{CMP15:roth} or \cite{DKR17:corner}.
\end{remark}

\section*{Acknowledgments}
V. K. is supported in part by the \emph{Croatian Science Foundation} under project UIP-2017-05-4129 (MUNHANAP), and in part by the \emph{Fulbright Scholar Program}. He also appreciates hospitality of the \emph{Georgia Institute of Technology} in the academic year 2019--2020 and that of the \emph{California Institute of Technology} while this research was performed.

%%%%%%%%%%%%%%%%%%%%%%%%%%%%%%%%%%%%%%%%%%%%%%%%

\end{document}